\documentclass[12pt]{article}
\usepackage{stmaryrd}
\usepackage{amsmath}
\usepackage{amsfonts}
\usepackage{amssymb}
\usepackage{amsthm}
\usepackage[all,2cell]{xy}
\usepackage{color}
\usepackage{etoolbox}
\usepackage{enumitem}

\usepackage[T2A,T1]{fontenc}
\usepackage[utf8]{inputenc}
\usepackage[russian,USenglish]{babel}

\usepackage[margin=1in]{geometry}

\UseTwocells

\newbool{hidedetails}
\booltrue{hidedetails}

\newtheorem{lemma}{Lemma}
\newtheorem{proposition}{Proposition}

\newtheorem{theorem}{Theorem}

\theoremstyle{definition}
\newtheorem{definition}{Definition}
\newtheorem{example}{Example}
\newtheorem{notation}{Notation}
\newtheorem{remark}{Remark}

\newcommand{\noneg}{{\mathbb Z_{\geq 0}}}								% the set of non-negative integers

\newcommand{\diagra}{{\mathcal D}}										% a diagram category
\newcommand{\colim}{{\rm colim}}										% the colim functor in a category
\renewcommand{\hom}[1]{{{\rm Hom}_{#1}}}								% hom set in a given category
\newcommand{\id}{{\rm Id}}											% the identity morphism
\newcommand{\op}{{\rm op}}											% sign of the opposite category
\newcommand{\sets}{{\rm Set}}											% the category of sets
\newcommand{\set}{{S}}												% a set
\newcommand{\element}{{s}}											% an element of a set
\newcommand{\fsets}{{{\rm FSet}}}										% the category of finite sets and surjections

\newcommand{\finor}[1]{{\underline{#1}}}									% a finite given ordinal
\newcommand{\fima}{{\mu}}											% an morphism in the category of finite ordinals
\newcommand{\ssets}{{\rm SSet}}										% the category of simplicial sets
\newcommand{\simplex}[1]{{\Delta_{#1}}}									% a standard simplex as a topological space
\newcommand{\horn}[2]{{{\Lambda}_{{#1},{#2}}}}							% a given horn

\newcommand{\fgrings}{{\rm C^\infty R}}									% category of finitely generated $C^\infty$-rings
\newcommand{\cring}{{A}}												% a C^\infty-ring
\newcommand{\celement}{{a}}											% an element of a C^\infty-ring
\newcommand{\preimage}[1]{{\widetilde{#1}}}								% a pre-image of an element with respect to a surjective morphism
\newcommand{\ima}[1]{{\underline{#1}}}									% the image of an element with respect to a surjective morphism
\newcommand{\charfun}{{f}}											% a characteristic function for an open subset
\newcommand{\ideal}{{\mathfrak A}}										% an ideal
\newcommand{\kerim}[1]{{\mathfrak I_{#1}}}								% kernel of a morphism of schemes
\newcommand{\gerim}[1]{{\mathfrak m^{\rm g}_{#1}}}						% ideal of zero-germs for a morphism of schemes
\newcommand{\igen}[1]{{(#1)}}											% the ideal generated by something
\newcommand{\kernel}[1]{{{\rm Ker}(#1)}}									% the kernel of a morphism
\newcommand{\nradical}[1]{{\sqrt[{\rm nil}]{#1}}}							% nilradical
\newcommand{\iradical}[1]{{\sqrt[\infty]{#1}}}								% infinity radical
\newcommand{\zerogerm}[1]{{\mathfrak m^{\rm g}_{#1}}}						% ideal of zero germs at something
\newcommand{\vset}[3]{{\{#1\in #2\,|\,#3\}}}								% the set of points somewhere satisfying some condition
\newcommand{\cinfty}{{\rm C^\infty}}										% the $C^\infty$-symbol
\newcommand{\nil}{{\rm nil}}											% the symbol for nilpotent
\newcommand{\redun}{{{\rm R}_{\nil}}}									% nilpotent reduction
\newcommand{\redui}{{{\rm R}_\infty}}									% \infty-reduction
\newcommand{\reduiof}[1]{{(#1)_{\rm red}}}								% \infty reduction of something

\newcommand{\csimplex}[1]{{\mathbf\Delta_{#1}}}							% a standard simplex as $C^\infty$-scheme
\newcommand{\cske}[2]{{{\rm Sk}_{#1}({#2})}}								% a skeleton of something
\newcommand{\chorn}[2]{{{\mathbf \Lambda}_{{#1},{#2}}}}						% a given horn
\newcommand{\flatt}[2]{{\Xi_{{#1},{#2}}}}									% flattening a simplex on its face
\newcommand{\colla}{{\kappa}}											% collapse on a face
\newcommand{\proco}[2]{{\pi_{{#1},{#2}}}}									% projections given by collarings
\newcommand{\ssim}{{\gamma}}										% a family of smooth simplices
\newcommand{\smon}[2]{{{ Hom}(\csimplex{#1},{#2})}}						% the presheaf of smooth simplices
\newcommand{\bouns}[1]{{\partial\,{#1}}}									% boundary of something
\newcommand{\sphere}[1]{{{\rm S}^{#1}}}									% sphere of the given dimension
\newcommand{\resti}[3]{{{#1}\times\csimplex{{#3}}}}							% restriction to a part of the boundary

\newcommand{\pt}{{\rm pt}}											% the real point
\newcommand{\rpt}{{\rm p}}											% a real point somewhere
\newcommand{\rqt}{{\rm q}}											% another real point
\newcommand{\spec}{{\rm Spec}}										% the full spectrum = going to the opposite category
\newcommand{\specof}[1]{{\spec(#1)}}									% the full spectrum of something
\newcommand{\germof}[2]{{#1_{#2}}}									% germ of something at something
\newcommand{\puncof}[2]{{\overset{\circ}{#2}}}									% germ of a puncture
\newcommand{\gima}[1]{{\widetilde{#1}}}									% germ of the image of a morphism
\newcommand{\softchart}{{\mathcal W}}									% a soft chart on a $C^\infty$-scheme
\newcommand{\manifold}{{\mathcal N}}									% a chosen manifold
\newcommand{\vmani}{{{\rm S}}}										% a varying manifold
\newcommand{\smoup}{{\mathcal G}}									% a smooth group
\newcommand{\Ran}[1]{{{\rm Ran}_{#1}}}									% Ran space of something
\newcommand{\produ}[2]{{{#1}^{#2}}}									% parameterized products
\newcommand{\diadia}[1]{{{#1}_\Delta}}									% the diagram of diagonals
\newcommand{\gra}[1]{{\Gamma_{#1}}}									% graph of a given map
\newcommand{\gegra}[1]{{\dot\Gamma_{#1}}}								% germ of graph
\newcommand{\pugra}[1]{{\overset{\circ}{\Gamma}_{#1}}}						% punctured germ of a graph
\newcommand{\pgrass}[3]{{\overset{\circ}{\rm Gr}_{#3}({#1},{#2})}}				% punctured grassmanian
\newcommand{\ggrass}[3]{{\dot{\rm Gr}_{#3}({#1},{#2})}}						% germ grassmanian
\newcommand{\grass}[2]{{{\rm Gr}_{#2}({#1})}}								% affine grassmannian
\newcommand{\abgro}{{\mathcal A}}										% an abelian group object
\newcommand{\classin}[1]{{\mathbb B^{#1}\abgro}}							% the classifying space of the multiplicative group of a given iteration
\newcommand{\barco}[2]{{\mathcal B^{#1}{#2}}}							% the iterated bar construction of a group
\newcommand{\classi}[1]{{{\mathbb B}{#1}}}								% the classifying space of some group
\newcommand{\gerco}{{\alpha^{\nabla}}}									% a gerbe with connection
\newcommand{\gercos}[1]{{\alpha^\nabla_{#1}}}							% a lift of gerbe with connection to a relative site
\newcommand{\eva}{{\rm ev}}											% the evaluation map
\newcommand{\csite}{{\cinfty{\rm Sch}}}									% the site of $C^\infty$-schemes
\newcommand{\msite}{{\mathfrak{M}}}									% the site of classical manifolds
\newcommand{\bsite}{{\mathfrak {M}^{c}}}									% the site of manifolds with corners
\newcommand{\gsite}{{\mathfrak{AM}}}									% the site of asymptotic manifolds with corners
\newcommand{\ksite}{{\gsite_{{\rm l.c.}}}}									% the site of locally compact asymptotic manifolds
\newcommand{\sheaf}{{\mathcal F}}										% a sheaf
\newcommand{\zaritop}{{\mathcal Z}}									% the topos of pre-sheaves on the Zariski site
\newcommand{\zaritos}[1]{{{\mathcal Z}_{#1}}}								% Zariski topos on a given site
\newcommand{\zaritoss}[1]{{\underline{\mathcal Z}_{#1}}}						% simplicial Zariski topos on a given site
\newcommand{\infired}[1]{{#1_{\rm dR_\infty}}}								% the $\infty$-reduction on pre-sheaves
\newcommand{\algered}[1]{{#1_{\rm dR}}}									% the algebraic reduction of pre-sheaves
\newcommand{\cscheme}{{\mathcal X}}									% a C^\infty-scheme
\newcommand{\opca}[1]{{\mathfrak{O}_{#1}}}								% the category of open subschemes and inclusions

\newcommand{\emptys}{{\varnothing}}									% the empty C^\infty-scheme
\newcommand{\cmor}{{\Phi}}											% a C^\infty-morphism
\newcommand{\cmorf}{{\phi}}											% a C^\infty-morphism in functions
\newcommand{\germ}[2]{{(#1)_{#2}}}										% germ of something at something

\newcommand{\creali}{{\rho}}											% realization of finite simplicial sets as $C^\infty$-schemes
\newcommand{\crealio}[1]{{\rho(#1)}}										% realization of a given finite simplicial set
\newcommand{\homs}{{\underline{\rm Hom}}}								% simplicial hom
\newcommand{\homo}[1]{{\mathbf{Hom}_{#1}}}								% internal hom-object among pre-sheaves	
\newcommand{\Homo}[1]{{\underline{\mathbf{Hom}}_{#1}}}					% internal mapping object functor among pre-sheaves
\newcommand{\paigro}[1]{{{\rm P}{#1}}}									% the pair groupoid of a pre-sheaf	
\newcommand{\fungro}[1]{{\Pi(#1)}}										% the fundamental groupoid pre-sheaf
\newcommand{\trigro}[1]{{\underline{#1}}}									% the trivial groupoid pre-sheaf
\newcommand{\thingro}[2]{{\Theta^{#1}_{#2}}}								% the thin groupoid of a given thinness
\newcommand{\thinfo}[2]{{\mathfrak X^{#1}_{#2}}}							% the formal k-thin groupoid
\newcommand{\thinif}[2]{{\mathfrak Y^{#1}_{#2}}}							% the infinitesimal k-thin groupoid
\newcommand{\germgro}[1]{{{\mathbf \Delta}_{#1}}}							% the groupoid of the germ of the diagonal
\newcommand{\fordia}[1]{{\widehat{\Delta}_{#1}}}							% the groupoid of formal neighbourhoods of the diagonals
\newcommand{\infidia}[1]{{\widetilde{\Delta}_{#1}}}							% the groupoid of \infty-neighbourhoods of the diagonals 

\newcommand{\openset}{{U}}											% an open subset
\newcommand{\opensch}{{\mathcal U}}									% an open subscheme
\newcommand{\fullex}[1]{{\widetilde{#1}}}									% the full extension of a open subscheme in a subscheme
\newcommand{\pullba}[2]{{{#1}^{-1}(#2)}}									% pullback of something
\newcommand{\closim}[2]{{\overline{{#1}({#2})}}}							% closure of the image of a morphism of schemes
\newcommand{\clos}[1]{{\overline{#1}}}									% closure of a subscheme

\newcommand{\comple}[2]{{#1-#2}}										% complement of an open scheme in a scheme
\newcommand{\prinop}[2]{{{#1}_{#2}}}									% the  principal open subscheme corresponding to a function
\newcommand{\locin}[2]{{{#1}_{#2}}}										% the locally closed subscheme defined by an (in)equality
\newcommand{\closedset}{{V}}											% a closed subset
\newcommand{\manifo}{{\mathcal M}}									% a manifold
\newcommand{\regva}[1]{{{\rm Reg}(#1)}}									% the set of regular values of a function

\newcommand{\blue}{\color{blue}}

\newcommand{\hide}[1]{\ifbool{hidedetails}{}{{\blue #1}}}

\begin{document}

\title{Beyond perturbation 2: asymptotics and Beilinson--Drinfeld Grassmannians in differential geometry}
\author{\small\parbox{.5\linewidth}{Dennis Borisov\\ University of Windsor, Canada\\ 401 Sunset Ave, Windsor\\ dennis.borisov@uwindsor.ca}
\parbox{.5\linewidth}{Kobi Kremnizer\\ University of Oxford, UK\\ Woodstock Rd, Oxford OX2 6GG\\ yakov.kremnitzer@maths.ox.ac.uk}}
\date{\today}
\maketitle

\begin{abstract} We prove that $\forall k\geq 2$ given a smooth compact $k$-dimensional manifold $\manifold$ and a multiplicative $k-1$-gerbe on a Lie group $G$ together with an integrable connection, there is a line bundle on the Beilinson--Drinfeld Grassmannian $\grass{\manifold}{G}$ having the factorization property. We show that taking global sections of this line bundle we obtain a factorization algebra on $\manifold$.
\medskip



\end{abstract}

\section*{Introduction}

There is a geometric description of vertex operator algebras (\cite{FBZ04} \S19-20): start with a complex curve $X$, consider the moduli space $\Ran{X}$ of unordered finite sets of points in $X$, and try to build an interesting sheaf of modules on $\Ran{X}$, that has the factorization property. This property reflects the fact that $\Ran{X}$ has the structure of a monoid, given by union of punctures.

This is of course the theory of chiral algebras or factorization algebras developed in \cite{BD04}, and the important class of examples of them comes from a particular {\it factorization space} (e.g.\@ \cite{FBZ04} \S20), i.e.\@ a geometric object 
	\begin{equation*}Y\longrightarrow\Ran{X}\end{equation*} 
satisfying the factorization condition. This geometric object is yet another moduli space: the Beilinson--Drinfeld Grassmannian $\grass{X}{G}$, where $G$ is some group of our choice. It is defined as the moduli space of principal $G$-bundles on $X$ with chosen trivializations outside finite sets of points.

It is easy to describe $\grass{X}{G}$ in terms of stacks (e.g.\@ \cite{Ga13} \S0.5.3), and it is rather clear that there is a canonical $\grass{X}{G}\rightarrow\Ran{X}$, realizing $\grass{X}{G}$ as a factorization space (e.g.\@ \cite{FBZ04} \S20.3). Then one usually considers a line bundle on $\grass{X}{G}$ and, pushing it forward to a sheaf on $\Ran{X}$, obtains a factorization algebra, or equivalently a chiral algebra.

\smallskip

In this paper we apply these techniques \emph{in differential geometry}. Specifically we show how to build principal $\abgro$-bundles on $\grass{X}{G}$ (considered as a stack on the usual site of $\cinfty$-manifolds) starting with any differentiable Deligne cohomology class of a Lie group $G$ with values in an abelian Lie group $\abgro$. The dimension of $X$ has to be 1 less than the dimension of the cohomology class. \emph{Other than that there is no restriction on dimensions.}

This is very different from algebraic geometry, where $X$ has to be a curve. The reason we can do this is that there is no Hartogs' extension theorem in differential geometry, and the moduli space $\grass{X}{G}$ can be non-trivial for $X$ of any dimension. The price that we pay for working with $\cinfty$-functions is that the polar behavior, which in the theory of vertex algebras is described by Laurent series, is considerably more complicated in differential geometry.

\subsubsection*{Asymptotic manifolds}

To deal with poles of $\cinfty$-functions we are forced to develop a geometric description of asymptotic behavior. By geometric description we mean using algebraic geometry of $\cinfty$-rings (e.g.\@ \cite{MR91}). When looking at $\cinfty(\mathbb R)$ from the point of view of an algebraic geometer, one immediately notices that the corresponding geometric object is much more than just $\mathbb R$. 

For example the ideal of compactly supported functions has to be contained in some maximal ideals, which cannot correspond to points of $\mathbb R$. These ``asymptotic points'' correspond to maximal filters of closed subsets of $\mathbb R$ (e.g.\@ \cite{MR86}). The set of all such filters parameterizes the ways to approach the two infinities in $\mathbb R$. The arithmetics of this non-Archimedean geometry of $\cinfty(\mathbb R)$ is rather complicated (e.g.\@ \cite{ADH17} and references therein).

\smallskip

Instead of doing algebraic geometry with all these complicated points, we choose to group them together into what we call {\it asymptotic manifolds}. Instead of maximal filters of closed subsets we take intersections of directed systems of open subsets. In terms of $\mathbb R$-points such intersections might be empty, but algebraically they correspond to non-trivial $\cinfty$-rings obtained as colimits of $\cinfty$-rings of functions on the open subsets.

A directed system can be finite, in which case the corresponding asymptotic manifold is just a usual manifold. Thus we obtain an enlargement of the category of manifolds to a bigger full subcategory of $\fgrings^{\op}$, where $\fgrings$ is the category of finitely generated $\cinfty$-rings.

\smallskip

We organize asymptotic manifolds into a site, using the Zariski topology, since we need $\cinfty$-rings representing germs of punctures. Our definitions allow us to develop the usual machinery of open, closed submanifolds, intersections, unions, complements, dimension theory, density structure etc.

In fact we manage to develop enough of the usual geometric techniques to be able to define and integrate connections around germs of punctures, and as a result to produce line bundles on Beilinson--Drinfeld Grassmannians by transgression, starting from differentiable Deligne cohomology classes on the group (e.g.\@ \cite{Br00}, \cite{Gaj}). Here germs of punctures play the same role as spheres do in the usual transgression constructions (e.g.\@ \cite{Br08}).

\smallskip

Asymptotic manifolds can be very different from the usual manifolds, and not only because they might have no $\mathbb R$-points at all. One of the important differences is in the notion of compactness. If one tries to capture the asymptotic behavior around a point, or any compact puncture, it is enough to consider \emph{sequences} of open neighbourhoods of the puncture. 

If the puncture is not compact, e.g.\@ a line in a plane, the germ cannot be computed by a sequence of open neighbourhoods, but one needs an uncountable directed system. This is a consequence of the very well known fact (e.g.\@ \cite{Ha1910}) that any countable set or orders of growth can be dominated by one order. Because of this asymptotic manifolds are not necessarily first-countable, a property we find very useful in proving Brown--Gersten descent theorems for thin homotopy groupoids.

\subsubsection*{Thin homotopy and connections}

The strategy of building a line bundle on the affine Grassmannian, starting with a gerbe on the group, is of course by using transgression. This means that we start not just with a gerbe, but also with a connection on it. As we would like to integrate connections around germs of punctures, it is not enough, in general, to work with the usual infinitesimal theory based on nilpotent elements. We use the much more powerful infinitesimal theory in the geometry of $\cinfty$-rings developed in \cite{BK}. It is based on $\infty$-nilpotents.

The difference between nilpotents and $\infty$-nilpotents, is that we obtain $0$ by evaluating a monomial of sufficiently high degree on the former, while for the latter, in order to obtain $0$, we might need a $\cinfty$-function that decays faster than any monomial. In particular $\infty$-infinitesimals are much more than first order infinitesimals, and this complicates the theory of connections. 

\smallskip

Whichever infinitesimals one chooses to use, defining a flat infinitesimal connection on some geometric object (e.g.\@ a morphism from $X$ into some classifying space $Y$) is equivalent to postulating insensitivity to dividing out infinitesimals. For example a flat connection on a morphism $X\rightarrow Y$ consists of a factorization through the de Rham space of $X$ (e.g.\@ \cite{GR14}). 

If one wants to allow connections that are not necessarily flat, one should not use the de Rham space but, for example, the free groupoid generated by $X$ over the de Rham space. However, then even connections along curves within $X$ will not be flat. Something similar happens when one wants to define higher order connections: one needs to postulate one dimensional flatness separately (e.g.\@ \cite{E09}). In terms of polynomial infinitesimals this construction of a free groupoid and subsequent flattening of curves was performed in \cite{Kap}.

\smallskip

In the case of $k-1$-gerbes there is no reason to stop with flatness along curves, it makes sense to require flatness all the way to dimension $k$. This means that we need to postulate that the morphism into the classifying space of the gerbe factors not through the de Rham space of $X$, but de Rham space of \emph{the $k$-thin homotopy groupoid of $X$}. 

The $k$-thin homotopy groupoid of $X$ consists of maps $\csimplex{n}\rightarrow X$ that locally (on $\csimplex{n}$) factor through something of dimension $\leq k$. The notion of thin homotopy is well known (e.g.\@ \cite{Ba91}, \cite{CP94}, \cite{BS05}, \cite{SW07}). In our setting we would like to compute the thin homotopy groupoids explicitly, i.e.\@ to find their fibrant representatives in the category of pre-sheaves of simplicial sets on the Zariski site of asymptotic manifolds. This is the reason we need to develop the theory of Brown--Gersten descent, which in turn requires the theory of dimension for asymptotic manifolds.

\bigskip

\underline{Here are the contents of the paper:} Asymptotic manifolds are defined in Section \ref{SectionCorners} and in Section \ref{SectionDimension} it is proved that they have well defined dimensions. Just as for the usual manifolds, there is the notion of regular values for functions on asymptotic manifolds, and Sard's theorem tells us that almost all values are regular. This allows us, in Section \ref{SectionSubmanifolds}, to define asymptotic submanifolds as solutions to equations and weak inequalities. 

In Section \ref{SectionCompactness} we single out compact and locally compact asymptotic manifolds. The full subcategory of $\fgrings^{\op}$ consisting of locally compact asymptotic manifolds, together with the Zariski topology, is the site we will be using. In order to prove descent over this site we need to measure dimensions of complements of open asymptotic submanifolds. Such complements are not asymptotic manifolds themselves, but they have enough structure to be given well defined dimensions. We call such complements asymptotic spaces and describe them in Section \ref{SectionSpaces}.

In Sections \ref{SectionDensity} and \ref{SectionBrownGersten} we prove that dimensions of asymptotic spaces, together with asymptotic submanifolds, give our site a bounded cd structure, allowing us to use the Brown--Gersten descent. This implies that to prove that a sheaf of Kan complexes is a homotopy sheaf, it is enough to show that it is soft, i.e.\@ inclusions of closed submanifolds translate into fibrations.

\smallskip

In Sections \ref{SectionSimplices} and \ref{SectionFunGro} we develop the machinery of smooth simplices and $k$-thin maps. The goal here is to obtain, for each $k\in\noneg$ and an asymptotic manifold $X$, a homotopy sheaf of smooth families of $k$-thin simplices in $X$. All these constructions are rather standard, but are quite tedious.

Then in Sections \ref{SectionInfiGroupoids}-\ref{SectionIntegrability} we use Brown--Gersten descent again to obtain infinitesimal $k$-thin groupoids, and moreover to give an explicit description of them. Connections and integrable connections are defined then as factorizations of morphisms. We also show that, if $X$ is nice enough, e.g.\@ compact and of dimension $\leq k$, integrable infinitesimal $k$-connections factor through germs of diagonals. This allows us later to integrate connections in a combinatorial way.

\smallskip

Section 3 contains the main results of this paper. In the first part we give the precise formulation of the problem of constructing line bundles on Beilinson--Drinfeld Grassmannians, and we solve this problem in the second part using asymptotic manifolds. The third part contains a simple observation that in the differential geometric setting the fibers of Beilinson--Drinfeld Grassmannians are affine, if we allow all convenient algebras in addition to $\cinfty$-rings. This immediately gives us a way to take global sections of factorizable line bundles and obtain factorization algebras.

\smallskip

The Appendix contains necessary technical facts concerning $\cinfty$-schemes.

\medskip

\underline{Some notation:} As we deal with pre-sheaves a lot, we need different notation for different Hom-functors: $\hom{}$ stands for the usual set of morphisms, $\homs{}$ denotes the simplicial set of morphisms in simplicial categories, $\homo{}$ is the internal Hom-functor in a category of pre-sheaves of sets, while $\Homo{}$ is the internal Hom-functor in a category of pre-sheaves of simplicial sets.

\tableofcontents

%%%%%%%%%%%%%%%%%%%%%%%%%%%%%%%%%%%%%%%%%%%%%%%%%%%%%%%%%%%%%%%%%%%%%%%%%%%%%%%
\section{Asymptotic manifolds and Brown--Gersten descent}

We denote by $\csite$ the opposite category of the category $\fgrings$ of finitely generated $\cinfty$-rings (e.g.\@ \cite{MR91} \S I). Objects of $\csite$ will be called {\it $\cinfty$-schemes}.\footnote{We require our $\cinfty$-schemes to be Hausdorff, implying that all of them are affine.} For $\cring\in\fgrings$, $\cscheme\in\csite$ we write $\specof{\cring}$, $\cinfty(\cscheme)$ to mean the corresponding objects in $\csite$ and $\fgrings$. The empty scheme will be denoted by $\emptys:=\specof{0}$.

\smallskip

We equip $\csite$ with the Zariski topology (e.g.\@ \cite{MR91} \S VI), and write $\zaritos{\csite}$, $\zaritoss{\csite}$ to mean the categories of pre-sheaves of sets and respectively simplicial sets on this site. We will work with several sites equipped with Zariski topology, thus we keep the site as part of the notation. 

We do not restrict our attention only to sheaves or to homotopy sheaves,\footnote{By \emph{homotopy sheaves} we mean pre-sheaves of simplicial sets that satisfy the conditions of hyper-descent (which we call \emph{homotopy descent}).} but we always consider categories of pre-sheaves together with the notion of local equivalence. In particular $\zaritoss{\csite}$ comes with two model structures: local projective and local injective (e.g.\@ \cite{DSI04}). Homotopy descent can be complicated, but everything is simplified, if one can use some descent theorems, e.g.\@ Brown--Gersten descent. This result requires existence of an appropriate theory of dimension producing a bounded density structure (e.g.\@ \cite{V10}). 

\smallskip

Because of the presence of fractals, the notion of dimension for $\cinfty$-schemes is more complicated than the one in algebraic geometry. Instead of dealing with fractal dimensions we choose to work with full sub-categories of $\csite$, consisting of $\cinfty$-schemes that satisfy some regularity conditions, ensuring integrality of the dimension.

We would like to stress that it is not enough for us to work only with manifolds or manifolds equipped with infinitesimal structure, since we would like to use germs of subschemes and germs of punctures, the latter not even having any $\mathbb R$-points. This leads us to $\cinfty$-schemes that we call {\it asymptotic manifolds}. We describe them in the first two parts of this section, and then prove Brown--Gersten descent for them.

%%%%%%%%%%%%%%%%%%%%%%%%%%%%%%%%%%%%%%%%%%%%%%%%%%%%%%%%%%%%%%%%%%%%%%%%%
\subsection{Asymptotic manifolds and spaces}

For us a (classical) $n$-dimensional manifold ($n\in\noneg$) is a non-empty, Hausdorff, second countable topological space with a chosen equivalence class of $\cinfty$-atlases consisting of $\mathbb R^n$-charts. \emph{The category of manifolds $\msite$} is a full subcategory of $\csite$ (e.g.\@ \cite{MR91} Thm.\@ I.2.8). We will arrive at asymptotic manifolds with corners by enlarging $\msite$ to a bigger full subcategory of $\csite$. First we add corners and infinitesimal structure.

%%%%%%%%%%%%%%%%%%%%%%%%%%%%%%%%%%%%%%%%%%%%%%%%%%%%%%%%%%%%%%%%%%%
\subsubsection{Definition of asymptotic manifolds with corners}\label{SectionCorners}

Since we would like to do algebraic geometry with $\cinfty$-schemes, as with any algebraic geometry over $\mathbb R$, we are naturally led to consider solutions to inequalities, both strict and weak. This means that, as a first step, we need to enlarge the category of manifolds to include manifolds with corners. Different from several definitions based on local models (e.g.\@ \cite{J16}), we use inequalities themselves as the basis for our definition. However, we need to impose some regularity conditions in order to have a well defined dimension. 

%%%%%%
\begin{definition} Let $\manifo$ be a manifold, and let $\{\charfun_1,\ldots,\charfun_m\}\subseteq\cinfty(\manifo)$. We will say that $0$ is {\it a regular value for $\{\charfun_j\}_{j=1}^m$}, if $\forall j$ $0$ is regular value for $\charfun_j$ and also for restrictions $\charfun_j|_{\underset{i\in\set}\bigcap\,\locin{\manifo}{\charfun_i=0}}$ for each $\set\subseteq\{1,\ldots,\widehat{j},\ldots,m\}$.

A non-empty $\cinfty$-scheme $\cscheme$ is {\it an $n$-dimensional manifold with corners}, if there are an $n$-dimensional manifold $\manifo$ and a set $\{\charfun_1,\ldots,\charfun_m\}\subseteq\cinfty(\manifo)$ having $0$ as a regular value, s.t.\@ writing $\cap$ for $\times$ in $\csite/\manifo$ we have\footnote{Definitions of closures and solutions to inequalities are given in Def.\@ \ref{DefinitionLocallyClosed} and \ref{Inequalities}.}
	\begin{equation}\label{DefCorners}\cscheme\cong\underset{1\leq j\leq m}\bigcap\,\locin{\manifo}{\charfun_j\leq 0}=
	\clos{\underset{1\leq j\leq m}\bigcap\,\locin{\manifo}{\charfun_j< 0}}.\end{equation}
The empty $\cinfty$-scheme $\emptys$ is the $-1$-dimensional manifold with corners. The data $\{\manifo,\charfun_1,\ldots,\charfun_2\}$ will be called {\it a realization} of $\cscheme$. The full subcategory of $\csite$ consisting of manifolds with corners will be denoted by $\bsite$.\end{definition}
%%%%%%
Notice that a choice of realization is not part of the structure of a manifold with corners. However, the ability to choose one will be used often. The requirement in (\ref{DefCorners}) that $\cscheme$ is the closure of the set of solutions to strict inequalities implies that for $n\geq 0$ every $n$-dimensional manifold with corners contains a dense subscheme, that is an $n$-dimensional manifold. As $\emptys$ is the only $-1$-dimensional manifold with corners, it is clear then that manifolds with corners have well defined dimensions.

%%%%%%
\begin{definition}\label{RegularCorners} Let $\cscheme\in\bsite$, and let $\charfun\in\cinfty(\cscheme)$. An $r\in\mathbb R$ is {\it a regular value for $\charfun$}, if there is a realization $\{\manifo,\charfun_1,\ldots,\charfun_m\}$ of ${\cscheme}$ and an extension $\preimage{\charfun}\in\cinfty(\manifo)$ of $\charfun$, s.t.\@ $0$ is a regular value for $\{\charfun-r,\charfun_1,\ldots,\charfun_m\}$.\end{definition}
%%%%%%

%%%%%%
\begin{remark} Given a regular value $r$ for $\charfun$ on a manifold with corners $\cscheme$,  it is clear that, if $\locin{\cscheme}{\charfun\leq r}={\clos{\locin{\cscheme}{\charfun<r}}}$, then $\locin{\cscheme}{\charfun\leq r}\in\bsite$. Any non-empty open subscheme of $\cscheme\in\bsite$ is also a manifold with corners of the same dimension. Product (computed in $\csite$) of an $n_1$- and an $n_2$-dimensional manifolds with corners is an $n_1+n_2$-dimensional manifold with corners.\end{remark}
%%%%%%

As we have discussed in the Introduction, instead of trying to describe asymptotic points separately, we would like to work with many of them at once, so that we can have some smooth structure available. Asymptotic points can be described as maximal filters of closed subsets. Consequently our asymptotic manifolds will be intersections of directed systems of open subsets. 

Recall (e.g.\@ \cite{AR94} \S I.1.A) that a partially ordered set is {\it directed}, if any two elements have a common upper bound. One can view partially ordered sets as categories, where morphisms go from smaller to larger elements. We will call such categories {\it directed}. These are special examples of filtered categories. We will say that a directed category $\diagra$ is {\it finite} or {\it infinite}, if it has finitely many or respectively infinitely many objects.

%%%%%%   Definition of a regular system of open subschemes
\begin{definition}\label{RegularSequence} Let $\cscheme\in\csite$, and let $\opca{\cscheme}$ be the category of non-empty open subschemes of $\cscheme$ and inclusions. {\it A regular system of open subschemes} of $\cscheme$ is given by a functor $\diagra^{\op}\rightarrow\opca{\cscheme}$ where $\diagra\neq\emptyset$ is a directed category, s.t.\@ for any $i\rightarrow j$ in $\diagra$ the corresponding inclusion $\opensch_j\hookrightarrow\opensch_i$ factors through $\clos{\opensch}_j\subseteq\opensch_i$. We define $\underset{i\in\diagra}\bigcap\;\opensch_i:=\underset{i}\lim\;\opensch_i$, where the limit is taken in $\csite$.\end{definition}
%%%%%%
If the directed category  $\diagra$ is finite, it has a maximal element. Then $\underset{i\in\diagra}\bigcap\opensch_i$ is just a non-empty open subscheme of $\cscheme$. Also in the infinite case it is true that $\underset{i\in\diagra}\bigcap\opensch_i\ncong\emptys$, as the following simple lemma shows.

%%%%%%   An infinite regular system gives a surjective localization morphism
\begin{lemma}\label{SurjectiveLocalization} Let $\cscheme\in\csite$, let $\{\opensch_i\}_{i\in\diagra}$ be an infinite regular system of open subschemes of $\cscheme$. Then $\underset{i\in\diagra}\bigcap\opensch_i\ncong\emptys$ and $\cinfty(\cscheme)\rightarrow\cinfty(\underset{i\in\diagra}\bigcap\opensch_i)$ is surjective.\end{lemma}
%%%%%%
\begin{proof} Consider another directed category $\diagra'$, that is obtained from $\diagra$ by splitting each $i\in\diagra$ into $i'\rightarrow i$. We define a functor from ${\diagra'}^{\op}$ to the category of all non-empty subschemes of $\cscheme$ and inclusions as follows: $i\mapsto\opensch_i$ and $i'\mapsto\clos{\opensch}_i$. This shows that $\underset{i\in\diagra}\bigcap\opensch_i=\underset{i'\in\diagra'\setminus\diagra}\bigcap\clos{\opensch}_{i'}$. For each $i\in\diagra$ let $\ideal_i:=\kernel{\cinfty(\cscheme)\rightarrow\cinfty(\clos{\opensch}_i)}$, then $\cinfty(\underset{i'\in\diagra'\setminus\diagra}\bigcap\clos{\opensch}_{i'})\cong\cinfty(\cscheme)/\underset{i\in\diagra}\bigcup\ideal_i$. As $\forall i$ $\clos{\opensch}_i\neq\emptys$, clearly $1\notin\underset{i\in\diagra}\bigcup\ideal_i$, i.e.\@ it is a proper ideal and $\underset{i\in\diagra}\bigcap\opensch_i\ncong\emptys$.\end{proof}%

\smallskip

We would like to have infinitesimal tools available, hence we need to work with non-reduced $\cinfty$-schemes. We would like to use the much richer infinitesimal structure from \cite{BK}. Given $\cscheme\in\csite$ the corresponding {\it reduced $\cinfty$-scheme} is $\reduiof{\cscheme}:=\specof{\cinfty(\cscheme)/\iradical{0}}$, where $\iradical{0}\leq\cinfty(\cscheme)$ consists of $\infty$-nilpotent elements (\cite{BK} Def.\@ 2).

%%%%%%   Definition of an asymptotic manifold with corners
\begin{definition}\label{SmoothCell} An $\cscheme\in\csite$ is {\it a reduced $n$-dimensional asymptotic manifold with corners} ($n\geq 0$), if there is an $n$-dimensional $\manifo\in\bsite$ and a regular system $\{\opensch_i\}_{i\in\diagra}$ of open subschemes of $\manifo$ (Def.\@ \ref{RegularSequence}), s.t.\@ $\cscheme\cong\underset{i\in\diagra}\bigcap\opensch_i$. An $\cscheme\in\csite$ is {\it an $n$-dimensional asymptotic manifold with corners} ($n\geq 0$), if $\reduiof{\cscheme}$ is a reduced $n$-dimensional asymptotic manifold with corners.\footnote{In Section \ref{SectionDimension} we show that dimensions of asymptotic manifolds are well defined.} 

{\it The $-1$-dimensional asymptotic manifold} is the empty $\cinfty$-scheme $\emptys$. The data $\{\manifo,\{\opensch_i\}_{i\in\diagra}\}$ will be called {\it a presentation of $\cscheme$}. If $\manifo$ can be chosen to be a (usual) manifold, we will say that $\cscheme$ is {\it an asymptotic manifold}. We denote by $\gsite\subset\csite$ the full subcategory consisting of asymptotic manifolds with corners.\end{definition}
%%%%%%

%%%%%%   The localization morphism can always be made surjective
\begin{remark} Notice that according to Lemma \ref{SurjectiveLocalization} an infinite regular system of open subschemes produces a surjective localization map. Since open subschemes of manifolds with corners are themselves manifolds with corners, it follows that every $\cscheme\in\gsite$ has a presentation $\{\manifo,\{\opensch_i\}_{i\in\diagra}\}$, s.t.\@ $\cinfty(\manifo)\rightarrow\cinfty(\reduiof{\cscheme})$ is surjective. \hide{%
The comment on open subschemes of manifolds with corners is needed to accommodate for the case of a finite regular system.}%
\end{remark}
%%%%%%

Before considering examples we need the following lemma.

%%%%%%   Open subschemes of asymptotic manifolds with corners are asymptotic manifolds corners
\begin{lemma}\label{OpenSubcells} Let  $\cscheme\in\csite$ be an $n$-dimensional asymptotic manifold with corners ($n\geq 0$), and let $\opensch\subseteq\cscheme$ be an open subscheme, $\opensch\ncong\emptys$. Then $\opensch$ is an $n$-dimensional asymptotic manifold with corners. \end{lemma}
%%%%%%
\begin{proof} Let $\charfun\in\cinfty(\cscheme)$, s.t.\@ $\cinfty(\opensch)\cong\cinfty(\cscheme)\{\charfun^{-1}\}$, and let $\{\manifo,\{\opensch_i\}_{i\in\diagra}\}$ be a presentation of $\cscheme$, s.t.\@ $\cinfty(\manifo)\rightarrow\cinfty(\reduiof{\cscheme})$ is surjective. We choose a pre-image $\preimage{\charfun}\in\cinfty(\manifo)$ of the image of $\charfun$ in $\cinfty(\reduiof{\cscheme})$, and let $\manifo':=\spec(\cinfty(\manifo)\{\preimage{\charfun}^{-1}\})$. Clearly $\reduiof{\opensch}\cong\underset{i\in\diagra}\bigcap(\manifo'\cap\opensch_i)$. The system $\{\manifo'\cap\opensch_i\}_{i\in\diagra}$ is regular, since if $\exists k$ s.t.\@ $\manifo'\cap\opensch_k=\emptys$, then also $\opensch\cong\emptys$.\end{proof}%

%%%%%%   Basic examples of asymptotic manifolds
\begin{example}\label{BasicExamples}\begin{enumerate}[label={ (\alph*)}] 
\item All manifolds with corners are asymptotic manifolds with corners.
\item Let $\manifo$ be a manifold with corners, let $\closedset\subseteq\manifo$ be a compact subset $\closedset\neq\emptyset$. We embed $\manifo\hookrightarrow\mathbb R^m$ and $\forall k\in\mathbb N$ let $\opensch_k\subseteq\manifo$ be the open subscheme consisting of points of distance  $<\frac{1}{k}$ from $\closedset$. Then $\underset{k\in\mathbb N}\bigcap\opensch_k$ is an asymptotic manifold with corners -- {\it the germ of $\manifo$ at $\closedset$}. 
\item\label{Compact} If $\closedset$ is not compact, we can write $\closedset=\underset{s\in\mathbb N}\bigcup\,\closedset_s$, where each $\closedset_s$ is compact. For each $\alpha\colon\mathbb N\rightarrow\mathbb N$ let $\opensch_\alpha\subseteq\manifo$ be the open subscheme consisting of $\rpt\in\manifo$ s.t.\@ $\forall s\in\mathbb N$ $\forall\rqt\in\closedset_s$ $\|\rpt-\rqt\|<\frac{1}{\alpha(s)}$. Then $\germof{\manifo}{\closedset}:=\underset{\alpha}\bigcap\,\opensch_\alpha$ is an asymptotic manifold with corners -- {\it the germ of $\manifo$ at $\closedset$}.
\item\label{NonCompact} Let $\charfun\in\cinfty(\manifo)$, s.t.\@ $\vset{\rpt}{\manifo}{\charfun=0}=\closedset$, and let $\ima{\charfun}\in\cinfty(\germof{\manifo}{\closedset})$ be the image of $\charfun$. {\it The germ of $\manifo$ at the puncture $\manifo\setminus\closedset$} is $\puncof{\manifo}{\closedset}:=\specof{\cinfty(\germof{\manifo}{\closedset})\{\ima{\charfun}^{-1}\}}$. This is an asymptotic manifold with corners.
\item\label{NonGerm} Not all asymptotic manifolds with corners are germs. Consider $\set:=\{f_k\}_{k\geq 1}\subseteq\cinfty(\mathbb R^2)$, s.t.\@ $f_k=0$ exactly on $\mathbb R^2\setminus((-\frac{1}{k},\frac{1}{k})\times\mathbb R)$. Then $\spec(\cinfty(\mathbb R^2)\{\set^{-1}\})$ is an asymptotic manifold, but it is not the germ of $\mathbb R^2$ at the $y$-axis (e.g.\@ \cite{MR91}, p.\@ 49). \hide{%
We can write $\cinfty(\mathbb R^2)\{\set^{-1}\}\cong\cinfty(\mathbb R^2)/\mathfrak m_\epsilon$, where $f\in\cinfty(\mathbb R^2)$ belongs to $\mathfrak m_\epsilon$, if $\exists\epsilon>0$, s.t.\@ $f$ vanishes on $(-\epsilon,\epsilon)\times\mathbb R$. Let $g\in\cinfty(\mathbb R^2)$ be any function that vanishes on the connected component of $0\in\mathbb R^2\setminus\{|x y|=1\}$, but does not vanish anywhere on the other components. Clearly $g\in\zerogerm{x=0}$ yet $g\notin\mathfrak m_\epsilon$.}%
\end{enumerate}\end{example}
%%%%%%

Examples \ref{BasicExamples}.\ref{Compact}, \ref{BasicExamples}.\ref{NonCompact}, \ref{BasicExamples}.\ref{NonGerm} exhibit a general fact: asymptotic manifolds with corners built as germs around compact sets can be constructed using regular {\it sequences}, while non-compact sets require regular {\it systems} (in particular uncountable). We take a closer look at this in Section \ref{SectionCompactness}.

%%%%%%%%%%%%%%%%%%%%%%%%%%%%%%%%%%%%%%%%%%%%%%%%%%%%%%%%%%%%%%%%%%%
\subsubsection{Dimension of asymptotic manifolds with corners}\label{SectionDimension}

Even though dimension was part of the definition of asymptotic manifolds with corners, we have not shown yet that the same $\cinfty$-scheme cannot be an asymptotic manifold with corners of dimensions $m,n$ simultaneously, with $m\neq n$. We do this in this section. In fact we show something more: a locally closed subscheme\footnote{For locally closed subschemes see Def.\@ \ref{DefinitionLocallyClosed}.} of an $n$-dimensional asymptotic manifold with corners cannot be an asymptotic manifold with corners of higher dimension. 

%%%%%%   Locally closed asymptotic submanifold has smaller dimension
\begin{proposition}\label{DimensionOfCell} Let $\cscheme'\in\csite$, and let $\cmor\colon\cscheme\hookrightarrow\cscheme'$ be a locally closed subscheme, s.t.\@ $\cscheme$ is an $m$-dimensional and  $\cscheme'$ is an $n$-dimensional asymptotic manifold with corners. Then $m\leq n$.\end{proposition}
%%%%%%
\begin{proof} As asymptotic manifolds with corners are defined by putting conditions on their reduced parts, we can assume that both $\cscheme$ and $\cscheme'$ are reduced. 

According to Lemma \ref{SurjectiveLocalization} if $n\geq 0$, an $n$-dimensional asymptotic manifold with corners cannot be $\emptys$, thus we can assume $m,n\geq 0$. There is an open subscheme $\opensch\subseteq\cscheme'$, s.t.\@ $\cmor\colon\cscheme\rightarrow\opensch$ is a closed embedding. As $\opensch$ is an asymptotic manifold with corners of dimension $n$ ( Lemma \ref{OpenSubcells}) we can assume that $\cmor$ is a closed embedding. Let $\{\manifo,\{\opensch_i\}_{i\in\diagra}\}$, $\{\manifo',\{\opensch'_j\}_{j\in\diagra'}\}$ be some presentations of $\cscheme$ and $\cscheme'$ respectively. 

\smallskip

If there is at least one $\mathbb R$-point $\rpt$ in $\cscheme$, the claim trivially follows from comparing cotangent spaces. \hide{%
Indeed, since localizations commute, we have isomorphisms of germs of $\cinfty$-schemes
	\begin{equation*}\germ{\manifo}{\rpt}\cong\germ{\cscheme}{\rpt},\quad\germ{\cscheme'}{\rpt}\cong\germ{\manifo'}{\rpt},\end{equation*}
where we use the same notation $\rpt$ for the image of this $\mathbb R$-point in $\cscheme'$, $\manifo'$, $\manifo$. Thus $\cmor$ localizes to a closed embedding $\germ{\cmor}{\rpt}\colon\germ{\manifo}{\rpt}\rightarrow\germ{\manifo'}{\rpt}$, that induces a surjection $\mathfrak m'/\mathfrak m'^2\rightarrow\mathfrak m/\mathfrak m^2$, where $\mathfrak m\leq\cinfty(\germ{\cscheme}{\rpt})$, $\mathfrak m'\leq\cinfty(\germ{\cscheme'}{\rpt})$ are the maximal ideals. Hence $m\leq n$. }%
Therefore we can assume that $\cscheme$ does not have $\mathbb R$-points. Also in this case we proceed by comparing points, but this time ``asymptotic points'' rather than $\mathbb R$-points.

\smallskip

Let $n\in\noneg$ be the smallest with the property that $\exists m>n$ and a closed embedding $\cmor\colon\cscheme\rightarrow\cscheme'$ as above. I.e.\@ $\cscheme=\underset{i\in\diagra}\bigcap\opensch_i$, $\cscheme'=\underset{j\in\diagra'}\bigcap\opensch'_j$, and $\manifo$, $\manifo'$ are $m$- and $n$-dimensional manifolds with corners. We choose a well ordering of the set of objects in $\diagra$, and construct $\{\charfun_k\}\subseteq\cinfty(\manifo)$ inductively as follows: let $i\in\diagra$ be the first s.t.\@ $\opensch_i$ is disjoint from the union of supports of $\charfun_k$'s already constructed. Choose a function whose support is diffeomorphic to $\mathbb R^m$ and lies within $\opensch_i$, and add it to the end of the sequence $\{\charfun_k\}$. Using transfinite induction we obtain in the end $\charfun:=\underset{k}\sum\,\charfun_k$ s.t.\@
	\begin{equation*}\forall i\in\diagra\quad\locin{\manifo}{\charfun\neq 0}\cap\opensch_i\neq\emptyset,\end{equation*}
and $\locin{\manifo}{\charfun\neq 0}$ is a disjoint union of copies of $\mathbb R^m$.\footnote{The sum $\underset{k}\sum\,\charfun_k$ is countable (because $\manifo$ is second-countable), but it can be parameterized by an ordinal larger than $\mathbb N$.} Let $\ima{\charfun}\in\cinfty(\cscheme)$ be the image of $\charfun$ in $\cinfty(\cscheme)$, let $\ima{\charfun'}\in\cinfty(\cscheme')$ be any pre-image of $\ima{\charfun}$, and let $\charfun'\in\cinfty(\manifo')$ be any pre-image of $\ima{\charfun'}$. As $\charfun$ has non-zero values on every $\opensch_i$, clearly $\locin{\cscheme}{\ima{\charfun}\neq 0}\ncong\emptys$, hence $\locin{\cscheme'}{\ima{\charfun'}\neq 0}\ncong\emptys$, and they are $m$- and $n$-dimensional asymptotic manifolds respectively (Lemma \ref{OpenSubcells}). Explicitly
	\begin{equation*}\locin{\cscheme}{\ima{\charfun}\neq 0}\cong\underset{i\in\diagra}\bigcap(\locin{\manifo}{\charfun\neq 0}\cap\opensch_i),\quad
	\locin{\cscheme'}{\ima{\charfun'}\neq 0}\cong
	\underset{j\in\diagra'}\bigcap(\locin{\manifo'}{\charfun'\neq 0}\cap\opensch'_j).\end{equation*}
By construction $\locin{\manifo}{\charfun\neq 0}$ is a disjoint union of a countable set of copies of $\mathbb R^m$. Informally $\prinop{\cscheme}{\ima{\charfun}}$ is a bunch of ``asymptotic points'' in $\mathbb R^m$.

By assumption $m>n\geq 0$, so let $x\in\cinfty(\locin{\manifo}{\charfun\neq 0})$ be the function that restricts to $x_1$ on each copy of $\mathbb R^m$. Let $\ima{x}\in\cinfty(\locin{\cscheme}{\charfun\neq 0})$ be the image of $x$, $\ima{x'}\in\cinfty(\locin{\cscheme'}{\charfun'\neq 0})$ any pre-mage of $\ima{x}$,  and $x'\in\cinfty(\locin{\manifo'}{\charfun'\neq 0})$ any pre-image of $\ima{x'}$. For any $r\in\mathbb R$ the closed subscheme of $\locin{\cscheme}{\ima{\charfun}\neq 0}$ defined by $\ima{x}-r$ is not $\emptys$, therefore $x'\colon\locin{\manifo'}{\charfun'\neq 0}\rightarrow\mathbb R$ is surjective.

Since $\locin{\manifo'}{\charfun'\neq 0}$ is an $n$-dimensional manifold with corners, by Sard's theorem we can choose $r\in\mathbb R$, s.t.\@ $\locin{\manifo'}{\charfun'\neq 0,x'=r}$ is an $n-1$-dimensional manifold with corners. Since every value of $x$ on $\locin{\manifo}{\charfun\neq 0}$ is regular, clearly $\locin{\manifo}{\charfun\neq 0,x=r}$ is an $m-1$-dimensional manifold. Then we have
	\begin{equation}\label{ClosedSubscheme}\underset{i\in\diagra}\bigcap(\locin{\manifo}{\charfun\neq 0,x=r}\cap\opensch_i)\cong
	\spec(\cinfty(\locin{\cscheme}{\ima{\charfun}\neq 0})/\igen{\ima{x}-r}),\end{equation}
	\begin{equation}\label{AmbientScheme}\underset{j\in\diagra'}\bigcap(\locin{\manifo'}{\charfun'\neq 0,x'=r}\cap\opensch'_j)\cong
	\spec(\cinfty(\locin{\cscheme'}{\ima{\charfun'}\neq 0})/\igen{\ima{x'}-r}).\end{equation}
As (\ref{ClosedSubscheme}) is an $m-1$-dimensional asymptotic manifold, sitting as a closed subscheme in (\ref{AmbientScheme}), which is an $n-1$-dimensional asymptotic manifold with corners, this contradicts minimality of $n$.\end{proof}%

%%%%%%   Dimensions of asymptotic manifolds are well defined
\begin{remark} Since identity is a closed embedding, Prop.\@ \ref{DimensionOfCell} shows that any asymptotic manifold with corners $\cscheme$ has a well defined dimension $\dim\cscheme$.\end{remark}
%%%%%%

%%%%%%%%%%%%%%%%%%%%%%%%%%%%%%%%%%%%%%%%%%%%%%%%%%%%%%%%%%%%%%%%%%%
\subsubsection{Asymptotic submanifolds}\label{SectionSubmanifolds}

As with usual manifolds we would like to be able to define asymptotic submanifolds as solutions to equations, subject to some regularity conditions. Since we are working over $\mathbb R$, we would like to also have solutions to weak inequalities. We start with regularity conditions.

%%%%%%   Definition of regular values of functions on asymptotic manifolds
\begin{definition} Let $\cscheme\in\gsite$. For an $\charfun\in\cinfty(\cscheme)$ an $r\in\mathbb R$ is {\it a regular value}, if there is a presentation $\{\manifo,\{\opensch_i\}_{i\in\diagra}\}$ of $\cscheme$, s.t.\@ $r$ is a regular value for a pre-image in $\cinfty(\manifo)$ of the image of $\charfun$ in $\cinfty(\reduiof{\cscheme})$ (Def.\@ \ref{RegularCorners}). The set of regular values for $\charfun$ will be denoted by $\regva{\charfun}$.\end{definition}
%%%%%%
The classical theorem of Sard immediately implies the following.

%%%%%%   Almost all values are regular
\begin{proposition}\label{SardManifold} For any $\cscheme\in\gsite$ and $\forall\charfun\in\cinfty(\cscheme)$ the set $\mathbb R\setminus\regva{\charfun}$  has Lebesgue measure $0$.\end{proposition}
%%%%%%
\hide{%
\begin{proof} Let $\{\manifo,\{\opensch_i\}_{i\in\diagra}\}$ be a presentation of $\cscheme$ s.t.\@ $\cinfty(\manifo)\rightarrow\cinfty(\reduiof{\cscheme})$ is surjective. Let $\preimage{\charfun}\in\cinfty(\manifo)$ be any pre-image of $\charfun$. Choosing a realization $\{\manifo',\charfun_1,\ldots,\charfun_m\}$ of ${\manifo}$ we see that an $r\in\mathbb R$ is a regular value for $\charfun$, if it is a regular value for ${\charfun}$ on $\underset{1\leq k\leq m}\bigcap\locin{\manifo'}{\charfun_k<0}$ and on $\underset{k\in\set}\bigcap\locin{\manifo'}{\charfun_k=0}$ for each $\set\subseteq\{1,\ldots,k\}$. This is a finite set of manifolds of various dimensions, hence we can use Sard's theorem.\end{proof}

\smallskip}%

The following statement is only slightly more complicated.
%%%%%%
\begin{proposition}\label{ClosureRegular} Let $\cscheme\in\gsite$ and let $r\in\regva{\charfun}$ for $\charfun\in\cinfty(\cscheme)$. Then $\clos{\locin{\cscheme}{\charfun<r}}$ is an asymptotic manifold with corners.\footnote{For the notion of closure and solutions to inequalities see Def.\@ \ref{DefinitionLocallyClosed} and \ref{Inequalities}.}\end{proposition}
%%%%%%
\begin{proof} We can assume that $\cscheme$ is reduced. Let $\{\manifo,\{\opensch_i\}_{i\in\diagra}\}$ be a presentation of $\cscheme$, and let $\preimage{\charfun}\in\cinfty(\manifo)$ be a pre-image of $\charfun$. We can assume that $r$ is a regular value for $\preimage{\charfun}$. We denote $\ideal:=\kernel{\cinfty(\cscheme)\rightarrow\cinfty(\locin{\cscheme}{\charfun<r})}$ and $\preimage{\ideal}:=\kernel{\cinfty(\manifo)\rightarrow\cinfty(\locin{\manifo}{\preimage{\charfun}<r})}$. Then we have a commutative diagram
	\begin{equation*}\xymatrix{\cinfty(\manifo)\ar[d]\ar[rr] && \cinfty(\cscheme)\ar[d]\\
	\cinfty(\manifo)/\preimage{\ideal}\ar[d]\ar[rr] && \cinfty(\cscheme)/\ideal\ar[d]\\
	\cinfty(\locin{\manifo}{\preimage{\charfun}<r})\ar[rr] && \cinfty(\locin{\cscheme}{\charfun<r}).}\end{equation*}
An $\alpha\in\cinfty(\manifo)$ is mapped to $\ideal$, iff $\exists i$ $\alpha|_{\opensch_i\cap\locin{\manifo}{\preimage{\charfun}<r}}=0$. This means that $\cinfty(\cscheme)/\ideal\cong\underset{i\in\diagra}\colim\,\cinfty(\clos{\locin{\manifo}{\preimage{\charfun}<r}}\cap\opensch_i)$ and it is enough to prove that ${\clos{\locin{\manifo}{\preimage{\charfun}<r}}}$ is a manifold with corners. It might happen that $\locin{\manifo}{\preimage{\charfun}\leq r}\neq{\clos{\locin{\manifo}{\preimage{\charfun}<r}}}$, but since $\locin{\manifo}{\preimage{\charfun}\leq r}\setminus{\clos{\locin{\manifo}{\preimage{\charfun}<r}}}$ is a closed subset of ${\manifo}$, it is clear that ${\clos{\locin{\manifo}{\preimage{\charfun}<r}}}$ is a manifold with corners. \hide{%
To see that $\locin{\manifo}{\preimage{\charfun}\leq r}\setminus{\clos{\locin{\manifo}{\preimage{\charfun}<r}}}$ is closed notice that it consists only of isolated points -- the corners of $\manifo$. Any part of the boundary of $\manifo$ that has dimension $>0$ will intersect $\{\preimage{\charfun}=r\}$ transversely, and hence this intersection will be in ${\clos{\locin{\manifo}{\preimage{\charfun}<r}}}$.}%
\end{proof}%

%%%%%%
\begin{remark}\label{ValueRegular} Given $\cscheme\in\gsite$ and $\charfun\in\cinfty(\cscheme)$, not for every $r\in\regva{\charfun}$ the subscheme $\locin{\cscheme}{\charfun=r}$ is an asymptotic manifold with corners. The problem might be as follows: using the notation from the proof of Prop.\@ \ref{ClosureRegular} we might have that $\preimage{\charfun}$ has value $r$ on a corner of $\manifo$, but nowhere else in the neighbourhood of this corner in $\manifo$. Since there are at most countably many corners in $\manifo$, it is clear that $\locin{\cscheme}{\charfun=r}$ is an asymptotic manifold for almost all $r\in\mathbb R$. \end{remark}
%%%%%%

%%%%%%
\begin{definition}\label{DefSubmanifolds} Subschemes $\cscheme'\subseteq\cscheme$ as in Prop.\@ \ref{ClosureRegular} and Rem.\@ \ref{ValueRegular} will be called {\it asymptotic submanifolds}.\end{definition}
%%%%%%
Asymptotic manifolds with corners were defined using intersections of regular systems of open subschemes in manifolds with corners. One can ask what happens if we take such intersections within asymptotic manifolds with corners themselves. We will answer this question for regular {\it sequences} only, because these are the cases we will be using.

%%%%%%
\begin{lemma}\label{SoftOnManifolds} Let $\cscheme\in\gsite$ and let $\{\manifo,\{\opensch_i\}_{i\in\diagra}\}$ be a presentation of $\cscheme$. Let $\{\opensch'_k\}_{k\in\mathbb N}$ be a regular sequence of open subschemes of $\cscheme$, and let $\cscheme':=\underset{k\in\mathbb N}\bigcap\opensch'_k$. There are a directed category $\diagra'$, a cofinal functor $\lambda\colon\diagra'\rightarrow\diagra$,\footnote{A functor $\lambda\colon\diagra'\rightarrow\diagra$ is cofinal, if $\forall i\in\diagra$ there is $i\rightarrow\lambda(j)$.} and a regular system $\{\opensch''_j\}_{j\in\diagra'}$ of open subschemes of $\manifo$, such that $\forall j\in\diagra'$ $\opensch''_j\subseteq\opensch_{\lambda(j)}$ and we have a commutative diagram
	\begin{equation*}\xymatrix{\reduiof{\cscheme'}\ar[rr]^\cong\ar[d] && \underset{j\in\diagra'}\bigcap\opensch''_j\ar[d]\\
	\reduiof{\cscheme}\ar[rr]^\cong && \underset{i\in\diagra}\bigcap\opensch_{i}.}\end{equation*}
In particular $\cscheme'$ is an asymptotic manifold with corners of dimension $\dim\cscheme$.
\end{lemma}
%%%%%%
\begin{proof} Since localization commutes with reduction, we can assume that $\cscheme$ is reduced. Let $\{\charfun_k\}_{k\in\mathbb N}\subseteq\cinfty(\cscheme)$ be s.t.\@$\forall k$ $\opensch'_k=\locin{\cscheme}{\charfun_k\neq 0}$ and $\charfun_{k}$ factors $\charfun_{k+1}$. For each $k$ denote $\ideal_k:=\kernel{\cinfty(\cscheme)\rightarrow\cinfty(\opensch'_k)}$. Regularity of $\{\opensch'_k\}_{k\in\mathbb N}$ implies that $\forall k$ there are $\charfun'_{k}\in\cinfty(\cscheme)$, $\alpha_{k}\in\ideal_{k+1}$ s.t.\@ $\charfun_k\charfun'_{k}=1+\alpha_{k}$. We choose pre-images $\{\preimage{\charfun}_k,\preimage{\charfun'}_{k},\preimage{\alpha}_{k}\}_{k\in\mathbb N}\subseteq\cinfty(\manifo)$ s.t.\@ $\forall k$ $\preimage{\charfun}_k$ factors $\preimage{\charfun}_{k+1}$. For an $i\in\diagra$ let $\ideal_i:=\kernel{\cinfty(\manifo)\rightarrow\cinfty(\opensch_i)}$, $\forall k\in\mathbb N$ let $i_k\in\diagra$, $\beta_{i_k}\in\ideal_{i_k}$ s.t.\@
	\begin{equation*}\preimage{\charfun}_k\preimage{\charfun'}_k=1+\preimage{\alpha}_k+\beta_{i_k},\quad k'>k\Rightarrow i_{k'}>i_{k},\quad
	\preimage{\alpha}_k|_{\locin{\manifo}{\preimage{\charfun}_{k+1}\neq 0}\cap\,\opensch_{i_{k+1}}}=0.\end{equation*}
We define $\diagra':=\{(k,j)\,|\,k\in\mathbb N, j\in\diagra\text{ s.t.\@ }j\geq i_k\}$ and $\opensch''_{k,j}:=\locin{\manifo}{\preimage{\charfun}_k\neq 0}\cap\,\opensch_j$. This is a regular system of open subschemes of $\manifo$. The image of the obvious projection $\diagra'\rightarrow\diagra$ is cofinal in $\diagra$, hence $\underset{(k,j)\in\diagra'}\bigcap\opensch''_{k,j}\hookrightarrow\manifo$ factors through $\cscheme$, and therefore $\underset{(k,j)\in\diagra'}\bigcap\opensch''_{k,j}\cong\cscheme'$.\end{proof}%

%%%%%%%%%%%%%%%%%%%%%%%%%%%%%%%%%%%%%%%%%%%%%%%%%%%%%%%%%%%%%%%%%%%
\subsubsection{Compactness and local compactness}\label{SectionCompactness}

Now we come to the distinction between objects in $\gsite$ defined as intersections of regular systems and as intersections of regular sequences (Ex.\@ \ref{BasicExamples}). In the world of asymptotic objects this distinction reflects the notion of compactness for the usual manifolds. First we recall some terminology (e.g.\@ \cite{MR91} \S I.4): a $\cinfty$-ring $\cring$ is {\it point determined}, if $\cring\cong\cinfty(\mathbb R^n)/\ideal$, and $\forall\charfun\in\cinfty(\mathbb R^n)$ that vanishes at the common zeroes of $\ideal$ belongs to $\ideal$.

%%%%%%
\begin{definition} Let $\cscheme\in\gsite$. We will say that $\cscheme$ is {\it compact}, if there is a presentation $\{\manifo,\{\opensch_i\}_{i\in\diagra}\}$ of $\cscheme$ and a dense embedding $\manifo\hookrightarrow\cscheme'$, s.t.\@ ${\cscheme'}$ is point determined, and there is a compact $K\subseteq\cscheme'$ with $\reduiof{\cscheme}\cong\germof{\manifo}{K}$.\footnote{Recall (Def.\@ \ref{DefinitionOfGerm}) that $\germof{\manifo}{K}$ means the germ of $\manifo$ at $K$.} The data $\{\manifo,\cscheme',K\}$ will be called {\it a compact presentation of $\cscheme$}.\end{definition}
%%%%%%

%%%%%%
\begin{example}\label{CompactExamples}\begin{enumerate}[label={ (\alph*)}] 
\item Every manifold with corners, that is compact in the usual sense, is also compact according to our definition. 
\item Germs at compact subsets are clearly compact, as well as germs at such punctures. 
\item\label{UpperHalfPlane} The germ of the open upper half-plane at the $x$-axis is not compact. \hide{%
	Indeed, let $\cscheme$ be this germ, and let $\{\manifo,\cscheme',K\}$ be a compact presentation of $\cscheme$. As the open upper half-plane $\mathbb 	R_+^2$ is isomorphic to $\mathbb R^2$, the embedding $\cscheme\rightarrow\mathbb R^2_+$ can be extended to $\Phi\colon\manifo\rightarrow\mathbb 	R^2_+$. Closure of the image of $\Phi$ in $\mathbb R^2\supset\mathbb R^2_+$ has to contain the $x$-axis. Then the regular system of open subsets of 	$\mathbb R^2_+$ defining $\cscheme$ pulls back under $\Phi$ to a regular system that does not have a cofinal sub-sequence, but it has to, since germs 	at compact subsets do. }%
\end{enumerate}\end{example}
%%%%%%
It is natural to expect the following Lemma.

%%%%%%
\begin{lemma}\label{CompactSubmanifold} Let $\cscheme\in\gsite$ be compact, and let $\cscheme'\subseteq\cscheme$ be an asymptotic submanifold (Def.\@ \ref{DefSubmanifolds}). Then $\cscheme'$ is compact.\end{lemma}
%%%%%%
\begin{proof} Let $(\manifo,\cscheme'',K)$ be a compact presentation of $\cscheme$. Let $\charfun\in\cinfty(\cscheme)$ be s.t.\@ $\cscheme'=\clos{\locin{\cscheme}{\charfun<a}}$ or $\cscheme'=\locin{\cscheme}{\charfun=a}$ for a regular value $a$. There is an open $\opensch\subseteq\manifo$ s.t.\@ $\opensch\cap\reduiof{\cscheme}=\reduiof{\locin{\cscheme}{\charfun>a}}$ (respectively $\reduiof{\locin{\cscheme}{\charfun\neq a}}$). Let $K_1\subseteq K$ consist of $\rpt$, s.t.\@ there is an open $V\subseteq\cscheme''$ containing $\rpt$ with $V\cap\manifo\subseteq\opensch$. Clearly $K\setminus K_1$ is closed, and hence compact. It is easy to see that $\reduiof{\cscheme'}$ is the germ of $\manifo\setminus\opensch$ at $K\setminus K_1$.\end{proof}%

\smallskip

As examples \ref{BasicExamples}.\ref{NonGerm} and \ref{CompactExamples}.\ref{UpperHalfPlane} show there are plenty of asymptotic manifolds with corners, that are not compact, but can be broken into compact pieces. We would like to formalize this.

%%%%%%
\begin{definition} An $\cscheme\in\gsite$ is {\it locally compact}, if we can find $m\in\mathbb N$ and $\{\charfun_i\}_{i=1}^m\subset\cinfty(\cscheme)$, s.t.\@ for any set of pairs $\{a_i,b_i\}_{i=1}^m\subset\mathbb R^m$ of regular values for $\{\charfun_i\}_{i=1}^m$ the asymptotic submanifold $\underset{1\leq i\leq m}\bigcap\clos{\locin{\cscheme}{a_i<\charfun_i<b_i}}$ is compact.\end{definition}
%%%%%%
The most important example of an asymptotic manifold for us is $\germof{(\manifo\setminus\vmani)}{\vmani}$, which is the germ of a manifold $\manifo$ at a puncture, obtained by removing a finite union of submanifolds $\vmani\subseteq\manifo$. This is clearly a locally compact asymptotic manifold.

%%%%%%
\begin{proposition}\label{PropLoc} Let $\ksite\subset\csite$ be the full subcategory consisting of locally compact asymptotic manifolds with corners. Then\begin{enumerate}[label={ (\alph*)}] 
\item\label{open} $\forall\cscheme\in\ksite$ any open subscheme of $\cscheme$ is also in $\ksite$,
\item\label{submanifold} $\forall\cscheme\in\ksite$ any asymptotic submanifold of $\cscheme$ is also in $\ksite$,
\item\label{product} $\forall\cscheme_1,\cscheme_2\in\ksite$ $\cscheme_1\times\cscheme_2$, computed in $\csite$, belongs to $\ksite$.\end{enumerate}\end{proposition}
%%%%%%
\begin{proof}\begin{enumerate}[label={ (\alph*)}] 
\item We choose $\charfun\in\cinfty(\cscheme)$ s.t.\@ $\opensch=\locin{\cscheme}{\charfun<0}$. Then for any regular values $a<b<0$ we see that $\clos{\locin{\cscheme}{a<\charfun<b}}$ intersected with the compact pieces of $\cscheme$ is compact (Lemma \ref{CompactSubmanifold}).
\item Follows from Lemma \ref{CompactSubmanifold}.
\item We can assume that $\cscheme_1$, $\cscheme_2$ are reduced. Let $\{\manifo',\{\opensch'_i\}_{i\in\diagra'}\}$ and $\{\manifo'',\{\opensch''_j\}_{j\in\diagra''}\}$ be some presentations of $\cscheme_1$, $\cscheme_2$ respectively. Consider $\{\manifo'\times\manifo'',\{\opensch'_i\times\opensch''_j\}_{(i,j)\in\diagra'\times\diagra''}\}$, where $\manifo'\times\manifo''$ is computed in $\csite$. This clearly defines an object in $\gsite$. Suppose we are given another reduced $\cscheme\in\gsite$ with a presentation $\{\manifo,\{\opensch_k\}_{k\in\diagra}\}$, and let $\Phi_1\colon\cscheme\rightarrow\cscheme_1$, $\Phi_2\colon\cscheme\rightarrow\cscheme_2$ be any morphisms. As manifolds are finitely presented, both $\Phi_1$ and $\Phi_2$ lift to $\manifo\rightarrow\manifo'$, $\manifo\rightarrow\manifo''$, and then the corresponding $\manifo\rightarrow\manifo'\times\manifo''$ maps $\cscheme$ to $\underset{(i,j)\in\diagra'\times\diagra''}\bigcap(\opensch'_i\times\opensch''_j)$. This shows that product of objects in $\gsite$, computed in $\csite$, belongs to $\gsite$. Starting with objects of $\ksite$, means that we can decompose $\manifo'$, $\manifo''$ into pieces, that are germs at compacts. Since direct products of compact point determined $\cinfty$-schemes are again compact, this finishes the proof.\end{enumerate}\end{proof}%

\smallskip

Plenty of objects in $\gsite$ are not locally compact. 
%%%%%%
\begin{example} Let $\cscheme$ be the germ of $\mathbb R^2\setminus 0$ at $0$. The $y$-axis defines an asymptotic submanifold $\cscheme'\subset\cscheme$. Let $\cscheme''\subset\cscheme$ be the germ of $\cscheme$ at $\cscheme'$. It is easy to see that $\cscheme''$ is not locally compact. \end{example}
%%%%%%
The previous example shows an important fact: even locally compact asymptotic manifolds with corners are not first countable, meaning that germs at closed subschemes do not have to have a fundamental system of open neighbourhoods. The following lemma shows something more: quite often intersection of any regular sequence of open neighbourhoods contains another open neighbourhood.

%%%%%%
\begin{lemma}\label{OpenInSoft} For a compact $\cscheme\in\ksite$ without $\mathbb R$-points let $\cscheme'\subseteq\cscheme$ be an asymptotic submanifold. Then for any regular sequence $\{\opensch'_i\}_{i\in\mathbb N}$ of open subschemes of $\cscheme$, s.t.\@ $\cscheme'\subseteq\underset{i\in\mathbb N}\bigcap\opensch'_i$, there is an open subscheme $\opensch\subseteq\cscheme$, s.t.\@ $\cscheme'\subseteq\opensch\subseteq\underset{i\in\mathbb N}\bigcap\opensch'_i$.\end{lemma}
%%%%%%
\begin{proof} We can assume that $\cscheme$ is reduced. Let $\{\manifo,\{\opensch_i\}_{i\in\mathbb N}\}$ be a presentation of $\cscheme$. Let $\charfun\in\cinfty(\cscheme)$ be s.t.\@ $\cscheme'=\clos{\locin{\cscheme}{\charfun< r}}$ for a regular $r\in\mathbb R$. Let $\preimage{\charfun}\in\manifo$ be any pre-image of $\charfun$, and let $\closedset=\clos{\locin{\manifo}{\preimage{\charfun}<r}}$. Using Lemma \ref{SoftOnManifolds} we can assume that there is a regular sequence $\{\opensch''_i\}_{i\in\mathbb N}$ of open subscheme in $\manifo$, s.t.\@ $\forall i$ $\opensch''_i\subseteq\opensch_i$ and $\underset{i\in\mathbb N}\bigcap\opensch'_i\cong\underset{i\in\mathbb N}\bigcap\,\opensch''_i$. Choosing a subsequence, if necessary, we can assume that $\forall i$ $\opensch_i\cap V\subseteq\opensch''_i$.

We choose $\{\charfun_i\}_{i\in\mathbb N}\subset\cinfty(\manifo)$ s.t.\@ $\forall i$ $\locin{\manifo}{\charfun_i\neq 0}=\opensch''_i$ and $\{\charfun'_i\}_{i\in\mathbb N}\subset\cinfty(\manifo)$ s.t.\@ $\charfun'_i$ vanishes exactly on the complement of $\opensch_i\setminus\clos{\opensch}_{i+2}$. The sequence $\{\charfun_i^2{\charfun'}_i^2\}_{i\in\mathbb N}$ is locally finite and $\charfun:=\underset{i\in\mathbb N}\sum\charfun_i^2{\charfun'}_i^2$ does not vanish anywhere on $V$. On the other hand $\forall i\in\mathbb N$ $\locin{\manifo}{\charfun\neq 0}\cap\opensch_{i+1}\subseteq\opensch''_i$. Therefore $\opensch:=\cscheme\cap\locin{\manifo}{\charfun\neq 0}$ is the open subscheme of $\cscheme$ that we have wanted.\end{proof}%

%%%%%%%%%%%%%%%%%%%%%%%%%%%%%%%%%%%%%%%%%%%%%%%%%%%%%%%%%%%%%%%%%%%%%%%%%
\subsubsection{Asymptotic spaces}\label{SectionSpaces}
%%      Definition of $C^\infty$-cell complexes and cell decompositions
%%      Dimensions of $C^\infty$-cell complexes are well defined
%%      Definition of dimension of a $C^\infty$-cell complex
%%      Open subschemes are cell complexes too
%      Reformulation of $C^\infty$-cell decomposition as a sequence of open subschemes
%      Definition of $m$-dense open subschemes
The category $\ksite$, together with the Zariski topology, will be our site, where we will be defining connections and performing transgression. We would like to use Brown--Gersten descent, which requires a density structure based on the notion of dimension. 

This means that we need to be able to measure the dimension of a complement of an open subscheme within an asymptotic manifold with corners. Such complements do not have to belong to $\ksite$, for example  the union of intersecting lines is not a manifold with corners. In this section we enlarge $\gsite$ to include {\it asymptotic spaces}, which will be such complements. 

Not all complements will be allowed though. The defining property is existence of stratification by objects of $\gsite$. Our intuition comes from cell complexes in topology. We will use asymptotic spaces only to define the density structure, i.e.\@ we need them only to count dimensions. In particular, all of the asymptotic spaces will be reduced.

%%%%%%   Definition of $C^\infty$-cell complexes and cell decompositions
\begin{definition}\label{CellComplexes} A reduced $\cscheme\in\csite$ is {\it an asymptotic space}, if there are closed embeddings 
	\begin{equation*}\cscheme=\cscheme_0\hookleftarrow\cscheme_1\hookleftarrow\ldots\hookleftarrow\cscheme_t\cong\emptys,\end{equation*} 
s.t.\@ $\forall i\geq 1$ $\cscheme_i\cong\comple{\cscheme_{i-1}}{\opensch_i}$,\footnote{The operation of subtraction is described in Def.\@ \ref{ClosedComplement}.} where $\opensch_i\subseteq\cscheme_{i-1}$ is open and belongs to $\gsite$. The data $\{\{\cscheme_i\},\,\{\opensch_i\}\}$ will be called {\it a $\cinfty$-cell decomposition of $\cscheme$}.\end{definition}
%%%%%%
The choice of a $\cinfty$-cell decomposition is not part of the structure, only existence thereof. \hide{%
Notice that each $\cscheme_i$ is automatically reduced due to our definition of the subtraction operation.\footnote{Open subschemes of reduced schemes are reduced themselves. Indeed, if inverting a function on $\opensch$ kills $\cinfty(\opensch)$, there is a function on all of $\cinfty(\cscheme)$ that is $\infty$-nilpotent. (\cite{MR86} Thm.\@ 1.4)} }% 
We do not require dimensions of $\opensch_i$'s to satisfy any inequalities. For example a disjoint union of a finite set of manifolds is an asymptotic space, and we can choose a $\cinfty$-cell decomposition for any ordering of their dimensions. However, our first goal is to define dimensions of asymptotic spaces as the maximum of the dimensions of asymptotic manifolds involved. For this we have the following lemma.

%%%%%%   Dimensions of $C^\infty$-cell complexes are well defined
\begin{lemma}\label{DimSpaWell} Let $\cscheme$ be an asymptotic space, let $\{\{\cscheme_i\},\{\opensch_i\}\}$, $\{\{\cscheme'_j\},\{\opensch'_j\}\}$ be any two $\cinfty$-cell decompositions. Then $\underset{i}\max\dim\,\opensch_i=\underset{j}\max\dim\,\opensch'_j$.\end{lemma}
%%%%%%
\begin{proof} For any $j$ consider $\opensch'_j\cap\opensch_1$, where $\cap$ is the fiber product over $\cscheme$. This is an open subscheme of $\opensch'_j$, hence $\opensch'_j\cap\opensch_1\cong\emptys$ or $\opensch'_j\cap\opensch_1$ is an asymptotic manifold with corners of dimension $\dim\,\opensch'_j$ (Lemma \ref{OpenSubcells}). In the latter case let $\fullex{\opensch'_j}$ be the maximal open extension of $\opensch'_j$ in $\cscheme$ (Prop.\@ \ref{BasicFacts}), then $\opensch'_j\cap\opensch_1=\fullex{\opensch'_j}\cap\opensch_1\cap\cscheme'_{j-1}$, and $\opensch'_j\cap\opensch_1\subseteq\opensch_1$ is locally closed. Hence $\dim\,\opensch'_j\leq\dim\,\opensch_1$ (Prop.\@ \ref{DimensionOfCell}).

In the former case $\opensch'_j\rightarrow\cscheme$ factors through $\cscheme_1$ (Lemma \ref{AvoidingCell}). Then $\opensch'_j\cap\opensch_2$ is an open subscheme of $\opensch'_j$ (here $\cap$ means fiber product over $\cscheme_1$). As before we have either $\dim\opensch'_j\leq\dim\opensch_2$, or $\opensch'_j\rightarrow\cscheme_1$ factors through $\cscheme_2$. After a finite number of steps we stop and conclude that $\dim\opensch'_j\leq\underset{i}\max\dim\,\opensch_i$.\end{proof}%

%%%%%%   Definition of dimension of a $C^\infty$-cell complex
\begin{definition}\label{DimSpaDef} Let $\cscheme$ be an asymptotic space, and let $\{\{\cscheme_i\},\{\opensch_i\}\}$ be a $\cinfty$-cell decomposition of $\cscheme$. {\it The dimension of $\cscheme$} is $\dim\,\cscheme:=\underset{i}\max\dim\,\opensch_i$.\end{definition}
%%%%%%
We would like to glue asymptotic spaces and obtain asymptotic spaces again. The following two lemmas are essential for this.

%%%%%%   Open subschemes are cell complexes too
\begin{lemma}\label{OpenSubComplexes} Let $\cscheme$ be an asymptotic space, and let $\{\cscheme_i,\opensch_i\}$ be a $\cinfty$-cell decomposition. Let $\opensch\subseteq\cscheme$ be an open subscheme, then $\{\opensch\cap\cscheme_i,\opensch\cap\opensch_i\}$ is a $\cinfty$-cell decomposition of $\opensch$, in particular $\opensch$ is an asymptotic space and $\dim\,\opensch\leq\dim\,\cscheme$.\end{lemma}
%%%%%%
\begin{proof} Clearly $\opensch_1\cap\opensch$ is an asymptotic manifold with corners (of dimension $\leq\dim\,\opensch_1$). Since $\comple{\opensch}{\opensch\cap\opensch_1}=\opensch\cap\cscheme_1$ (Prop.\@ \ref{BasicFacts}) we can continue within $\cscheme_1$. After a finite number of steps we stop.\end{proof}%

%%%%%%   Reformulation of $C^\infty$-cell decomposition as a sequence\, of open subschemes
\begin{lemma}\label{ComplexCharacterization} A reduced non-empty $\cscheme\in\csite$ is an asymptotic space, if and only if there are open subschemes
	\begin{equation}\label{OpenCharacterization}\emptys=\opensch_t\subseteq\ldots\subseteq\opensch_1\subseteq\opensch_0=\cscheme,\end{equation}
s.t.\@ $\forall  i\geq 1$ ${(\comple{\cscheme}{\opensch_i})\cap\opensch_{i-1}}$ is an asymptotic manifold with corners.\footnote{Recall that $\cap$ means fiber product over $\cscheme$.}\end{lemma}
%%%%%%
\begin{proof} Given (\ref{OpenCharacterization}) $\forall i$ we define $\cscheme_i:=\comple{\cscheme}{\opensch_{t-i}}$, and choose $\{\charfun_i\}_{i=0}^t\subseteq\cinfty(\cscheme)$ s.t.\@ $\{\opensch_i=\locin{\cscheme}{\charfun_i\neq 0}\}_{i=0}^t$. Since $\opensch_i\subseteq\opensch_{i-1}$, $\iradical{\igen{\charfun_i}}\leq\iradical{\igen{\charfun_{i-1}}}$ (Lemma \ref{OpensRadicals}), and we have closed embeddings $\emptys=\cscheme_t\hookrightarrow\ldots\hookrightarrow\cscheme_1\hookrightarrow\cscheme_0=\cscheme$. Let $\ima{\charfun_{t-i}}\in\cinfty(\cscheme_{i-1})$ be the image of $\charfun_{t-i}$, then $\cscheme_i=\comple{\cscheme_{i-1}}{\prinop{(\cscheme_{i-1})}{\ima{\charfun_{t-i}}}}$, and $\prinop{(\cscheme_{i-1})}{\ima{\charfun_{t-i}}}=\cscheme_{i-1}\cap\opensch_{t-i}=(\comple{\cscheme}{\opensch_{t-i+1}})\cap\opensch_{t-i}$ is an asymptotic manifold.

Conversely, let $\{\cscheme_i,\opensch_i\}_{i=1}^t$ be a $\cinfty$-cell decomposition. Let $\{\fullex{\opensch_i}\}_{i=1}^t$ be the maximal extensions of $\{\opensch_i\}_{i=1}^t$ to open subschemes of $\cscheme$ (Prop.\@ \ref{BasicFacts}). We have $\cscheme_i=\comple{\cscheme}{\fullex{\opensch_i}}$, in particular $\forall i$ $\fullex{\opensch_i}\subseteq\fullex{\opensch}_{i+1}$. Then $(\comple{\cscheme}{\fullex{\opensch}_{i-1}})\cap\fullex{\opensch}_{i}\cong\opensch_{i}$ (Rem.\@ \ref{DifDif}) and taking $\{\emptys\}\cup\{\fullex{\opensch}_i\}_{i=1}^t$ with the opposite order we get (\ref{OpenCharacterization}).\end{proof}%

\smallskip

As with asymptotic manifolds with corners, also on asymptotic spaces we would like to consider regular values of functions.

%%%%%%
\begin{definition} Let $\cscheme$ be an asymptotic space and let $\charfun\in\cinfty(\cscheme)$. We say that $r\in\mathbb R$ is {\it a regular value} for $\charfun$, if there is a $\cinfty$-cell decomposition $\{\{\cscheme_i\},\,\{\opensch_i\}\}$ of $\cscheme$, s.t.\@ $r$ is a regular value for restriction of $\charfun$ to each $\opensch_i$. The set of regular values will be denoted by $\regva{\charfun}$.\end{definition}
%%%%%%
As finite unions of sets of Lebesgue measure $0$ also have Lebesgue measure $0$, it is clear that almost all $r\in\mathbb R$ are regular for a given $\charfun$. However, solving equations and inequalities, given by regular values of functions, does not always give us asymptotic submanifolds (Prop.\@ \ref{ClosureRegular}, Rem.\@ \ref{ValueRegular}). Similarly for asymptotic spaces. Hence we need the following proposition.

%%%%%%
\begin{proposition}\label{RegularValues}  Let $\cscheme$ be an asymptotic space and $\charfun\in\cinfty(\cscheme)$. Let $\set\subseteq\mathbb R$ consist of $r$ s.t.\@ $\locin{\cscheme}{\charfun\leq r}=\emptys$ or is an asymptotic space of dimension $\dim\cscheme$ and $\locin{\cscheme}{\charfun=r}=\emptys$ or is an asymptotic space of dimension $\dim\cscheme-1$. Then $\mathbb R\setminus\set$ has Lebesgue measure $0$.\end{proposition}
%%%%%%
\begin{proof} Let $\{\{\cscheme_i\},\,\{\opensch_i\}\}$ be a $\cinfty$-cell decomposition of $\cscheme$, let $\{\charfun_i\}$ be the restrictions of $\charfun$ to $\{\opensch_i\}$. Let $r\in\underset{i}\bigcap\,\regva{\charfun_i}$ and consider
	\begin{equation*}\locin{\cscheme}{\charfun=r}\hookleftarrow\cscheme_1\cap\locin{\cscheme}{\charfun=r}\hookleftarrow\ldots
	\hookleftarrow\cscheme_{t-1}\cap\locin{\cscheme}{\charfun=r}\hookleftarrow\emptys.\end{equation*}
Clearly $\locin{\cscheme}{\charfun=r}\cap\cscheme_i=\comple{(\locin{\cscheme}{\charfun=r}\cap\cscheme_{i-1})}{(\locin{\cscheme}{\charfun=r}\cap\opensch_i)}$, i.e.\@ this would be a $\cinfty$-cell decomposition, if each $\locin{\cscheme}{\charfun=r}\cap\opensch_i$ was an asymptotic manifold with corners. It can be that $\locin{\cscheme}{\charfun=r}\cap\opensch_i$ is not an asymptotic manifold with corners: let $\{\manifo,\{\opensch_j\}_{j\in\diagra}\}$ be a presentation of $\opensch_i$, and let $\preimage{\charfun}\in\cinfty(\manifo)$ be a pre-image of $\charfun$, s.t.\@ $r$ is a regular value for $\charfun$ on the interior of $\manifo$ and on each piece of the boundary. Then the submanifold $\locin{\manifo}{\preimage{\charfun}=r}$ can have intersections with the boundary of $\manifo$ that are isolated. 

Since there are at most countably many such corners, excluding such values from $\underset{i}\bigcap\,\regva{\charfun_i}$ we get an $\set$ s.t.\@ $\mathbb R\setminus\set$ has Lebesgue measure $0$. We claim that $\forall r\in\set$ also $\locin{\cscheme}{\charfun\leq r}$ is an asymptotic space. We define a new sequence of closed subschemes of $\cscheme$: 
	\begin{equation*}\forall i\quad\cscheme'_i:=\cscheme_i\cap\locin{\cscheme}{\charfun\leq r},\quad
	\cscheme''_i:=(\cscheme_i\cap\locin{\cscheme}{\charfun=r})\cup\cscheme'_{i+1}.\end{equation*} 
Then $\cscheme''_i=\comple{\cscheme'_i}{(\opensch_{i+1}\cap\locin{\cscheme}{\charfun<r})}$, $\cscheme'_i=\comple{\cscheme''_{i-1}}{(\opensch_i\cap\locin{\cscheme}{\charfun=r})}$. Therefore
	\begin{equation*}\cscheme\hookleftarrow\cscheme'_0\hookleftarrow\cscheme''_0\hookleftarrow\ldots
	\hookleftarrow\cscheme'_t\hookleftarrow\cscheme''_t\end{equation*}
is a $\cinfty$-cell decomposition of $\cscheme$.\end{proof}%

%%%%%%%%%%%%%%%%%%%%%%%%%%%%%%%%%%%%%%%%%%%%%%%%%%%%%%%%%%%%%%%%%%%%%%%%%
\subsection{Density structure and Brown--Gersten descent}

As a subcategory of $\csite$, $\ksite$ inherits the Zariski topology. According to Prop.\@ \ref{PropLoc} open subschemes of an $\cscheme\in\ksite$ are themselves in $\ksite$, hence finite open covers constitute a basis for this topology on $\ksite$. Let $\zaritos{\ksite}$, $\zaritoss{\ksite}$ be the categories of pre-sheaves of sets and of simplicial sets on $\ksite$ respectively. 

Recall (e.g.\@ \cite{DSI04}) that the global projective model structure on $\zaritoss{\ksite}$ is given by objectwise weak equivalences and fibrations. The local projective model structure is given by global cofibrations and local weak equivalences, i.e.\@ morphisms that induce isomorphisms on {\it sheaves} of homotopy groups.

We would like to be able to characterize fibrant objects in the local projective model structure, which we call {\it homotopy sheaves}. Also we would like to explicitly compute homotopy pullbacks in $\zaritoss{\ksite}$. For this we use Brown--Gersten descent. To prove Brown--Gersten descent in $\zaritoss{\ksite}$ we need to introduce cd and density structures.

%%%%%%%%%%%%%%%%%%%%%%%%%%%%%%%%%%%%%%%%%%%%%%%%%%%%%%%%%%%%%%%%%%%
\subsubsection{Density and cd structure}\label{SectionDensity}

%%%%%%
\begin{definition}\label{DefSquares} A pullback square
	\begin{equation*}\xymatrix{\cscheme_1\cap\cscheme_2\ar[d]\ar[rr] && \cscheme_2\ar[d]\\
	\cscheme_1\ar[rr] && \cscheme}\end{equation*}
in $\ksite$ will be called\begin{enumerate}
\item {\it open}, if each $\cscheme_1\rightarrow\cscheme$, $\cscheme_2\rightarrow\cscheme$ is an open subscheme and $\cscheme=\cscheme_1\cup\cscheme_2$,
\item {\it closed}, if $\cscheme_1\rightarrow\cscheme$ is an open subscheme and $\cscheme_2\rightarrow\cscheme$ is an asymptotic submanifold (Def.\@ \ref{DefSubmanifolds}), s.t.\@ there is an open subscheme $\cscheme'_2\rightarrow\cscheme$ factoring through $\cscheme_2$, and $\cscheme=\cscheme_1\cup\cscheme'_2$.\end{enumerate}\end{definition}
%%%%%%
Clearly each of the two kinds of squares is stable with respect to isomorphisms of square diagrams. Hence we have {\it two cd structures} (\cite{V10} Def.\@ 2.1): the open and the closed.

%%%%%%
\begin{proposition} Both of the cd structures on $\ksite$ are complete and regular, and each generates the Zariski topology.\end{proposition}
%%%%%%
\begin{proof} The open cd structure generates Zariski topology by definition of this topology. Since closed squares can be refined by open squares, it is clear that the topology generated by closed squares is not stronger than the Zariski one. To prove the opposite inequality consider a Zariski covering: $\cscheme=\opensch_1\cup\opensch_2$. 

According to Prop.\@ \ref{ZariskiByFunction} $\exists\charfun\in\cinfty(\cscheme)$ and $r_2<r_1\in\mathbb R$, s.t.\@ $\opensch_1=\locin{\cscheme}{\charfun<r_1}$, $\opensch_2=\locin{\cscheme}{\charfun>r_2}$. According to Prop.\@ \ref{SardManifold} we can choose $r_1<r<r_1$ that is a regular value for $\charfun$. This and Prop.\@ \ref{ClosureRegular} imply that the topology generated by closed squares is not weaker than Zariski topology. 

\smallskip

According to Lemma 2.4 in \cite{V10} to prove that a cd structure is complete it is enough to show that every pullback of a distinguished square can be refined by another distinguished square. Since open squares are stable under pullbacks and we have seen that they can be refined by closed squares, we conclude that both of the cd structures are complete.

Localizations and surjective morphisms in $\fgrings$ are epimorphisms, hence distinguished squares in either of the two cd structures satisfy the conditions of Lemma 2.11 in \cite{V10}, i.e.\@ these cd structures are regular.\end{proof}%

\smallskip

Now we define the density structure on $\ksite$ (\cite{V10} Def.\@ 2.20).

%%%%%%   Definition of $m$-dense open subschemes
\begin{definition}\label{DefinitionDensity} For an $\cscheme\in\ksite$ an open subscheme $\opensch\subseteq\cscheme$ is {\it $m$-dense}, if $\comple{\cscheme}{\opensch}$ is an asymptotic space of dimension $\leq\max(-1,\dim(\cscheme)-m)$.\end{definition}
%%%%%%
For Brown--Gersten descent the most important property of a density structure is whether it is reducing for a given cd structure. Even before that we need to show that Def.\@ \ref{DefinitionDensity} does give us a density structure that is locally of finite dimension. This is the purpose of parts \ref{FirstDensity}-\ref{ThirdDensity} the following lemma.

%%%%%%
\begin{lemma}\label{DensityStructure} Let $\cscheme\in\ksite$ be of dimension $n\geq-1$. Then\begin{enumerate}
\item\label{FirstDensity} $\emptys\hookrightarrow\cscheme$ is $0$-dense; if $\opensch\hookrightarrow\cscheme$ is $m+1$-dense, it is also $m$-dense;
\item $\opensch\hookrightarrow\cscheme$ is $n+1$-dense, iff it is an isomorphism;
\item\label{ThirdDensity} if $\opensch\hookrightarrow\cscheme$ and $\opensch'\hookrightarrow\opensch$ are $m$-dense, the composite $\opensch'\hookrightarrow\cscheme$ is $m$-dense;
\item let $\opensch_1\subseteq\cscheme$ be an $m$-dense open subscheme, and let $\opensch_2\subseteq\cscheme$ be any open subscheme, then $\opensch_1\cap\opensch_2\hookrightarrow\opensch_2$ is $m$-dense;
\item if $\opensch_1,\opensch_2\hookrightarrow\cscheme$ are $m$-dense open subschemes, so is $\opensch_1\cap\opensch_2\hookrightarrow\cscheme$.\end{enumerate}\end{lemma}
%%%%%%
\begin{proof} Parts 1.\@ and 2.\@ are obvious. \hide{%
For part 2.\@ we need to show that, if $\comple{\cscheme}{\opensch}=\emptys$, then $\opensch=\cscheme$. Let $\charfun\in\cinfty(\cscheme)$ be a defining function for $\opensch$. By assumption $1\in\iradical{\igen{\charfun}}$, but then the image of $1$ in $\cinfty(\cscheme)/\igen{\charfun}$ is in $\iradical{0}$, i.e.\@ $1\in\igen{\charfun}$, and hence $\charfun$ is invertible. }%
For part 3.\@ let $\{\opensch_i\}$ be a sequence of open subschemes of $\comple{\cscheme}{\opensch}$ as in Lemma \ref{ComplexCharacterization}, and let $\{\fullex{\opensch}_i\}$ be the maximal extensions to open subschemes of $\cscheme$ (Prop.\@ \ref{BasicFacts}), in particular $\forall i$ $\opensch\subseteq\fullex{\opensch_i}$. Let $\{\opensch'_j\}$ be a sequence of open subschemes of $\comple{\opensch}{\opensch'}$ as in Lemma \ref{ComplexCharacterization}. Then $\{\{\opensch'_j\},\{\fullex{\opensch}_i\cap(\comple{\cscheme}{\opensch'})\}\}$ is a sequence of open subschemes as in Lemma \ref{ComplexCharacterization}, realizing $\comple{\cscheme}{\opensch'}$ as an asymptotic space. 

The asymptotic manifolds with corners in the $\cinfty$-cell decomposition we have constructed are the same as in decompositions of $\comple{\cscheme}{\opensch}$ and $\comple{\opensch}{\opensch'}$. Therefore, since $\dim(\cscheme)\geq\dim(\opensch)$ (Lemma \ref{OpenSubComplexes}) we finish the proof of 3.

\smallskip

To prove 4.\@ we note that $\comple{\opensch_2}{(\opensch_1\cap\opensch_2)}=\opensch_2\cap(\comple{\cscheme}{\opensch_1})$ (Prop.\@ \ref{BasicFacts}), and open subschemes of asymptotic spaces are asymptotic spaces of smaller or equal dimension (Lemma \ref{OpenSubComplexes}). Part 5.\@ follows from 3.\@ and 4.\end{proof}%

%%%%%%%%%%%%%%%%%%%%%%%%%%%%%%%%%%%%%%%%%%%%%%%%%%%%%%%%%%%%%%%%%%%
\subsubsection{Brown--Gersten descent}\label{SectionBrownGersten}

In this section we show that the density structure from Def.\@ \ref{DefinitionDensity} is reducing (\cite{V10} Def.\@ 2.22) for both of our cd structures. As in the beginning of the proof of Prop.\@ 2.10 in \cite{V10b} we start with the following lemma.

%%%%%%
\begin{lemma}\label{OnlyIntersection} Let $\cscheme\in\ksite$ and let $\cscheme_1,\cscheme_2\rightarrow\cscheme$ in $\ksite$ be asymptotic submanifolds (Def.\@ \ref{DefSubmanifolds}) or open subschemes containing an open cover $\{\opensch_1,\opensch_2\}$ of $\cscheme$. Let $\softchart_1\hookrightarrow\cscheme_1$, $\softchart_2\hookrightarrow\cscheme_2$ be $m$-dense open subschemes. There is an $m$-dense open subscheme $\opensch\subseteq\cscheme$ and an open cover $\opensch'_1,\opensch'_2\rightarrow\opensch$, s.t.\@ $\opensch'_1\subseteq\opensch_1\cap\softchart_1$, $\opensch'_2\subseteq\opensch_2\cap\softchart_2$.\end{lemma}
%%%%%%
\begin{proof} According to Prop.\@ \ref{ZariskiByFunction} we can find $\charfun\in\cinfty(\cscheme)$ and $r_2<r_1\in\mathbb R$ s.t.\@ $\opensch_1=\locin{\cscheme}{\charfun<r_1}$, $\opensch_2=\locin{\cscheme}{\charfun>r_2}$. By assumption $\comple{\cscheme_1}{\softchart_1}$, $\comple{\cscheme_2}{\softchart_2}$ are asymptotic spaces, hence we can find $r_2<r'_2<r'_1<r_1$ s.t.\@ $(\comple{\cscheme_1}{\softchart_1})\cap\locin{(\cscheme_1)}{\charfun\leq r'_1}$ and $(\comple{\cscheme_2}{\softchart_2})\cap\locin{(\cscheme_2)}{\charfun\geq r'_2}$ are asymptotic spaces as well. We define
	\begin{equation*}\opensch:=(\comple{\cscheme}{(\comple{(\cscheme_1}{\softchart_1})\cap\locin{(\cscheme_1)}{\charfun\leq r'_1}}))\cap
	(\comple{\cscheme}{(\comple{(\cscheme_1}{\softchart_2})\cap\locin{(\cscheme_2)}{\charfun\geq r'_2}})).\end{equation*}
Being an intersection of two $m$-dense open subschemes of $\cscheme$, also $\opensch\hookrightarrow\cscheme$ is $m$-dense. We define $\opensch'_1:=\locin{\opensch}{\charfun<r'_1}$, $\opensch'_2:=\locin{\opensch}{\charfun>r'_2}$.\end{proof}%

\smallskip

The previous lemma, together with the following one, shows that the density structure is reducing for the open cd structure. 

%%%%%%
\begin{lemma}\label{ReducingLemma} Consider a cartesian diagram of open subschemes in $\ksite$:
	\begin{equation*}\xymatrix{\opensch_{1}\cap\opensch_2\ar[r]\ar[d] & \opensch_2\ar[d]\\ \opensch_1\ar[r] & \cscheme.}\end{equation*}
Let $\opensch\subseteq\opensch_{1}\cap\opensch_2$ be an $m$-dense open subscheme. There is an $m+1$-dense open subscheme $\opensch'\subseteq\cscheme$ and an open cover $\opensch'_1,\opensch'_2\rightarrow\opensch'$ s.t.\@ $\opensch'_1\subseteq\opensch_1$, $\opensch'_2\subseteq\opensch_2$ and $\opensch'_1\cap\opensch'_2\subseteq\opensch$.\end{lemma}
%%%%%%
\begin{proof} According to Prop.\@ \ref{ZariskiByFunction} we can find $\charfun\in\cinfty(\cscheme)$ and $r_2<r_1\in\mathbb R$, s.t.\@ $\opensch_1=\locin{\cscheme}{\charfun_1<r_1}$, $\opensch_2=\locin{\cscheme}{\charfun_2>r_2}$. Let $\cscheme':=\comple{(\opensch_1\cap\opensch_2)}{\opensch}$. According to Prop.\@ \ref{RegularValues} there is $r\in(r_2,r_1)$, s.t.\@ $\cscheme'':=\cscheme'\cap\locin{\cscheme}{\charfun=r}$ is an asymptotic space of dimension $<\dim(\cscheme')$. Since $\cscheme'\subseteq\opensch_1\cap\opensch_2$ is a principal closed subscheme and $\locin{\cscheme}{\charfun=r}\subseteq\opensch_1\cap\opensch_2$, $\cscheme''\subseteq\cscheme$ is also principal, and clearly $\opensch':=\comple{\cscheme}{\cscheme''}\subseteq\cscheme$ is an $m+1$-dense open subscheme. We define $\opensch'_1:=\locin{\opensch'}{\charfun<r}$, $\opensch'_2:=\comple{\opensch_2}{(\cscheme'\cap\locin{\cscheme}{\charfun\leq r})}$ ($\locin{\cscheme}{\charfun\leq r}\cap\opensch_2$ is principal closed in $\opensch_2$, and hence so is $\cscheme'\cap\locin{\cscheme}{\charfun\leq r}$).\end{proof}%

%%%%%%
\begin{theorem}\label{ThmBrownGersten} The open and the closed cd structures on $\ksite$ are bounded.\end{theorem}
%%%%%%
\begin{proof} We need to show that every distinguished square
	\begin{equation*}\xymatrix{\opensch\cap\softchart\ar[r]\ar[d] & \softchart\ar[d]\\ \opensch\ar[r] & \cscheme,}\end{equation*}
where $\opensch$ is open and $\softchart$ is either open or an asymptotic submanifold, is reducing with respect to the density structure (\cite{V10} Def.\@ 2.21). According to Lemma \ref{OnlyIntersection} we can assume $\softchart\hookrightarrow\cscheme$ to be open and only need to show that, for any $m$-dense open subscheme $\opensch'\subseteq\opensch\cap\softchart$, there is an $m+1$-dense open $\cscheme'\subseteq\cscheme$ and a distinguished square on $\cscheme'$ that is contained in $\opensch$, $\softchart$ and $\opensch'$. Lemma \ref{ReducingLemma} gives us this $\cscheme'$ and a distinguished open square on it. Making one part slightly smaller, we obtain a closed distinguished square.\end{proof}%

\smallskip

Theorem \ref{ThmBrownGersten} allows us to use Brown--Gersten descent (e.g.\@ \cite{V10}): one declares a pre-sheaf of simplicial sets to be {\it flasque}, if it turns every distinguished square into a homotopy limit square. Since the usual pullback of simplicial sets is a homotopy pullback, if every corner is fibrant and at least one of the original arrows is a fibration, it is natural in our situation to switch from flasque to soft pre-sheaves.

%%%%%%
\begin{definition}\label{DefSoftSheaf} An $\sheaf\in\zaritoss{\ksite}$ is {\it a soft sheaf}, if it is a sheaf of Kan complexes and $\forall\cscheme\in\ksite$ and for any asymptotic submanifold (Def.\@ \ref{DefSubmanifolds}) $\cscheme'\hookrightarrow\cscheme$ the map $\sheaf(\cscheme)\rightarrow\sheaf(\cscheme')$ is a fibration.\end{definition}
%%%%%%

%%%%%%
\begin{theorem}\label{BrownGersten} Any soft $\sheaf\in\zaritoss{\ksite}$ is a homotopy sheaf.

Consider a diagram of soft sheaves in $\zaritoss{\ksite}$
	\begin{equation}\label{CartesianSoft}\xymatrix{& \sheaf_2\ar[d]\\ \sheaf_1\ar[r] & \sheaf}\end{equation}
s.t.\@ for any $\cscheme_1\rightarrow\cscheme_2$ in $\ksite$ the map $\sheaf_2(\cscheme_2)\rightarrow\sheaf(\cscheme)$ is a fibration. Then the usual limit of (\ref{CartesianSoft}) is also a homotopy limit in $\zaritoss{\ksite}$.\end{theorem}
%%%%%%
\begin{proof} According to \cite{Blander03} Lemma 4.1, to prove that a pre-sheaf is a homotopy sheaf it is enough to show that it takes values in Kan complexes and turns distinguished squares into homotopy pullbacks. According to Def.\@ \ref{DefSoftSheaf} soft sheaves do take values in Kan complexes, and moreover softness allows us to identify usual pullbacks with homotopy pullbacks. Finally, being sheaves they turn every distinguished square into a pullback square.

To compute a homotopy pullback of (\ref{CartesianSoft}) it is enough to substitute $\sheaf_2\rightarrow\sheaf$ with a fibrant replacement and take the usual limit. According to \cite{V10} Lemma 3.5 a local weak equivalence between soft sheaves is necessarily a global weak equivalence, hence, since limits of pre-sheaves are computed object-wise, the usual limit of (\ref{CartesianSoft}) is also a homotopy limit.\end{proof}%

%%%%%%%%%%%%%%%%%%%%%%%%%%%%%%%%%%%%%%%%%%%%%%%%%%%%%%%%%%%%%%%%%%%%%%%%%%%%%%%
\section{Groupoids and connections}
%      Definition of germs of standard simplices
%      Standard smooth simplices and smooth realization functors
%      Definition of mapping spaces and the fundamental $\infty$-groupoid of a homotopy sheaf

%      Trivial $\infty$-groupoids

%      Geometric realization of simplicial sets defines an enrichment of pre-sheaves in the category of simplicial sets
%      Geometric realization and the internal hom-object make the category of simplicial pre-sheaves into a cartesian simplicial category
%      Description of the injective local model structure, homotopy sheaves and \infty-groupoids
%      Description of the trivial groupoid of a sheaf
%      Definition of the fundamental $\infty$-groupoid of a homotopy sheaf
%      Fundamental groupoids are homotopy sheaves
%      Embedding of the trivial into the fundamental groupoids
%      Definition of k-thin \infty-groupoids
%      Thin \infty-groupoids are homotopy sheaves
%      The fundamental \infty-groupoid is a sequential colimit of thin \infty-groupoids

As we have discussed in the Introduction, when working with higher order differential operators there is no automatic flatness of connections even along curves. On the other hand, sometimes we would like to have say $2$-dimensional flatness, but not necessarily the $3$-dimensional one. 

This means we need to distinguish parts of manifolds according to their dimension. In terms of fundamental groupoids this amounts to looking at {\it thin} maps. We use $\cinfty$-realizations of the standard topological simplices to measure thinness of a morphism. To organize these simplices into an $\infty$-groupoid we need to smoothen them at the edges and corners. This occupies the first part of this section. Then we define fundamental groupoid up to a given level of thinness and finally construct the infinitesimal versions of everything by taking infinitesimal neighbourhoods of the diagonals.

%%%%%%%%%%%%%%%%%%%%%%%%%%%%%%%%%%%%%%%%%%%%%%%%%%%%%%%%%%%%%%%%%%%%%%%%%
\subsection{Groupoids}
%      Definition of germs of standard simplices
%      Standard smooth simplices and $C^\infty$-realization functors
%      Notation for germs of open boundaries
%      Contractions of germs of boundaries into boundaries
%      Reparameterization into a smooth simplex

%%%%%%%%%%%%%%%%%%%%%%%%%%%%%%%%%%%%%%%%%%%%%%%%%%%%%%%%%%%%%%%%%%%
\subsubsection{Smooth simplices}\label{SectionSimplices}
Recall (Def.\@ \ref{DefinitionOfGerm}) that given $\cscheme\in\csite$ and a subscheme $\cscheme'\subseteq\cscheme$ the ideal $\gerim{\cscheme,\cscheme'}\leq\cinfty(\cscheme)$ consists of functions that have $0$-germs at $\cscheme$.
%%%%%%      Definition of germs of standard simplices
\begin{definition}\label{SmoothSimplices} For $n\in\mathbb Z_{\geq 0}$ let $\csimplex{n}\in\ksite$ be defined as follows: 
	\begin{equation*}\csimplex{n}:=\specof{\cinfty(\mathbb R^{n+1})/\ideal},\quad
	\ideal=\igen{1-\underset{0\leq i\leq n}\sum\,x_i}+\gerim{\mathbb R^{n+1},\mathbb R_{\geq 0}^{n+1}},\end{equation*} 
i.e.\@ $\ideal$ is generated by functions that vanish on the hyper-plane $\underset{0\leq i\leq n}\sum\,x_i=1$, together with functions that have $0$-germs at $\vset{\rpt}{\mathbb R^{n+1}}{\{x_i\geq 0\}_{i=0}^n}$.\end{definition}
%%%%%%
\noindent Notice that $\csimplex{0}\cong\specof{\mathbb R}$, while for $n>0$ $\csimplex{n}$ is the germ of $\mathbb R^n$ at $\simplex{n}\subsetneq\mathbb R^n$.

%%%%%%   Standard smooth simplices and $C^\infty$-realization functors
\begin{remark} Definition \ref{SmoothSimplices} gives us a functor $\finor{n}\mapsto\csimplex{n}$ from the category of finite non-empty ordinals and weakly order preserving maps to $\ksite$. \hide{%
Indeed, for any monoid $M$ we have the canonical cosimplicial object $\{M^{\times^k}\}_{k\geq 1}$, where the structure maps are given by inserting the unit and composing adjacent elements.\footnote{In other words we view the category of finite ordinals as a non-symmetric PROP (with ordered disjoint union as the monoidal structure), and then any monoid is just a representation of this PROP. For the cosimplicial object we omit the empty ordinal.} Composing everything in each cosimplicial dimension gives us $\mu\colon\{M^{\times^k}\}_{k\geq 1}\rightarrow\{M\}$ (the constant cosimplicial diagram on $M$). 

Applying this to $(\mathbb R,+)$ and taking $\mu^{-1}(1)$ we get the cosimplicial diagram of hyper-planes. We can do the same with $(\mathbb R_{\geq 0},+)$. The sequence $\{\csimplex{n}\}_{n\geq 0}$ is the germ of the hyper-planes in $\{\mathbb R^{\times^k}\}_{k\geq 1}$ at the hyper-planes in $\{\mathbb R_{\geq 0}^{\times^k}\}_{k\geq 1}$. }% 
Composing with the Yoneda embedding $\ksite\rightarrow\zaritos{\ksite}$ we obtain {\it the $\cinfty$-realization functor} $\creali\colon\ssets\rightarrow\zaritos{\ksite}$ as a left Kan extension of $\finor{n}\mapsto\csimplex{n}$ along the canonical embedding $\finor{n}\mapsto\simplex{n}\in\ssets$.\end{remark}
%%%%%%
\noindent Notice that $\creali$ does not preserve direct products, e.g.\@ $\crealio{\simplex{1}}\times\crealio{\simplex{1}}\subset\mathbb R^2$, while $\crealio{\simplex{1}\times\simplex{1}}$ is only piece-wise smooth. To compensate for this we need to smoothen our simplices at the edges. The standard way to do it is by using collarings. First some notation: $\forall\cscheme\in\ksite$ and $\forall\fima\colon\finor{k}\hookrightarrow\finor{n}$ we write $\resti{\cscheme}{n}{\fima}$ to mean the corresponding image of $\cscheme\times\csimplex{k}$ in $\cscheme\times\csimplex{n}$.

%%%%%%
\begin{definition}\label{DefinitionCollaring} Let $\sheaf\in\zaritos{\ksite}$, $n\in\noneg$ and $\cscheme\in\ksite$. {\it A collaring} on $\ssim\colon\cscheme\times\csimplex{n}\rightarrow\sheaf$ is given by the following data: for any $k'<k\leq n$, $\fima\colon\finor{k}\hookrightarrow\finor{n}$ and $\fima'\colon\finor{k'}\hookrightarrow\finor{k}$ an isomorphism in $(\resti{\cscheme}{n}{\fima\circ\fima'})/\ksite/\cscheme$
	\begin{equation}\label{collaring}\colla_{\fima,\fima'}\colon\germof{(\resti{\cscheme}{n}{\fima})}{(\resti{\cscheme}{n}{\fima\circ\fima'})}
	\overset{\cong}\longrightarrow
	\resti{\cscheme}{n}{\fima\circ\fima'}\times\germof{\mathbb R^{k-k'}}{0},\end{equation}
where $\germof{\mathbb R^{k-k'}}{0}$ is the germ of $\mathbb R^{k-k'}$ at the origin. The data $\{\colla_{\fima,\fima'}\}$ should satisfy the following conditions\begin{enumerate}
\item {\bf invariance of $\ssim$}: we have commutative diagrams
	\begin{equation}\label{invariance}\xymatrix{\germof{(\resti{\cscheme}{n}{\fima})}{(\resti{\cscheme}{n}{\fima\circ\fima'})}
	\ar@{^(->}[r]\ar[d]_{\proco{\fima}{\fima'}} & 
	\cscheme\times\csimplex{n}\ar[r]^\ssim & \sheaf\\ 
	\resti{\cscheme}{n}{\fima\circ\fima'}\ar@{^(->}[r] & \cscheme\times\csimplex{n}\ar[ru]_\ssim, &}\end{equation}
	where $\proco{\fima}{\fima'}$ is given by the obvious projection on the r.h.s.\@ of (\ref{collaring});
\item {\bf associativity}: for any $k''<k'<k\leq n$ and $\fima''\colon\finor{k''}\hookrightarrow\finor{k'}$, we have
	\begin{equation}\label{compatibility}\xymatrix{\germof{(\resti{\cscheme}{n}{\fima})}{(\resti{\cscheme}{n}{\fima\circ\fima'\circ\fima''})}
	\ar[d]_{\colla_{\fima,\fima'}}\ar[rr]^{\colla_{\fima,\fima'\circ\fima''}}
	&& \resti{\cscheme}{n}{\fima\circ\fima'\circ\fima''}\times\germof{\mathbb R^{k-k''}}{0}\\
	\germof{(\resti{\cscheme}{n}{\fima\circ\fima'})}{(\resti{\cscheme}{n}{\fima\circ\fima'\circ\fima''})}\times\germof{\mathbb R^{k-k'}}{0}
	\ar[rru]_{\colla_{\fima\circ\fima',\fima''}\times\id}. &&}\end{equation}
\end{enumerate}\end{definition}
%%%%%%

%%%%%%
\begin{remark}\label{ReasoningCollaring} Here is the reasoning behind Def.\@ \ref{DefinitionCollaring}: $\forall\fima\colon\finor{k}\rightarrow\finor{n}$ we can choose an open neighbourhood $\opensch_{\fima}\subseteq\cscheme\times\csimplex{n}$ of $\resti{\cscheme}{n}{\fima}$, and a projection $\pi_{\fima}\colon\opensch_{\fima}\rightarrow\resti{\cscheme}{n}{\fima}$ representing $\proco{\id}{\fima}$. Let $\opensch:=\underset{\fima}\bigcup\,\opensch_\fima$, then associativity implies that, making $\opensch$ smaller if necessary, we can find an $\iota\colon\cscheme\times\sphere{n-1}\hookrightarrow\opensch$ and $\pi\colon\opensch\rightarrow\iota(\cscheme\times\sphere{n-1})$, s.t.\@ $\pi\circ\iota$ contracts an open neighbourhood of $\pi(\cscheme\times\cske{n-2}{\csimplex{n}})$ in $\iota(\cscheme\times\sphere{n-1})$ and is the identity away from this neighbourhood. \hide{%
Using $\{\pi_{\fima}\}_{\fima\colon\finor{n-1}\hookrightarrow\finor{n}}$ we can cut $\iota(\cscheme\times\sphere{n-1})$ into a union of copies of $\cscheme\times\csimplex{n-1}$. Pre-images in $\cscheme\times\bouns{\csimplex{n}}$ of these copies are disjoint. Then for all $k'<n-1$ and $\fima'\colon\finor{k'}\hookrightarrow\finor{n}$ we pre-compose with contractions of open neighbourhoods of $\{\resti{\cscheme}{n}{\fima'}\}$ in these pre-images in accordance to $\{\colla_{\fima,\fima'}\}$. Then the projections extend to all of $\opensch$. }%
 Moreover, if $\sheaf$ is represented by $\cscheme'\in\csite$, invariance of $\ssim$ with respect to contracting germs extends to invariance with respect to contracting open neighbourhoods, and $\ssim|_{\opensch}=\ssim|_{\iota(\cscheme\times\sphere{n-1})}\circ\pi$.\end{remark}
 %%%%%%
 
 %%%%%%
 \begin{definition}\label{DefSmoothFamilies} Let $\sheaf\in\zaritos{\ksite}$, $n\in\noneg$. For any $\cscheme\in\ksite$ {\it a smooth $\cscheme$-family of $n$-simplices in $\sheaf$} is any $\ssim\colon\cscheme\times\csimplex{n}\rightarrow\sheaf$ that admits a collaring (Def.\@ \ref{DefinitionCollaring}). The pre-sheaf of smooth $n$-simplices in $\sheaf$ will be denoted by $\smon{n}{\sheaf}$.\end{definition}
 %%%%%%
 Notice that a choice of collaring is not part of the definition, only existence thereof. Thus $\smon{n}{\sheaf}$ is a sub-pre-sheaf of $\homo{\zaritop}(\csimplex{n},\sheaf)$. If $\sheaf$ is a sheaf, also $\homo{\zaritop}(\csimplex{n},\sheaf)$ is a sheaf, and then $\smon{n}{\sheaf}$ is automatically a mono-pre-sheaf. To show that $\smon{n}{\sheaf}$ is in fact a sheaf, we need to prove that collarings can be glued. 
 
 %%%%%%
 \begin{proposition}\label{GluingCollarings} Let $\sheaf\in\zaritos{\ksite}$, $\cscheme\in\ksite$, $n\in\noneg$ and $\ssim\colon\cscheme\times\csimplex{n}\rightarrow\sheaf$. Given an open covering $\cscheme=\opensch_1\cup\opensch_2$, s.t.\@ $\ssim|_{\opensch_1}$, $\ssim|_{\opensch_2}$ have collarings, there is a collaring on $\ssim$, whose restrictions to $\comple{\cscheme}{\opensch_1}$, $\comple{\cscheme}{\opensch_2}$ equal restrictions of the collarings on $\ssim|_{\opensch_2}$ and $\ssim|_{\opensch_1}$ respectively.\end{proposition} 
 %%%%%%
 \begin{proof} We argue inductively on the simplicial dimension, starting with dimension $0$. The beginning of the induction and each step are based on the following: for $k<n$ let $\Phi\colon\mathbb R^n\rightarrow\mathbb R^n$ be an automorphism s.t.\@ $\Phi|_{\mathbb R^k\times 0}=\id$, where we choose $\mathbb R^n=\mathbb R^k\times\mathbb R^{n-k}$; then there is $\Theta\colon[0,1]\times\mathbb R^n\longrightarrow\mathbb R^n$, s.t.\@ $\forall t\in[0,1]$ $\Theta_t$ is an isomorphism fixing $\mathbb R^k\times 0$, $\Theta_1=\Phi$ and $\Theta_0$ is an automorphism over $\mathbb R^k$. This $\Theta$ is constructed as follows: denoting a point in $\mathbb R^n$ by $(\rpt,\rqt)\in\mathbb R^k\times\mathbb R^{n-k}$, and correspondingly $\Phi=(\Phi_k,\Phi_{n-k})$ we define
 	\begin{equation}\label{Homotopy}\Theta(t,\rpt,\rqt):=(\Phi_k(\rpt,t\rqt),\frac{1}{t}\Phi_{n-k}(\rpt,t\rqt)).\end{equation}
Since $\Phi_{n-k}$ vanishes on $\mathbb R^k\times 0$, (\ref{Homotopy}) is defined also for $t=0$, where it equals $(\rpt,(\partial_\rqt\Phi_{n-k})(\rpt,0))$. It is easy to see that (\ref{Homotopy}) is smooth and satisfies all the conditions we wanted. Moreover, if ${\rm d}\Phi|_{\mathbb R^k\times 0}$ is orientation preserving, it is clear that $\Theta|_0$ is $\cinfty$-homotopic to $\id_{\mathbb R^n}$. 

Coming back to $\ssim\colon\cscheme\times\csimplex{n}\rightarrow\sheaf$ we choose $\charfun\in\cinfty(\cscheme,[0,1])$, s.t.\@ $\locin{\cscheme}{\charfun<1}\subseteq\opensch_1$ and $\locin{\cscheme}{\charfun>0}\subseteq\opensch_2$. Starting with $k=0$ we use (\ref{Homotopy}) and $\charfun$ to construct a collaring around $\cscheme\times\cske{k}{\csimplex{n}}$, s.t.\@ restrictions of this collaring to $\locin{\cscheme}{\charfun\leq\frac{1}{3}}$, $\locin{\cscheme}{\charfun\geq\frac{2}{3}}$ equal the collarings from $\opensch_1$ and $\opensch_2$ respectively. Going from $\cscheme\times\cske{n}{\csimplex{n}}$ to $\cscheme\times\cske{k+1}{\csimplex{n}}$ is also by using (\ref{Homotopy}), since $\cske{k+1}{\csimplex{n}}$ away from $\cske{k}{\csimplex{n}}$ is a disjoint union of germs of $\mathbb R^n$ at $\mathbb R^{k+1}$.

The resulting collaring on $\ssim$ satisfies the invariance condition, since (\ref{Homotopy}) is constructed by moving along the fibers of one collaring, then the other collaring, and finally the first one again. Each of these three steps produces an automorphism of $\ssim$, and hence so does their composition.\end{proof}%

\smallskip

Proposition \ref{GluingCollarings} immediately implies that $\forall n\in\noneg$ $\smon{n}{\sheaf}$ is a sheaf, if $\sheaf$ itself is a sheaf. We would like to argue that smooth simplices abound. The following simple proposition shows that any map $\csimplex{n}\rightarrow\sheaf$ can be reparameterized into a smooth simplex.

%%%%%%   Reparameterization into a smooth simplex
\begin{proposition} For any $n\in\mathbb N$ there is a smooth $\cmor\colon\csimplex{n}\rightarrow\csimplex{n}$, s.t.\@ it is the identity morphism outside an open neighbourhood of the boundary, and it defines a surjective map between the underlying topological spaces. \end{proposition}
%%%%%%
\begin{proof} We have defined $\csimplex{n}$ as the germ of $\mathbb R^n$ at a linear embedding $\simplex{n}\subseteq\mathbb R^n$. For any $k<n$ and $\fima\colon\finor{k}\hookrightarrow\finor{n}$ we can choose coordinates in an open neighbourhood of $\fima(\csimplex{k})$ s.t.\@ all $n-1$-faces of $\csimplex{n}$ that intersect in $\fima(\csimplex{k})$ lie in coordinate hyperplanes. Moreover, we can make these choices in a compatible way, i.e.\@ $\forall k'<k$ and $\forall\fima'\colon\finor{k'}\hookrightarrow\finor{k}$ the chosen coordinates for $\fima(\csimplex{k})$ restrict to the chosen coordinates for $\fima'(\csimplex{k'})$. 

Now we use induction on dimension of the faces, and starting from $k=0$ we construct maps $\csimplex{n}\rightarrow\csimplex{n}$ that contract neighbourhoods of faces of dimensions $\leq k$ onto these faces, and are identities aways from these neighbourhoods. We define these contractions using projections on coordinates planes, hence we obtain associativity (Def.\@ \ref{DefinitionCollaring}).\end{proof}%

%%%%%%%%%%%%%%%%%%%%%%%%%%%%%%%%%%%%%%%%%%%%%%%%%%%%%%%%%%%%%%%%%%%
\subsubsection{Fundamental and thin groupoids}\label{SectionFunGro}

%%%%%%
\begin{definition}{\it The fundamental $\infty$-groupoid} of a sheaf $\sheaf\in\zaritos{\ksite}$ is 
	\begin{equation*}\fungro{\sheaf}:=\{\smon{n}{\sheaf}\}_{n\in\noneg}.\end{equation*}
\end{definition}
%%%%%%
Since we smoothen by pre-composing with contractions, the fundamental $\infty$-groupoid is functorial in $\sheaf$. To justify calling $\fungro{\sheaf}$ an $\infty$-groupoid, we need to show that smooth horns extend to smooth simplices.
 
 %%%%%%
\begin{notation}\label{Flattening} For any $0\leq k<n$, applying the $\cinfty$-realization functor, we obtain {\it the $k$-th $\cinfty$-horn} in $\zaritos{\ksite}$, that we denote by $\chorn{n}{k}\subset\csimplex{n}$.

Realizing $\csimplex{n}$ in $\mathbb R^n$ s.t.\@ the center of the $k$-face is the origin, the $k$-th face itself lies within $\mathbb R^{n-1}=\{x_k=0\}$, and the $k$-th vertex is the point $x_k=1$, $x_{\neq k}=0$ we have {\it the flattening} $\flatt{n}{k}\colon\csimplex{n}\rightarrow\csimplex{n-1}$ given by the projection $\mathbb R^n\rightarrow\mathbb R^{n-1}$. In particular we have $\flatt{n}{k}\colon\chorn{n}{k}\rightarrow\csimplex{n-1}$.\end{notation}
%%%%%%

Our definition of collarings uses germs of boundary, but to extend smooth horns we will need to approximate them by simplices. This requires extending collarings from germs to open neighbourhoods. In some cases this extension is always possible. 

%%%%%%
\begin{lemma}\label{GluingToCell} Let $\sheaf\in\zaritos{\ksite}$ s.t.\@ for any finite set $\set$ and any $\{\cscheme_\element\rightarrow\sheaf\}_{\element\in\set}$ there is $\cscheme'\in\csite$, a morphism of pre-sheaves $\cscheme'\rightarrow\sheaf$ and $\forall\element\in\set$ a factorization $\cscheme_\element\rightarrow\cscheme'\rightarrow\sheaf$. Then for any $n>k\in\noneg$, $\cscheme\in\ksite$, if the restriction of $\ssim$ to each $\cscheme\times\csimplex{n-1}$ in $\cscheme\times\chorn{n}{k}$ is smooth, there is a factorization 
	\begin{equation*}\xymatrix{\cscheme\times\chorn{n}{k}\ar[rr]^{\ssim}\ar[rd]_{\id\times\flatt{n}{k}} && \sheaf\\ 
	& \cscheme\times\csimplex{n-1}\ar[ru]^{\ssim'}, &}\end{equation*}
s.t.\@ $\ssim'$ is smooth.\end{lemma}
%%%%%%
\begin{proof} By definition $\chorn{n}{k}$ is a colimit of $\cinfty$-simplices computed within $\zaritos{\ksite}$. Since $\ssim$ factors through $\cscheme\times\chorn{n}{k}\rightarrow\cscheme'\hookrightarrow\sheaf$, with $\cscheme'\in\csite$, $\ssim$ factors through $\underset{\simplex{m}\hookrightarrow\horn{n}{k}}\colim(\cscheme\times\csimplex{m})$ with the colimit computed within $\csite$. The assumption that each $\ssim|_{\cscheme\times\csimplex{n-1}}$ has a collaring immediately implies then that $\ssim$ factors through $\id\times\flatt{n}{k}$ (gluing over germs). \hide{%
Indeed, since $\ssim$ is defined on $\cscheme\times\chorn{n}{k}$, $\forall m<n-1$ restrictions of $\ssim$ to each $\cscheme\times\csimplex{m}$ are independent of the choice of the ambient $\cscheme\times\csimplex{n-1}$. Existence of collaring on each $\ssim|_{\cscheme\times\csimplex{n-1}}$ implies that $\ssim$ factors through $\cscheme\times\chorn{n}{k}\rightarrow\underset{k<n,\csimplex{k}\hookrightarrow\chorn{n}{k}}\colim\germof{(\cscheme\times\csimplex{n-1})}{(\cscheme\times\csimplex{k})}$. This colimit is $\cscheme\times\csimplex{n-1}$. Identifying every $\cscheme\times\csimplex{n-1}$ in $\cscheme\times\chorn{n}{k}$ with $\cscheme\times\flatt{n}{k}(\csimplex{n-1})$ we have a factorization
	\begin{equation*}\xymatrix{\cscheme\times\chorn{n}{k}\ar[r]\ar[d]_= &
	\underset{k<n,\csimplex{k}\hookrightarrow\chorn{n}{k}}\colim\germof{(\cscheme\times\csimplex{n-1})}{(\cscheme\times\csimplex{k})}
	\ar[r] & \cscheme\times\csimplex{n-1}\\
	\cscheme\times\chorn{n}{k}\ar[rru]_{\id\times\flatt{n}{k}}.}\end{equation*}}%
\end{proof}%

%%%%%%
\begin{proposition}\label{PropKanComplexes} Let $\sheaf\in\zaritos{\ksite}$ be as in Lemma \ref{GluingToCell}. Then $\fungro{\sheaf}$ is a sheaf of Kan complexes.\end{proposition}
%%%%%%
\begin{proof} From Lemma \ref{GluingToCell} we immediately conclude that every $\cscheme\times\chorn{n}{k}\rightarrow\sheaf$, whose restriction on each $\cscheme\times\csimplex{n-1}$ is smooth, can be extended to a smooth $\cscheme\times\csimplex{n}\rightarrow\sheaf$ (just compose $\cscheme\times\csimplex{n-1}\rightarrow\sheaf$ with the flattening morphism).\end{proof}%

\smallskip

We would like to prove that for some $\sheaf\in\zaritos{\ksite}$ of interest to us, the pre-sheaf of simplicial sets $\fungro{\sheaf}$ is a soft sheaf (Def.\@ \ref{DefSoftSheaf}). Of course the main tool in proving such results is by using a partition of unity. The following lemma formalizes this procedure.

%%%%%%
\begin{lemma}\label{BasicExtension} Let $\cscheme\in\ksite$, $\opensch,\opensch'\subseteq\cscheme$ open subschemes, s.t.\@ $\clos{\opensch}\subseteq\opensch'$. Suppose we have smooth $\ssim\colon\cscheme\times\chorn{n}{k}\rightarrow\sheaf$, $\ssim'\colon\opensch'\times\csimplex{n}\rightarrow\sheaf$, s.t.\@
	\begin{equation*}\ssim|_{\opensch'\times\chorn{n}{k}}=\ssim'|_{\opensch'\times\chorn{n}{k}}.\end{equation*}
There is a smooth $\ssim''\colon\cscheme\times\csimplex{n}\rightarrow\sheaf$ s.t.\@ $\ssim''|_{\cscheme\times\chorn{n}{k}}=\ssim$, $\ssim''|_{\clos{\opensch}\times\csimplex{n}}=\ssim'|_{\clos{\opensch}\times\csimplex{n}}$.\end{lemma}
%%%%%%
\begin{proof} Making $\opensch'$ smaller, if necessary, we can assume that we can choose collarings on $\ssim$ and $\ssim'$ that are equal where they are both defined. Using Lemma \ref{GluingToCell} we see that $\ssim$ factors through the flattening morphism
	\begin{equation}\label{flattening}\flatt{n}{m}\colon\cscheme\times\chorn{n}{m}\hookrightarrow\cscheme\times\csimplex{n}
	\rightarrow\cscheme\times\csimplex{n-1}.\end{equation}
As in Rem.\@ \ref{ReasoningCollaring} we can find an open neighbourhood $\overline{\opensch'\times\chorn{n}{m}}$ of $\opensch'\times\chorn{n}{m}$ in $\opensch'\times\csimplex{n}$ and $\iota\colon\opensch'\times\csimplex{n-1}\hookrightarrow\overline{\opensch'\times\chorn{n}{m}}$, s.t.\@ $\iota(\opensch'\times\csimplex{n-1})\cap(\opensch'\times\chorn{n}{m})=\opensch'\times\bouns{\chorn{n}{m}}$, and the collaring on $\ssim'|_{\overline{\opensch'\times\chorn{n}{m}}}$ defines a projection
	\begin{equation*}\pi\colon\overline{\opensch'\times\chorn{n}{m}}\longrightarrow\iota(\opensch'\times\csimplex{n-1}),\end{equation*}	
s.t.\@ $\ssim'|_{\overline{\opensch'\times\chorn{n}{m}}}=\ssim'|_{\iota(\opensch'\times\csimplex{n-1})}\circ\pi$. Since $\flatt{n}{m}|_{\opensch'}$ and $\pi|_{\opensch'\times\chorn{n}{m}}$ are defined using the same collaring on $\ssim'$, we have a factorization
	\begin{equation*}\xymatrix{\opensch'\times\chorn{n}{m}\ar[rr]^{\pi}\ar[d]_{\flatt{n}{m}} && \iota(\opensch'\times\csimplex{n-1})\\
	\opensch'\times\csimplex{n-1}.\ar[rru] &&}\end{equation*}
Therefore, applying $\flatt{n}{m}$ to all of $\opensch'\times\csimplex{n}$ we have an embedding 
	\begin{equation*}\alpha\colon\opensch'\times\csimplex{n}\rightarrow\iota(\opensch'\times\csimplex{n-1})
	\hookrightarrow\opensch'\times\csimplex{n}.\end{equation*}
Using convexity of $\csimplex{n}$ we can find a homotopy $h\colon[0,1]\times\opensch'\times\csimplex{n}\rightarrow\opensch'\times\csimplex{n}$ between $\alpha$ (at $0$) and $\id_{\opensch'\times\csimplex{n}}$ (at $1$). Let $\opensch''\subseteq\cscheme$ be an open subscheme, s.t.\@ $\opensch''\cup\opensch'=\cscheme$ and $\clos{\opensch''}\cap\clos{\opensch}=\emptys$. We can choose $\charfun\in\cinfty(\cscheme,[0,1])$, s.t.\@ $\clos{\opensch''}\subseteq\locin{\cscheme}{\charfun=0}$, $\clos{\opensch}\subseteq\locin{\cscheme}{\charfun=1}$. Now we define a composite morphism
	\begin{equation*}\nu\colon\opensch'\times\csimplex{n}\longrightarrow[0,1]\times\opensch'\times\csimplex{n}\longrightarrow
	\opensch'\times\csimplex{n},\end{equation*}
where the first arrow is $\charfun\times\id\times\id$ and the second arrow is $h$. It is clear that $\nu|_{\opensch''\cap\opensch'}=\alpha|_{\opensch''\cap\opensch'}$, while $\mu|_{\clos{\opensch}}=\id$. By construction $\ssim'|_{\opensch''\cap\opensch'}\circ\nu|_{\opensch''\cap\opensch'}$ equals $(\opensch''\cap\opensch')\times\csimplex{n}\rightarrow\sheaf$ obtained from $\ssim$ using (\ref{flattening}). Therefore we have an extension of $\ssim'\circ\nu$ to $\ssim''\colon\cscheme\times\csimplex{n}\rightarrow\sheaf$ s.t.\@ $\ssim''|_{\cscheme\times\chorn{n}{m}}=\ssim$.\end{proof}%

\smallskip

In order to use Lemma \ref{BasicExtension}, we need to show that for certain $\sheaf\in\zaritos{\ksite}$ a smooth $\cscheme'$-family of $n$-simplices can be extended to a smooth $\opensch$-family where $\cscheme'\subseteq\cscheme$ is an asymptotic submanifold and $\opensch\supseteq\cscheme$ is an open neighbourhood. 

If $\sheaf$ is represented by a manifold $\manifo$, we can choose {\it a good covering}  $\manifo=\underset{\element\in\set}\bigcup\opensch_\element$, which is a locally finite open covering, s.t.\@ each non-empty intersection is isomorphic to $\mathbb R^k$, $k=\dim\manifo$. Then we can try to extend families locally, starting with the smallest intersections. Since we have plenty of objects in $\ksite$, that are not germ-determined (mostly due to the infinitesimal structure), we need to demand the good coverings to be finite. 

%%%%%%
\begin{lemma}\label{ManifoldExtension} Let $\cscheme\in\ksite$, and let $\cscheme'\subseteq\cscheme$ be an asymptotic submanifold. Any $\Phi'\colon\cscheme'\rightarrow\manifo$, where $\manifo$ is a manifold admitting a finite good covering, can be extended to $\Phi\colon\opensch\rightarrow\manifo$, where $\opensch\supseteq\cscheme'$ is an open subscheme of $\cscheme$.\end{lemma}
%%%%%%
\begin{proof} Let $\{\opensch_\element\}_{\element\in\set}$ be a finite good covering on $\manifo$, and denote $\opensch'_\element:={\Phi}^{-1}(\opensch_\element)$. Consider the germ $\germof{\cscheme}{\cscheme'}=\underset{\element\in\set}\bigcup\germof{\cscheme}{\opensch'_\element}$. For any $\set'\subseteq\set$, if $\underset{\element\in\set'}\bigcap\opensch_\element=\emptys$, also $\underset{\element\in\set'}\bigcap\germof{\cscheme}{\opensch'_\element}=\emptys$, therefore we can extend $\Phi'$ to $\germof{\cscheme}{\cscheme'}\rightarrow\manifo$. Now using the fact that $\manifo$ is finitely presentable, we see that $\Phi'$ extends to an open neighbourhood of $\cscheme'$ in $\cscheme$.\end{proof}%

%%%%%%
\begin{theorem}\label{HomotopyFun} For $\cscheme''\in\ksite$ $\fungro{\cscheme''}$ is a soft sheaf of Kan complexes in either one of the following cases:\begin{enumerate}[label={ (\alph*)}] 
\item $\cscheme''$ is a manifold $\manifo$ admitting a finite good covering,
\item $\cscheme''$ is the germ of $\manifo$ at a puncture, where $\manifo$ is a manifold admitting a finite good covering 
\end{enumerate}\end{theorem}
%%%%%%
\begin{proof}\begin{enumerate}[label={ (\alph*)}]
\item Let $n\in\mathbb N$, $0\leq m\leq n$, and suppose we are given a smooth family 
	\begin{equation*}\mu\colon(\cscheme\times\chorn{n}{m})\cup(\cscheme'\times\csimplex{n})\longrightarrow\manifo.\end{equation*}
Lemma \ref{ManifoldExtension} tells us that $\mu$ can be extended to an open neighbourhood $\opensch\subseteq\cscheme\times\csimplex{n}$ of $(\cscheme\times\chorn{n}{m})\cup(\cscheme'\times\csimplex{n})$, and then, according to Lemma \ref{BasicExtension}, we can find an extension to all of $\cscheme\times\csimplex{n}$.
\item Suppose we are given
	\begin{equation*}\mu\colon(\cscheme\times\chorn{n}{m})\cup(\cscheme'\times\csimplex{n})\longrightarrow\cscheme''
	\hookrightarrow\manifo.\end{equation*}
Using the previous part we can extend $\mu$ to 
	\begin{equation*}\mu'\colon\cscheme\times\csimplex{n}\longrightarrow\manifo.\end{equation*}
Let $\{\opensch_i\}_{i\in\diagra}$ be a regular system of open submanifolds of $\manifo$, s.t.\@ $\cscheme''\cong\underset{i\in\diagra}\bigcap\opensch_i$. We would like show that there is an open neighbourhood 
	\begin{equation*}(\cscheme\times\chorn{n}{m})\cup(\cscheme'\times\csimplex{n})\subseteq\opensch\subseteq\cscheme\times\csimplex{n},\end{equation*} 
s.t.\@ $\forall i\in\diagra$ $\mu'(\opensch)\subseteq\opensch_i$. By assumption $\cscheme''$ is locally compact, hence working with pre-images of compact pieces, we can assume that all of $\cscheme''$ is compact, i.e.\@ $\diagra$ is a sequence. Then $\{{\mu'}^{-1}(\opensch_i)\}_{i\in\diagra}$ is a regular sequence of open neighbourhoods of $(\cscheme\times\chorn{n}{m})\cup(\cscheme'\times\csimplex{n})$. Also $\cscheme\times\csimplex{n}$ is locally compact, and working on each compact piece separately, we can assume that all of $\cscheme\times\csimplex{n}$ is compact. Finally absence of $\mathbb R$-points in $\cscheme$ means that also $\cscheme\times\csimplex{n}$ cannot have $\mathbb R$-points, and we can use Lemma \ref{OpenInSoft}. Having found $\opensch$ we apply Lemma \ref{BasicExtension}.
\end{enumerate}\end{proof}%

\smallskip

Now we would like to describe the filtration of fundamental groupoids by the level of thinness that the simplices are allowed to have. 

We start with some notation. For $\sheaf\in\zaritos{\ksite}$ we have {\it the constant simplicial pre-sheaf} $\trigro{\sheaf}$, i.e.\@ $\trigro{\sheaf}$ equals $\sheaf$ in each simplicial dimension. If $\sheaf$ is a sheaf, $\trigro{\sheaf}$ is a homotopy sheaf. We will call $\trigro{\sheaf}$ {\it the trivial $\infty$-groupoid} on $\sheaf$. The unique projections $\{\csimplex{n}\rightarrow\pt\}_{n\in\noneg}$ define the obvious $\trigro{\sheaf}\hookrightarrow\fungro{\sheaf}$, whose image consists of degenerate simplices. 

We have defined $\csimplex{n}$ as the $\cinfty$-germ of $\mathbb R^n$ at the topological simplex $\simplex{n}\subseteq\mathbb R^n$. Therefore for any $\mathbb R$-point $\rpt\in\csimplex{n}$, the germ $\germ{\csimplex{n}}{\rpt}$ of $\csimplex{n}$ at $\rpt$ is isomorphic to the germ of $\mathbb R^n$ at a point. In particular for any $k\leq n$ there are many possible $\cinfty$-morphisms $\colla\colon\germ{\csimplex{n}}{\rpt}\rightarrow\germ{\csimplex{k}}{\rpt}$.

%%%%%%   Definition of k-thin \infty-groupoids
\begin{definition}\label{DefThin} Let $\cscheme\in\ksite$, $n\geq k\in\noneg$, $\sheaf\in\zaritos{\ksite}$. A smooth $\cscheme$-family $\cscheme\times\csimplex{n}\rightarrow\sheaf$ is {\it $k$-thin}, if $\forall\rpt\in\csimplex{n}$ there is a finite open covering $\cscheme=\underset{\element\in\set_\rpt}\bigcup\opensch_\element$, for each $\element\in\set_{\rpt}$ a morphism $\colla_\element\colon\opensch_{\element}\times\germ{\csimplex{n}}{\rpt}\rightarrow\opensch_\element\times\cscheme_\element$ in $\ksite/\opensch_{\element}$ with $\dim\cscheme_\element\leq k$, and a factorization
	\begin{equation*}\xymatrix{\opensch_\element\times\germ{\csimplex{n}}{\rpt}\ar@{^(->}[r]\ar[rd]_{\colla_\element} & 
	\cscheme\times\csimplex{n}\ar[r] & 
	\sheaf\\ & \opensch_\element\times\cscheme_\element.\ar[ru] &}\end{equation*}
\end{definition}
%%%%%%
It is immediately clear that for any morphism $\cscheme'\rightarrow\cscheme$ a $k$-thin $\cscheme$-family of simplices induces a $k$-thin $\cscheme'$-family of simplices, and moreover $k$-thin simplices are stable with respect to simplicial operators.
%%%%%%
\begin{definition}\label{ThinGro} Let $\sheaf\in\zaritop$, $k\in\noneg$, {\it the $k$-thin $\infty$-groupoid of $\sheaf$} is the sub-pre-sheaf $\thingro{k}{\sheaf}\subseteq\fungro{\sheaf}$, consisting of $k$-thin families of simplices.\end{definition}
%%%%%%
Since we allow arbitrary finite coverings in Def.\@ \ref{DefThin}, and our topology is Zariski, it is clear that $\thingro{k}{\sheaf}$ is a sheaf of simplicial sets, whenever $\fungro{\sheaf}$ is a sheaf. In the proof of  Lemma \ref{GluingToCell} we extended smooth families of horns to smooth families of simplices by pre-composing with the flattening morphism. Therefore we immediately have the following result.
%%%%%%
\begin{proposition} Let $\sheaf\in\zaritop$ be as in Lemma \ref{GluingToCell}, $\forall k\in\noneg$, $\forall\cscheme\in\ksite$ $\thingro{k}{\sheaf}(\cscheme)$ is a Kan complex.\end{proposition}
%%%%%%
\hide{%
\begin{proof} We need to show that there is always an extension from $\cscheme\times\chorn{n}{m}\rightarrow\sheaf$ to $\cscheme\times\csimplex{n}\rightarrow\sheaf$ that is $k$-thin. We claim that the extensions constructed in Prop.\@ \ref{GluingToCell} are $k$-thin. Indeed, by assumption the germ of $\chorn{n}{m}$ at every point factors through the germ of $\mathbb R^k$ at a point, and $\cscheme\times\chorn{n}{m}\rightarrow\sheaf$ factors through $\flatt{n}{m}\colon\cscheme\times\chorn{n}{m}\rightarrow\cscheme\times\csimplex{n-1}$. Therefore $\forall\rpt\in\csimplex{n-1}$ the morphism $\cscheme\times\germ{\csimplex{n-1}}{\rpt}\rightarrow\sheaf$ factors through $\cscheme\times\germ{\csimplex{k}}{\rpt}$. Then so does the composition $\cscheme\times\csimplex{n}\rightarrow\cscheme\times\csimplex{n-1}\rightarrow\sheaf$.\end{proof}

\smallskip}%

Also in the proof of Thm.\@ \ref{HomotopyFun} we have used pre-composition. Partially with the flattening map, and partially with projection on an $n-1$-simplex contained in an $n$-simplex. Hence we have the following corollary.
%%%%%%
\begin{theorem}\label{SoftThin} Let $\sheaf$ be as in Thm.\@ \ref{HomotopyFun}, then $\thingro{k}{\sheaf}$ is a soft sheaf $\forall k\in\noneg$.\end{theorem}
%%%%%%

Clearly $\cscheme\times\csimplex{n}\rightarrow\sheaf$ is $0$-thin, iff it factors through the projection $\cscheme\times\csimplex{n}\rightarrow\cscheme$. Therefore $\thingro{0}{\sheaf}\cong\trigro{\sheaf}$. When we let $k$ grow, we allow more and more maps $\cscheme\times\csimplex{n}\rightarrow\sheaf$, and thus we have a sequence of inclusions
	\begin{equation*}\label{ThinFiltration}\trigro{\sheaf}\cong\thingro{0}{\sheaf}\subseteq\thingro{1}{\sheaf}\subseteq\ldots\subseteq\thingro{k}{\sheaf}\subseteq
	\ldots\subseteq\fungro{\sheaf}.\end{equation*}
Moreover, since $\fungro{\sheaf}_n\cong\thingro{k}{\sheaf}_n$ for all $n\leq k$, we have $\fungro{\sheaf}=\underset{k\in\noneg}\colim\;\thingro{k}{\sheaf}$ with the colimit computed in $\zaritoss{\ksite}$.

%%%%%%%%%%%%%%%%%%%%%%%%%%%%%%%%%%%%%%%%%%%%%%%%%%%%%%%%%%%%%%%%%%%%%%%%%
\subsection{Infinitesimal connections}
%      Recalling the two de Rham constructions
%      Definition of the formal and infinitesimal k-thin \infty-groupoids
%      Infinitesimal k-thin \infty-groupoid of the germ of \mathbb R^m at a point with m\leq k is weakly equivalent to a point
%      

%%%%%%%%%%%%%%%%%%%%%%%%%%%%%%%%%%%%%%%%%%%%%%%%%%%%%%%%%%%%%%%%%%%
\subsubsection{Infinitesimal groupoids}\label{SectionInfiGroupoids}

Recall (\cite{BK}, \S 3) that there are are two functors $\redun,\redui\colon\csite\rightarrow\csite$ that send $\cring\in\fgrings$ to $\cring/\nradical{0}$, $\cring/\iradical{0}$ respectively, where $\nradical{0}$, $\iradical{0}$ are the nil- and $\infty$-radicals (\cite{BK}, \S 2). Correspondingly to every $\sheaf\in\zaritos{\csite}$ there are associated two de Rham spaces: 
	\begin{equation*}\algered{\sheaf}:=\sheaf\circ\redun,\quad\infired{\sheaf}:=\sheaf\circ\redui.\end{equation*}
According to our definition $\cscheme\in\ksite$, if and only if $\redui(\cscheme)$ is a locally compact reduced asymptotic manifolds with corners. This immediately implies that $\cscheme\in\ksite\Rightarrow\redui(\cscheme)\in\ksite$ and $\redun(\cscheme)\in\ksite$. Therefore we can define $\algered{\sheaf}$, $\infired{\sheaf}$ also for $\zaritos{\ksite}$.

Both de Rham space constructions are functorial in $\sheaf$. On the other hand there are canonical natural transformations 
	\begin{equation*}\redun\longrightarrow\id_{{\ksite}},\quad\redui\longrightarrow\id_{{\ksite}}.\end{equation*} 

%%%%%%   Definition of infinitesimal versions of thin groupoids
\begin{definition}\label{DefInfi} For any $\sheaf\in\zaritos{\ksite}$, $k\in\noneg$, we define {\it the formal $k$-thin} and {\it the infinitesimal $k$-thin groupoids of $\sheaf$} to be the respective homotopy pullbacks (computed in $\zaritoss{\ksite}$):
	\begin{equation}\label{ComputeInfi}\xymatrix{\thinfo{k}{\sheaf}\ar[r]\ar[d] & \thingro{k}{\sheaf}\ar[d] & \thinif{k}{\sheaf}\ar[r]\ar[d] & \thingro{k}{\sheaf}\ar[d]\\
	\algered{\trigro{\sheaf}}\ar[r] & \algered{\thingro{k}{\sheaf}} & \infired{\trigro{\sheaf}}\ar[r] & \infired{\thingro{k}{\sheaf}},}\end{equation}
where use use the canonical inclusion $\trigro{\sheaf}\hookrightarrow\thingro{k}{\sheaf}$ as degenerate simplices.

Using $\fungro{\sheaf}$ instead of $\thingro{k}{\sheaf}$ we obtain the {\it formal and infinitesimal fundamental groupoids} $\thinfo{\infty}{\sheaf}$, $\thinif{\infty}{\sheaf}$ respectively.\end{definition}
%%%%%%

We would like to describe sections of $\thinfo{k}{\sheaf}$ and $\thinif{k}{\sheaf}$ explicitly. This means computing the homotopy pullbacks. We have already seen that in some cases $\thingro{k}{\sheaf}$ is a soft sheaf and in particular fibrant, while $\trigro{\sheaf}$ is always fibrant, if $\sheaf$ is a sheaf. We need to see when $\algered{\thingro{k}{\sheaf}}$ and $\infired{\thingro{k}{\sheaf}}$ are fibrant.

%%%%%%
\begin{lemma}\label{SoftDeRham}, Lemma \ref{SoftDeRham} and  If $\sheaf\in\zaritoss{\ksite}$ is a soft sheaf, so are $\algered{\sheaf}$ and $\infired{\sheaf}$. \end{lemma}
%%%%%%
\begin{proof} Since reductions of Zariski open subschemes are again Zariski open subschemes, it is clear that $\reduiof{\sheaf}$ is a sheaf. The fact that $\reduiof{\sheaf}(\cscheme)$ is a Kan complex is obvious, since $\sheaf$ takes values in Kan complexes. As asymptotic submanifolds reduce to asymptotic submanifolds, the same reasoning shows that $\reduiof{\sheaf}$ is soft.\end{proof}%

%%%%%%
\begin{theorem}\label{ReductionFibration} Let $\cscheme\in\ksite$ be as in Thm.\@ \ref{HomotopyFun}, then 
	\begin{equation*}\fungro{\cscheme}(\cscheme')\longrightarrow\infired{\fungro{\cscheme}}(\cscheme'),\quad
	\fungro{\cscheme}(\cscheme')\longrightarrow\algered{\fungro{\cscheme}}(\cscheme'),\end{equation*}
	\begin{equation*}\thingro{k}{\cscheme}(\cscheme')\longrightarrow\infired{\thingro{k}{\cscheme}}(\cscheme'),\quad
	\thingro{k}{\cscheme}(\cscheme')\longrightarrow\algered{\thingro{k}{\cscheme}}(\cscheme')\end{equation*} 
are fibrations $\forall\cscheme'\in\ksite$, $\forall k\in\noneg$.\end{theorem}
%%%%%%
\begin{proof} We need to show that, given $\ssim\colon\cscheme'\times\chorn{n}{m}\rightarrow\cscheme$, $\ssim'\colon\reduiof{\cscheme'}\times\csimplex{n}\rightarrow\cscheme$, that agree on $\reduiof{\cscheme'}\times\chorn{n}{m}$, there is an extension to $\ssim''\colon\cscheme'\times\csimplex{n}\rightarrow\cscheme$. 

Since the inclusion $\reduiof{\cscheme'}\hookrightarrow\cscheme'$ corresponds to a surjective morphism $\cinfty(\cscheme')\rightarrow\cinfty(\reduiof{\cscheme'})$, we can use the finite good covering on $\manifo$ to extend $\ssim\cup\ssim'$ to $\cscheme'\times\csimplex{n}\rightarrow\cscheme$. This shows the first case.

If $\cscheme$ is the germ of $\manifo$ at a puncture, we first extend the composite of $\ssim\cup\ssim'$ with $\cscheme\hookrightarrow\manifo$ to $\cscheme'\times\csimplex{n}\rightarrow\manifo$, and then notice that the image has to be contained in $\underset{i\in\diagra}\bigcap\opensch_i$, where $\{\opensch_i\}_{i\in\diagra}$ is a regular system computing the germ of the puncture, because this is true in the reduced case.\end{proof}%

\smallskip

Using Thm.\@ \ref{ReductionFibration}, Lemma \ref{SoftDeRham} and Thm.\@ \ref{SoftThin} we can compute $\thinfo{k}{\cscheme}$, $\thinfo{\infty}{\cscheme}$, $\thinif{k}{\cscheme}$, $\thinif{\infty}{\cscheme}$ explicitly for each $k\in\noneg$, simply by computing the usual pullbacks in (\ref{ComputeInfi}). For example given $\cscheme'\in\ksite$ the simplicial set $\thinif{k}{\cscheme}(\cscheme')$ consists of $k$-thin maps $\cscheme'\times\csimplex{n}\rightarrow\cscheme$, s.t.\@ the composite $\reduiof{\cscheme'}\times\csimplex{n}\hookrightarrow\cscheme'\times\csimplex{n}\rightarrow\cscheme$ factors through the projection on $\reduiof{\cscheme'}$.

%%%%%%%%%%%%%%%%%%%%%%%%%%%%%%%%%%%%%%%%%%%%%%%%%%%%%%%%%%%%%%%%%%%
\subsubsection{Relation to pair groupoids}\label{SectionPairGro}

Now we would like to compare different groupoids and infinitesimal groupoids associated to the same $\cscheme\in\ksite$. The first statement is obvious.

%%%%%%
\begin{lemma}\label{ThinToFull} Let $\cscheme\in\ksite$ be of dimension $k\in\noneg$. Then $\forall k'\geq k$ we have an equality\footnote{This is not just an isomorphism, but an actual equality of sheaves.}
	\begin{equation*}\thingro{k'}{\cscheme}=\fungro{\cscheme}.\end{equation*}
\end{lemma}
%%%%%%
\hide{%
\begin{proof} Let $\ssim\colon\cscheme'\times\csimplex{n}\rightarrow\cscheme$ be any smooth family of $n$-simplices. We have a factorization of $\ssim$ as follows:
	\begin{equation*}\cscheme'\times\csimplex{n}\longrightarrow\cscheme'\times\cscheme\longrightarrow\cscheme,\end{equation*}
where the first arrow is given by $\ssim$ and the projection on $\cscheme'$ and the second arrow is the projection. This shows that $\ssim$ is $k$ thin.\end{proof}

\smallskip}%

In addition to the fundamental groupoids of various thinness we have the pair groupoid $\paigro{\sheaf}\in\zaritoss{\ksite}$ associated to each $\sheaf\in\zaritos{\ksite}$. In simplicial dimension $n$ $\paigro{\sheaf}$ is just $\sheaf^{\times^n}$. The corresponding infinitesimal groupoids, are obtained as homotopy pullbacks 
	\begin{equation*}\xymatrix{\fordia{\sheaf}\ar[r]\ar[d] & \paigro{\sheaf}\ar[d] & \infidia{\sheaf}\ar[r]\ar[d] & \paigro{\sheaf}\ar[d]\\
	\algered{\trigro{\sheaf}}\ar[r] & \algered{\paigro{\sheaf}} & \infired{\trigro{\sheaf}}\ar[r] & \infired{\paigro{\sheaf}}.}\end{equation*}
In the case $\sheaf$ is represented by an $\cscheme$ with a finite good covering, the same argument as in the case of thin groupoids shows that the usual pullbacks are also homotopy pullbacks. Then $\fordia{\sheaf}$ and $\infidia{\sheaf}$ are the same groupoids as in \cite{BK} Prop.\@ 6. I.e.\@ $\fordia{\sheaf}$ consists of formal neighbourhoods of the diagonal and $\infidia{\sheaf}$ in dimension $n\in\noneg$ is the colimit (computed in $\zaritos{\ksite})$ of $\cscheme'\subseteq\cscheme^{\times^n}$, s.t.\@ $\redui(\cscheme')$ is the diagonal $\Delta\subseteq\cscheme^{\times^n}$. Moreover, as it was shown in \cite{BK}, the reflection of $\infidia{\cscheme}$ in $\ksite/\cscheme$ is the groupoid $\{\germof{\cscheme^{\times^n}}{\Delta}\}_{n\in\noneg}$, i.e.\@ it consists of the germs around diagonals. 

Evaluating at the vertices of simplices we have morphisms of groupoids (natural in $\sheaf$):
	\begin{equation*}\nu_\sheaf\colon\fungro{\sheaf}\longrightarrow\paigro{\sheaf}.\end{equation*}
The following lemma is obvious.

%%%%%%
\begin{lemma}\label{ContractibleCase} For any $k\in\noneg$ $\nu_{\mathbb R^k}\colon\fungro{\mathbb R^k}\rightarrow\paigro{\mathbb R^k}$ is a trivial fibration.\end{lemma}
%%%%%%
\hide{%
\begin{proof} Let $\cscheme\in\ksite$, and consider $\ssim\colon\horn{n}{m}\rightarrow\fungro{\mathbb R^k}(\cscheme)$, $\ssim'\colon\simplex{n}\rightarrow\paigro{\mathbb R^k}(\cscheme)$, s.t.\@ evaluating $\ssim$ on the vertices of $\chorn{n}{m}$ we obtain $\ssim'$. We would like to find $\simplex{n}\rightarrow\fungro{\mathbb R^k}(\cscheme)$, that extends $\ssim$ and projects to $\ssim'$. In other words, given $\cscheme\times\chorn{n}{m}\rightarrow\mathbb R^k$ we would like to extend it to $\cscheme\times\csimplex{n}\rightarrow\mathbb R^k$. Since $\cinfty(\mathbb R^k)$ is a free $\cinfty$-ring, this is always possible. 

Having shown that $\nu_{\mathbb R^k}$ consists of fibrations, we would like to prove that the fibers are acyclic. This means that given $\ssim\colon\cscheme\times{\csimplex{n}}\rightarrow\mathbb R^k$, s.t.\@ restriction to $\cscheme\times\bouns{\csimplex{n}}$ factors through the projection on $\cscheme$, we would like to find $\cscheme\times\csimplex{n+1}\rightarrow\mathbb R^k$, which is $\ssim$ on one $n$-face and factors through the projection on $\cscheme$, when restricted to any other $n$-face. This is always possible, again due to the freeness of $\cinfty(\mathbb R^k)$.\end{proof}

\smallskip}%

As $\thinfo{\infty}{\sheaf}$, $\thinif{\infty}{\sheaf}$, $\fordia{\sheaf}$, $\infidia{\sheaf}$ are defined as homotopy pullbacks, we have the induced $\nu_\sheaf\colon\thinfo{\infty}{\sheaf}\rightarrow\fordia{\sheaf}$, $\nu_\sheaf\colon\thinif{\infty}{\sheaf}\rightarrow\infidia{\sheaf}$. For an arbitrary manifold $\manifo$ the groupoid $\fungro{\manifo}$ can be complicated, but the infinitesimal versions $\thinfo{\infty}{\manifo}$, $\thinif{\infty}{\manifo}$ depend only on local data, hence we can use Lemma \ref{ContractibleCase} and obtain the following.  

%%%%%%
\begin{proposition}\label{FunToPair} Let $\cscheme\in\ksite$ be a manifold admitting a finite good covering, then
	\begin{equation}\label{CornerEva}\nu_\sheaf\colon\thinfo{\infty}{\cscheme}\longrightarrow\fordia{\cscheme},\quad
	\nu_\sheaf\colon\thinif{\infty}{\cscheme}\longrightarrow\infidia{\cscheme}\end{equation}
are weak equivalences.\end{proposition}
%%%%%%
\begin{proof} The requirements that we have put on $\cscheme$ allow us to compute all four groupoids involved as simple pullbacks. In particular the morphisms in (\ref{CornerEva}) are over $\algered{\sheaf}$ and $\infired{\sheaf}$ respectively. Therefore it is enough to look at the morphisms between the fibers over $\algered{\sheaf}$ and $\infired{\sheaf}$. Taking pre-images of the contractible charts on $\cscheme$ we obtain a finite open covering of $\cscheme'$, on each piece of which the two maps are weak equivalences, hence the global morphism of simplicial sheaves is a weak equivalence.\end{proof}%

\smallskip

If we admit punctures in $\mathbb R^k$, the statement in Lemma \ref{ContractibleCase} ceases to be true, of course. However, by restricting to the germs around the diagonals in $\paigro{\mathbb R^k}$ we can find a section of $\nu_{\mathbb R^k}$.

%%%%%%
\begin{proposition}\label{SecOnGerms} Let $\cscheme\in\ksite$ be any asymptotic manifold. We can choose $\eta\colon\germgro{\cscheme}\rightarrow\fungro{\cscheme}$, s.t.\@ $\nu_\cscheme\circ\eta=\id_{\germgro{\cscheme}}$.\end{proposition}
%%%%%%
\begin{proof} Suppose first that $\cscheme$ is a manifold. We choose a metric on $\cscheme$ and obtain isomorphisms between germs of $0$'s in tangent spaces and germs of the manifold itself. Then any morphism $\cscheme'\rightarrow(\germgro{\cscheme})_n$ can be seen as $n+1$ morphisms into then germ of $0$ in the corresponding tangent space. Using the linear structure on the tangent space we extend it to $\cscheme'\times\csimplex{n}$.

If $\cscheme$ is the germ of a manifold $\manifo$ at a puncture, we have $\germgro{\cscheme}\hookrightarrow\germgro{\manifo}$ as well as $\fungro{\cscheme}\hookrightarrow\fungro{\manifo}$. Since our extensions to simplices never leave germs of points, it is clear that any $\eta$ as above for $\manifo$ maps $\germgro{\cscheme}\rightarrow\fungro{\cscheme}$.\end{proof}%

\smallskip

Looking at what happens if we go from $\cscheme'$ to $\reduiof{\cscheme'}$ in the proof of Prop.\@ \ref{SecOnGerms}, we see that the infinitesimal part of $\paigro{\cscheme}$ lifts to the infinitesimal part of $\fungro{\cscheme}$. Of course, if we want to conclude something about $\thinfo{\infty}{\cscheme}$ of $\thinif{\infty}{\cscheme}$, we should be careful that our constructions do compute the required homotopy pullback. Hence the assumptions in the following.

%%%%%%
\begin{theorem}\label{SecOnInf} Let $\cscheme$ be as in Thm.\@ \ref{HomotopyFun}. Then we can choose sections of
	\begin{equation*}\nu_\cscheme\colon\thinfo{\infty}{\cscheme}\longrightarrow\fordia{\cscheme},\quad
	\nu_\cscheme\colon\thinif{\infty}{\cscheme}\longrightarrow\infidia{\cscheme}.\end{equation*}
In each case any two possible choices are homotopically equivalent.\end{theorem}
%%%%%%
\begin{proof} The only part we have not explained yet is why any two choices of sections in each case should be homotopically equivalent. This would immediately follow from $\nu_{\cscheme}$ being a weak equivalence in each case. Since infinitesimal simplices (of both kinds) have to be contained in $\mathbb R^k$-charts of $\cscheme$, and we have only finitely many of these charts, it is clear that each $\nu_{\cscheme}$ consists of trivial fibrations.\end{proof}%

%%%%%%%%%%%%%%%%%%%%%%%%%%%%%%%%%%%%%%%%%%%%%%%%%%%%%%%%%%%%%%%%%%%
\subsubsection{Connections and integrability}\label{SectionIntegrability}

%%%%%%
\begin{definition}\label{DefConnections} Let $\Phi\colon\sheaf\rightarrow\sheaf'$ be a morphism in $\zaritos{\ksite}$ and let $k\in\noneg$ or $k=\infty$. {\it A $k$-flat connection on $\Phi$} is a factorization of $\Phi$ as follows
	\begin{equation*}\xymatrix{\sheaf\ar[rr]^{\Phi}\ar[d] && \sheaf'\\
	\thingro{k}{\sheaf}.\ar[rru] &&}\end{equation*}
{\it An infinitesimal $k$-flat connection on $\Phi$} is a factorization of $\Phi$ as follows
	\begin{equation*}\xymatrix{\sheaf\ar[rr]^{\Phi}\ar[d] && \sheaf'\\
	\thinif{k}{\sheaf}.\ar[rru] &&}\end{equation*}
\end{definition}
%%%%%%
Instead of $\thinif{k}{\sheaf}$ we could have used $\thinfo{k}{\sheaf}$, which contains strictly less information, in general. In that case we would use the adjective {\it formal} instead of {\it infinitesimal}. Since we are mostly interested in the infinitesimal connections, we will suppress the formal ones from now on. 

For some $\sheaf'$ it is natural to expect a $k$-connection for any morphism into $\sheaf'$, for example if $\sheaf'=\classin{k}$ for an abelian group in $\zaritos{\ksite}$. We will not address this question here, but simply assume that we can start with a given connection.

However, we would like to single out {\it integrable} infinitesimal connections. Recall that $\thinif{k}{\sheaf}$ is defined as the pullback of $\thingro{k}{\sheaf}$ over $\infired{\sheaf}\rightarrow\infired{\thingro{k}{\sheaf}}$. In particular there is the canonical embedding 
	\begin{equation}\label{Integrability}\thinif{k}{\sheaf}\longrightarrow\thingro{k}{\sheaf}.\end{equation}
%%%%%%
\begin{definition}\label{DefIntegrable} An infinitesimal $k$-flat connection $\thinif{k}{\sheaf}\rightarrow\sheaf'$ is {\it integrable}, if it factors through (\ref{Integrability}).\end{definition}
%%%%%%

In the Section \ref{SectionPairGro} we have seen that, if $\sheaf$ is representable by an asymptotic manifold admitting a finite good covering, an integrable infinitesimal $k$-connection on $\sheaf\rightarrow\sheaf'$ will factor through the groupoid $\germgro{\sheaf}$ consisting of germs around the diagonals in $\{\sheaf^{\times^n}\}_{n\in\noneg}$. Moreover, from \cite{BK} we know that such factorization is unique, if it exists. 

In the next section we assume that the infinitesimal connections we work with are always integrable, allowing us to use germs of diagonals. This is a somewhat lazy approach. Why use infinitesimal structure at all, if we have integrable connections from the start? However, we would like to argue that infinitesimal $k$-flat connections are interesting in their own right, with or without the integrability assumption. In this paper we make this assumption, for the lack of space, if not anything else. 

In general there are two questions that come to mind immediately: what is the cohomology theory governing obstructions to integrability? The corresponding deformation theory should be quite unusual, as it should describe obstructions for all orders of vanishing, not just the polynomial ones.

The second question is whether we can integrate non integrable infinitesimal connections in a wider class of functions than just the $\cinfty$-functions. Even more interesting would be to find solutions in classes that still have the $\cinfty$-ring structure.

%%%%%%%%%%%%%%%%%%%%%%%%%%%%%%%%%%%%%%%%%%%%%%%%%%%%%%%%%%%%%%%%%%%%%%%%%%%%%%%
\section{Beilinson--Drinfeld Grassmannians and factorization algebras}

%%%%%%%%%%%%%%%%%%%%%%%%%%%%%%%%%%%%%%%%%%%%%%%%%%%%%%%%%%%%%%%%%%%%%%%%%
\subsection{The problem of constructing factorizable line bundles}\label{SectionProblem}

Recall that $\msite\subset\ksite\subset\gsite$ denote the sites of the usual manifolds, the locally compact asymptotic manifolds with corners, and all asymptotic manifolds with corners, equipped with {\it Zariski topologies}. We denote by $\zaritos{\msite}$, $\zaritos{\ksite}$, $\zaritos{\gsite}$, $\zaritoss{\msite}$, $\zaritoss{\ksite}$, $\zaritoss{\gsite}$ the corresponding categories of pre-sheaves with values is $\sets$ and $\ssets$ respectively. 

\smallskip

Let $\abgro$ be an abelian group-object in $\zaritoss{\gsite}$. We assume that $\abgro$ and all of its classifying spaces $\classin{k}$, $k\in\mathbb N$ are obtained by a left Kan extension from an abelian group object over ${\msite}$, equipped with the usual (not Zariski) topology. In particular $\forall\cscheme\in\gsite$ any $\cscheme\rightarrow\classin{k}$ factors through an object in $\msite$. This assumption is automatically satisfied, if $\abgro$ is representable by an abelian finite dimensional Lie group. 

\hide{%
Indeed, for $\abgro$ itself the condition is clearly satisfied, for $\classin{k}$ we proceed as follows: using the Yoneda embedding we view $\abgro$ as an object in $\zaritoss{\gsite}$ and take the usual iterated bar construction $\barco{k}{\abgro}$ of $\abgro$, i.e.\@ each time it is a bar construction followed by taking the diagonal. The resulting $\barco{k}{\abgro}$ is a simplicial manifold, but it is not necessarily fibrant as an object in $\zaritoss{\gsite}$. 

To obtain a fibrant object with respect to either the local projective or the local injective model structure, we use the fact that both of these model structures are simplicial. Taking the homotopy colimit of this diagram of simplicial manifolds we obtain a fibrant object (in either of the model structures), that we denote by $\classin{k}$. 

By construction, $\forall n\in\noneg$ $(\classin{k})_n$ is a colimit of a diagram of sheaves representable by objects of $\msite$. Since $\cinfty$-rings of functions on manifolds are finitely presented, for any $\cscheme\in\gsite$ any $\cscheme\rightarrow(\classin{k})_n$ factors through $\cscheme\hookrightarrow\manifo$, where $\{\manifo,\{\opensch_i\}_{i\in\diagra}\}$ is a presentation of $\cscheme$ (Def.\@ \ref{SmoothCell}).

\smallskip}%

We fix a $k\geq 2$ and a $k$-dimensional $\manifold\in\msite$ admitting finite good covering. Given a group object $\smoup\in\zaritos{\ksite}$ and a multiplicative $k$-gerbe with integrable infinitesimal $k$-flat connection 
	\begin{equation*}\gerco\colon\thinfo{k}{\classi{\smoup}}\longrightarrow\thingro{k}{\smoup}\longrightarrow\classin{k+1},\end{equation*} 
{\it we would like to construct an $\abgro$-torsor on the Beilinson--Drinfeld Grassmannian corresponding to $\manifold$ and $\smoup$, that has the factorization property}.

\smallskip

First we need to explain what we mean by Beilinson--Drinfeld Grassmannians in our setting. We follow \cite{Ga13} \S0.5.3. Let $\fsets$ be the category of finite non-empty sets and surjections. For every $\set\in\fsets$ we write $\produ{\manifold}{\set}$ for the $\set$-fold direct product. Every $\pi\colon\set_1\rightarrow\set_2$ in $\fsets$ defines $\produ{\manifold}{\set_2}\rightarrow\produ{\manifold}{\set_1}$ by mapping the copy of $\manifold$ corresponding to $\element\in\set_2$ diagonally into $\produ{\manifold}{\pi^{-1}(\element)}$. In this way we obtain an $\fsets^{\op}$-diagram $\diadia{\manifold}$ in $\msite$, and using Yoneda embedding we consider it as a diagram in $\zaritos{\msite}$.

%%%%%%
\begin{definition} The Ran-space $\Ran{\manifold}\in\zaritos{\msite}$ corresponding to $\manifold$ is the colimit of $\diadia{\manifold}$ computed in $\zaritos{\msite}$.\end{definition}
%%%%%%
Given $\vmani\in\msite$, a morphism $\Phi\colon\vmani\rightarrow\Ran{\manifold}$ is an equivalence class of sets of maps $\{\Phi_i\colon\vmani\rightarrow\manifold\}_{i=1}^n$ with arbitrary $n\in\mathbb N$, where the equivalence is given by renumbering. We write $\gra{\Phi_i}$ for the graph of $\Phi_i$ in $\vmani\times\manifold$, and $\gra{\Phi}:=\underset{1\leq i\leq n}\bigcup\gra{\Phi_i}\subseteq\vmani\times\manifold$. Clearly $\gra{\Phi}$ is independent of the numbering. Now we take the germs
	\begin{equation*}\gegra{\Phi}:=\germof{(\vmani\times\manifold)}{\gra{\Phi}},\quad
	\pugra{\Phi}:=\germof{(\vmani\times\manifold\setminus\gra{\Phi})}{\gra{\Phi}}.\end{equation*}
These are objects in $\ksite$, and we define pre-sheaves on $\msite$:
	\begin{equation*}\ggrass{\manifold}{\smoup}{}(\vmani):=\underset{\Phi\in\Ran{\manifold}(\vmani)}\bigsqcup\hom{}(\gegra{\Phi},\smoup),\quad
	\pgrass{\manifold}{\smoup}{}(\vmani):=\underset{\Phi\in\Ran{\manifold}(\vmani)}\bigsqcup\hom{}(\pugra{\Phi},\smoup).\end{equation*}
The obvious projections $\ggrass{\manifold}{\smoup}{}\rightarrow\Ran{\manifold}$, $\pgrass{\manifold}{\smoup}{}\rightarrow\Ran{\manifold}$ have groups as fibers, and the former are subgroups of the latter. 

%%%%%%
\begin{definition} {\it The Beilinson--Drinfeld Grassmannian} corresponding to $\manifold$, $\smoup$ is $\grass{\manifold}{\smoup}:=\pgrass{\manifold}{\smoup}{}/\ggrass{\manifold}{\smoup}{}$.\end{definition}
%%%%%%
Recall (e.g.\@ \cite{FBZ04} \S20.3) that $\grass{\manifold}{\smoup}$ is the moduli space of principal $\smoup$-bundles on $\manifold$, trivialized away from finite sets of points. Indeed, without dividing by $\ggrass{\manifold}{\smoup}{}$ we obtain bundles on $\manifold$ glued out of trivial bundles on the germs and on the complements. Dividing leaves a chosen trivialization only on the complements.

%%%%%%%%%%%%%%%%%%%%%%%%%%%%%%%%%%%%%%%%%%%%%%%%%%%%%%%%%%%%%%%%%%%%%%%%%
\subsection{The solution}\label{LineBundle}

The strategy for constructing an $\abgro$-torsor on $\grass{\manifold}{\smoup}$ is by constructing a bundle on $\pgrass{\manifold}{\smoup}{}$, that restricts to a trivial bundle on $\ggrass{\manifold}{\smoup}{}$. We would like not just any bundle on $\pgrass{\manifold}{\smoup}{}$, but a central extension (relative to $\Ran{\manifold}$), hence we need $\pgrass{\manifold}{\classi{\smoup}}{}\in\zaritoss{\msite}$ defined as follows
	\begin{equation*}\pgrass{\manifold}{\classi{\smoup}}{}(\vmani):=
	\underset{\Phi\in\Ran{\manifold}(\vmani)}\bigsqcup\homs(\pugra{\Phi},\classi{\smoup}).
	\footnote{Recall that $\homs$ denotes the simplicial set of maps.}\end{equation*}
Fixing a $\Phi\colon\vmani\rightarrow\Ran{\manifold}$ we have a simplicial set $\classi{\smoup}(\pugra{\Phi})$. We would like to promote this simplicial set to a pre-sheaf of simplicial sets on $\ksite/\vmani$:
	\begin{equation*}\forall\cscheme\rightarrow\vmani\quad\pgrass{\manifold}{\classi{\smoup}}{\Phi}(\cscheme):=
	\homs_{\vmani}(\cscheme\underset{\vmani}{\times}\pugra{\Phi},\classi{\smoup}\times\vmani),\end{equation*}
where we use the composition $\pugra{\Phi}\hookrightarrow\vmani\times\manifold\rightarrow\vmani$ to realize $\pugra{\Phi}$ in $\ksite/\vmani$, and the subscript in $\homs_{\vmani}$ means computing the mapping space in $\zaritoss{\ksite/\vmani}$. Now we have the evaluation morphism
	\begin{equation*}\eva\colon\pugra{\Phi}\underset{\vmani}\times\pgrass{\manifold}{\classi{\smoup}}{\Phi}\longrightarrow
	\classi{\smoup}\times\vmani.\end{equation*}
Just as with $\classi{\smoup}$ we lift $\gerco$ to $\gercos{\vmani}\colon\thinfo{k}{\classi{\smoup}}\times\vmani\rightarrow\classin{k+1}\times\vmani$, and composing with the evaluation morphism we have
	\begin{equation}\label{ConEva}(\thinfo{k}{\pugra{\Phi}/\vmani})\underset{\vmani}\times\pgrass{\manifold}{\classi{\smoup}}{\Phi}\longrightarrow
	\thinfo{k}{\classi{\smoup}}\times\vmani\overset{\gerco}\longrightarrow
	\classin{k+1}\times\vmani.\end{equation}
Since $\gerco$ is integrable, the composite morphism in (\ref{ConEva}) factors through $(\thingro{k}{\pugra{\Phi}/\vmani})\underset{\vmani}\times\pgrass{\manifold}{\classi{\smoup}}{\Phi}$, and therefore (Section \ref{SectionIntegrability}) through 
	\begin{equation*}\germgro{\pugra{\Phi}}\underset{\vmani}\times\pgrass{\manifold}{\classi{\smoup}}{\Phi}\longrightarrow
	\classin{k+1}\times\vmani,\end{equation*}
where $\germgro{\pugra{\Phi}}$ is the groupoid consisting of germs around diagonals in direct powers of $\pugra{\Phi}$. We would like to rewrite this map slightly differently:
	\begin{equation}\label{EntirePuncture}\pgrass{\manifold}{\classi{\smoup}}{\Phi}\longrightarrow
	\Homo{\vmani}(\germgro{\pugra{\Phi}},\classin{k+1}\times\vmani).\footnote{Recall that $\Homo{}$ denotes the simplicial pre-sheaf of mapping spaces.}		\end{equation}
Now we are ready to perform transgression.
%%%%%%
\begin{proposition}\label{ExistenceOfTheMap} There is a morphism
	\begin{equation}\label{TheMap}\Homo{\vmani}(\germgro{\pugra{\Phi}},\classin{k+1}\times\vmani)\longrightarrow\classin{2}\times\vmani,\end{equation}
that is uniquely defined up to a unique homotopy equivalence.
\end{proposition}	
%%%%%%
\begin{proof} To perform transgression we need to integrate a connection around a cycle. To do this we would like to substitute $\pugra{\Phi}$ with an asymptotic manifold, whose {\it homotopy type over $\vmani$ is that of a disjoint union of $k-1$-spheres}. By this we mean an asymptotic manifold, that can be broken into  contractible pieces over $\vmani$ (i.e.\@ germs of manifolds contractible to $\vmani$), s.t.\@ the diagram of the pieces is that of a disjoint union of $k-1$-spheres.

We choose a metric on $\manifold$, and obtain a geodesic coordinate neighbourhood around each point. For any $\rpt\in\vmani$ we can choose a small enough open $\opensch\ni\rpt$, s.t.\@ over $\opensch$ each connected component of $\gra{\Phi}$ is contained in the geodesic neighbourhood around the image of $\rpt$. Therefore, restricting to $\opensch$ we can assume that $\vmani$ is connected and there is a chosen coordinate system around each component of $\gra{\Phi}$.

We choose a numbering of the copies of $\vmani$ in each connected component of $\gra{\Phi}$. Then, within each connected component, we connect by curves consecutive punctures in the fibers of $\manifold\times\vmani\rightarrow\vmani$.\footnote{For example to choose the shortest path from the $i$-th puncture to the previous $i-1$ punctures.} We denote the resulting subset of $\manifold\times\vmani$ by $\overline{\gra{\Phi}}$. Clearly fibers of $\overline{\gra{\Phi}}\rightarrow\vmani$ are disjoint unions of contractible topological spaces. Let $\bf S$ be the germ of $\manifold\times\vmani$ at $(\manifold\times\vmani)\setminus\overline{\gra{\Phi}}$. As $\manifold$ is $k$-dimensional, it is clear that ${\bf S}/\vmani$ has the homotopy type of a disjoint union of $k-1$-spheres.

\smallskip

The inclusion ${\bf S}\hookrightarrow\pugra{\Phi}$ gives us 
	\begin{equation*}\Homo{\vmani}(\germgro{\pugra{\Phi}},\classin{k+1}\times\vmani)\longrightarrow
	\Homo{\vmani}(\germgro{\bf S},\classin{k+1}\times\vmani).\end{equation*}
Now we use the assumption on $\abgro$, that it is a left Kan extension of a sheaf on $\msite$. This implies that we can regard ${\bf S}$ as $k$-dimensional manifold, and instead of $\germgro{\bf S}$ we have a simplicial manifold $\opensch_\bullet=\{\opensch_n\}_{n\in\noneg}$ consisting of open neighbourhoods of the diagonals in $\{{\bf S}^{\times^n}\}_{n\in\noneg}$.

For any $\cscheme\in\ksite$ an $n$-simplex in $\Homo{\vmani}(\opensch_\bullet,\classin{k+1}\times\vmani)$ over $\cscheme$ is a morphism
	\begin{equation*}\cscheme\underset{\vmani}\times\opensch_\bullet\rightarrow(\classin{k+1})^{\simplex{n}}\times\vmani,\end{equation*}
i.e.\@ it is a morphism $\cscheme\times{\bf S}\rightarrow(\classin{k+1})^{\simplex{n}}\times\vmani$ together with a chosen trivialization on $\cscheme\times\softchart\rightarrow\cscheme$ for each contractible $\softchart\subseteq{\bf S}$. Breaking ${\bf S}$ into contractible pieces over $\vmani$ and choosing one point in each piece, we obtain
	\begin{equation*}\Homo{\vmani}(\germgro{\bf S},\classin{k+1}\times\vmani)\longrightarrow
	\Homo{\vmani}(\underset{1\leq i\leq m}\bigsqcup\mathbb S^{k-1},\classin{k+1}\times\vmani)\cong
	\underset{1\leq i\leq m}\prod\classin{2}\times\vmani,\end{equation*}
where $m$ is the number of $k-1$-spheres in the disjoint union. Now, using the group structure on $\classin{2}$, we obtain altogether 
	\begin{equation*}\Homo{\vmani}(\germgro{\pugra{\Phi}},\classin{k+1}\times\vmani)\longrightarrow\classin{2}\times\vmani.\end{equation*}
We need to show that this morphism is independent of the choices that we have made, up to a homotopy equivalence. We chose four things: the numbering of copies of $\vmani$ in $\gra{\Phi}$, curves that connected punctures, decomposition of ${\bf S}$ into contractible pieces over $\vmani$, points in each of the contractible piece.

Changing the latter two choices clearly results in a unique connecting homotopy equivalence, since the trivializations are uniquely defined on each contractible piece of ${\bf S}$.

Having two different choices of the curves, connecting the punctures, and the corresponding two different ${\bf S}$, ${\bf S}'$, we can choose surfaces connecting the curves, and then take the complement ${\bf S}''$ of the resulting union of contractible pieces. Clearly ${\bf S''}\subseteq{\bf S}\cap{\bf S}'$, and the corresponding three asymptotic manifolds have the same homotopy type over $\vmani$. In this way ${\bf S}''$ provides the required homotopy equivalence. 

This homotopy equivalence is dependent on the choice of the connecting surfaces, but choosing a filling of dimension $3$, we obtain a $2$-homotopy, etc.\footnote{The ability to always choose a filling exists because every connected component of $\gra{\Phi}$ lies within an $\mathbb R^k\times\vmani$-chart on $\manifold\times\vmani$.} The process stops when we reach the dimension of $\manifold$.

\smallskip

Finally we need to show that our construction is independent of $\rpt\in\vmani$ and the neighbourhood $\opensch$ around it that we have chosen. In other words we need to show that the morphism into $\classin{2}$ over a set of distinct punctures (restricting to a subset of $\opensch$) is the sum of morphisms over individual punctures (computing the morphism over the subset itself). 

The difference between the two is in using of connecting curves. I.e.\@ integrating the connection separately around each puncture and then adding up, or connecting them all and integrating around the resulting contractible space. But adding up the results of integration around two punctures is the same as choosing a trivialization over the disjoint union of two contractible pieces, each at a different puncture. Since $\abgro$ is abelian, any such choice of trivialization produces the same result (the conjugation is trivial). Working with connected components over the entire $\opensch$ just provides one possible choice.\end{proof}%

\smallskip

If we restrict (\ref{TheMap}) to the image of $\Homo{\vmani}(\germgro{\gegra{\Phi}},\classin{k+1}\times\vmani)$, given by the inclusion $\pugra{\Phi}\hookrightarrow\gegra{\Phi}$, we clearly obtain trivial gerbes, because the entire $\gegra{\Phi}/\vmani$ is a disjoint union of contractible pieces. The factorization property is obvious from the construction.

%%%%%%%%%%%%%%%%%%%%%%%%%%%%%%%%%%%%%%%%%%%%%%%%%%%%%%%%%%%%%%%%%%%%%%%%%%%%%%%
\subsection{Taking global sections}

Constructing factorization algebras from factorizable line bundles on Beilinson--Drinfeld Grassmannians is fairly straightforward in differential geometry. The reason for this is that the fibers of $\grass{\manifold}{\smoup}\rightarrow\Ran{\manifold}$ become affine, if we allow all convenient algebras, and not just $\cinfty$-rings.

Let $\Phi\colon\vmani\rightarrow\Ran{\manifold}$ be represented by an $n$-tuple $\{\Phi_i\colon\vmani\rightarrow\manifold\}_{1\leq i\leq n}$, and let $\pugra{\Phi}\subseteq\manifold\times\vmani$ be the punctured germ around the union $\gra{\Phi}\subseteq\manifold\times\vmani$ of the graphs of $\{\Phi_i\}_{1\leq i\leq n}$. For any $\vmani'\in\msite$ the value of $\pgrass{\manifold}{\smoup}{\Phi}$ on $\vmani'\times\vmani$ is $\hom{\ksite}(\pugra{\Phi}\times\vmani',\smoup)$. By definition $\pugra{\Phi}$ is a limit of a system $\{U_\alpha\}$, where  $U_\alpha=V_\alpha\setminus\gra{\Phi}$ and $V_\alpha$ is an open neighbourhood of $\gra{\Phi}$ in $\manifold\times\vmani$. Assuming that $\smoup$ is a finite dimensional Lie group we immediately see that
	\begin{equation}\label{LimLim}\hom{\ksite}(\pugra{\Phi}\times\vmani',\smoup)\cong\underset{\alpha}\lim\;\hom{\msite}(U_\alpha\times\vmani',\smoup).\end{equation}
Suppose that each $\homo{\msite}(U_\alpha,\smoup)$ is representable, i.e.\@ it is the spectrum of some ring, then the right hand side of (\ref{LimLim}) is also representable, if our category of rings has enough limits. This is of course not true for the category of finitely generated $\cinfty$-rings. Also $\homo{\msite}(U_\alpha,\smoup)$ are not representable in this category in general. We need to consider much larger categories.

\smallskip

Let $Born$ be the category of complete bornological spaces over $\mathbb R$, and let $\mathcal R(Born)$ be the category of commutative, associative unital algebras in $Born$, considered together with the usual (projective) tensor product. As for any manifold the ring $\cinfty(U_\alpha)$ has a natural bornology, given by the canonical Fr\'echet topology on $\cinfty(U_\alpha)$, i.e.\@ we have $\cinfty(U_\alpha)\in\mathcal R(Born)$. In particular we obtain the category of quasi-coherent sheaves on $U_\alpha$ as the category of modules in $Born$ over the algebra $\cinfty(U_\alpha)$.

The natural bornologies on rings of smooth functions on finite dimensional manifolds are very special. They are the \emph{topological bornologies}, these spaces are often called \emph{convenient}. One of the properties of the subcategory $Conv\subset Born$ of convenient spaces, is that there are two tensor products present, and they agree, if one of the factors is nuclear. It so happens that the rings of smooth functions on finite dimensional manifolds are nuclear. So in our case we can substitute $\cinfty(U_\alpha)\widehat\otimes\cinfty(\vmani')$ with the injective tensor product. This tensor product preserves all limits, and therefore, using the special adjunction theorem, we conclude the following

%%%%%%
\begin{proposition}\label{ExistenceAdjoint} The functor
	\begin{equation*}\homo{}(\cinfty(\smoup),\cinfty(\pugra{\Phi})\widehat\otimes-)\colon\msite^\op\longrightarrow\sets\end{equation*}
is representable by a commutative ring ${\rm\bf co}\homo{}(\cinfty(\smoup),\cinfty(\pugra{\Phi}))$ in the category of convenient spaces over $\mathbb R$. A similar statement holds for $\gegra{\Phi}$.\end{proposition}
%%%%%%

Notice that we have three different levels: our site $\msite$ (as well as $\ksite$) consists of finitely generated $\cinfty$-rings, but we also consider convenient rings, which are not in this site, on the other hand the modules over all of our rings can be be any bornological spaces, not necessarily convenient.

\medskip

Going back to line bundles over the Beilinson--Drinfeld Grassmannians we can view $\pgrass{\manifold}{\classi\smoup}{\Phi}$ as the spectrum of a (simplicial) convenient ring $\mathfrak P_\bullet$. This ring is obtained using Prop.\@ \ref{ExistenceAdjoint} and then tensoring with $\cinfty(\vmani)$, hence it is a $\cinfty(\vmani)$-algebra. In a similar fashion let $\mathfrak G_\bullet$ be the simplicial convenient ring whose spectrum is $\ggrass{\manifold}{\classi\smoup}{\Phi}$. A morphism $\spec(\mathfrak P_\bullet)\rightarrow\classin{2}$ defines a module $\mathfrak M_\bullet$ over $\mathfrak P_\bullet$. As the two fibers over $\Phi$ are both groups, the rings $\mathfrak P_0$, $\mathfrak G_0$ have comultiplications, and since $\mathfrak M_\bullet$ is defined over the nerve of $\pgrass{\manifold}{\smoup}{}$, $\mathfrak M_0$ is a bimodule over $\mathfrak P_0$. Consider the following diagram:
	\begin{equation}\label{Comu}\mathfrak M_0\rightrightarrows\mathfrak M_0\widehat\otimes\mathfrak P_0\rightarrow\mathfrak M_0\widehat\otimes\mathfrak G_0,\end{equation}
where the first two arrows are the comultiplication and the trivial comultiplication, and the last arrow is given by the restriction $\ggrass{\manifold}{\smoup}{}\rightarrow\pgrass{\manifold}{\smoup}{}$. Let $\mathfrak K$ be the equalizer of the two composite maps in (\ref{Comu}). This is the space of global sections of the line bundle over the fiber of $\grass{\manifold}{\smoup}{}\rightarrow\Ran{\manifold}$ over $\Phi$. It carries a natural $\cinfty(\vmani)$-module structure, so altogether we obtain the following

\begin{theorem}
	Any factorizable line bundle on the Beilinson--Drinfeld Grassmannian $\grass{\manifold}{\smoup}{}$, constructed in section \ref{LineBundle}, induces a factorization algebra on $\manifold$.
\end{theorem}

\appendix

%%%%%%%%%%%%%%%%%%%%%%%%%%%%%%%%%%%%%%%%%%%%%%%%%%%%%%%%%%%%%%%%%%%%%%%%%%%%%%%
\section{Subschemes of $\cinfty$-schemes}\label{SectionBasicGeometry}
%%      Definition of open, closed, locally closed subschemes
%%      Twice open is open, union of opens is given by the sum of squares
%%      Comparing principal opens and the reductions of their complements
%%      Definition of locally closed subscheme defined by inequalities
%%      Definition by inequalities makes sense
%%      Zariski cover through a function

%%      Definition of complements and boundaries of open subschemes
%%      Example of the boundary of an open submanifold 
%%      Basic facts about intersections, unions and complements
%%      Complement of an open subscheme is indeed a complement

In this appendix we collect some definitions and statements, used throughout the paper. All these constructions are rather standard, either in differential or in algebraic geometry. The only original input that we put here is by using the $\infty$-radicals from \cite{BK}. They allow us to translate many standard algebraic-geometric results into differential geometry.

%%%%%%%%%%%%%%%%%%%%%%%%%%%%%%%%%%%%%%%%%%%%%%%%%%%%%%%%%%%%%%%%%%%%%%%%%
\subsection{Urysohn's lemma}

%%%%%%   Definition of open, closed, locally closed subschemes
\begin{definition}\label{DefinitionLocallyClosed}\begin{enumerate}\item Let $\cscheme\in\csite$, $\charfun\in\cinfty(\cscheme)$, {\it an open subscheme of $\cscheme\in\csite$ defined by $\charfun$} is $\locin{\cscheme}{\charfun\neq 0}:=\spec(\cinfty(\cscheme)\{\charfun^{-1}\})$. 
\item {\it A closed subscheme of $\cscheme$} is $\specof{\cinfty(\cscheme)/\ideal}$ for some ideal $\ideal\leq\cinfty(\cscheme)$. Given two closed subschemes $\cscheme_1,\cscheme_2\subseteq\cscheme$, defined by $\ideal_1,\ideal_2\leq\cinfty(\cscheme)$, {\it their union} is $\cscheme_1\cup\cscheme_2:=\specof{\cinfty(\cscheme)/(\ideal_1+\ideal_2)}$.
\item A morphism $\cmor\colon\cscheme'\rightarrow\cscheme$ is {\it a locally closed subscheme}, if there is an open $\opensch\subseteq\cscheme$, s.t.\@ $\cmor$ factors through a closed embedding $\cscheme\hookrightarrow\opensch$.
\item {\it The closed image of a morphism} $\cmor\colon\cscheme'\rightarrow\cscheme$ is $\closim{\cmor}{\cscheme'}:=\specof{\cinfty(\cscheme)/\kerim{\cmor}}$, where $\kerim{\cmor}$ is the kernel of $\cinfty(\cscheme)\rightarrow\cinfty(\cscheme')$.\end{enumerate}\end{definition}
%%%%%%
An open subscheme is equipped with the canonical embedding into $\cscheme$, that we will denote by $\prinop{\cscheme}{\charfun}\subseteq\cscheme$. We will write $\prinop{\cscheme}{\charfun_1}\subseteq\prinop{\cscheme}{\charfun_2}$, if $\prinop{\cscheme}{\charfun_1}\subseteq\cscheme$ factors through $\prinop{\cscheme}{\charfun_2}\subseteq\cscheme$. Similarly for closed subschemes. The closed image of an open subscheme $\opensch\subseteq\cscheme$ will be denoted by $\clos{\opensch}$.

%%%%%%   Twice open is open, union of opens is given by the sum of squares
\begin{remark} Let $\opensch\subseteq\cscheme$, $\opensch'\subseteq\opensch$ be open subschemes, the composite $\opensch'\hookrightarrow\cscheme$ is an open subscheme, i.e.\@ $\exists\charfun\in\cinfty(\cscheme)$, s.t.\@ $\opensch'\cong\prinop{\cscheme}{\charfun}$ (e.g.\@ \cite{MR86} Thm.\@ 1.4(i)).

Given $\prinop{\cscheme}{\charfun_1},\prinop{\cscheme}{\charfun_2}\subseteq\cscheme$ and an open subscheme $\opensch\subseteq\cscheme$ the set $\{\prinop{\cscheme}{\charfun_1},\prinop{\cscheme}{\charfun_2}\}$ is a Zariski cover of $\opensch$, iff $\opensch\cong\prinop{\cscheme}{\charfun_1}\cup\prinop{\cscheme}{\charfun_2}:=\prinop{\cscheme}{\charfun_1^2+\charfun_2^2}$ (e.g.\@ \cite{MR91} Lemma VI.1.2).\end{remark}

%%%%%%   Comparing principal opens and the reductions of their complements
\begin{lemma}\label{OpensRadicals} Let $\cscheme\in\csite$, and let $\charfun_1,\charfun_2\in\cinfty(\cscheme)$. Then
	\begin{equation*}\prinop{\cscheme}{\charfun_1}\subseteq\prinop{\cscheme}{\charfun_2}\Leftrightarrow
	\iradical{\igen{\charfun_1}}\leq\iradical{\igen{\charfun_2}}.\footnote{The $\infty$-radical $\iradical{\quad}$ was defined in \cite{BK} Def.\@ 1.}\end{equation*}
\end{lemma}
%%%%%%
\begin{proof}\begin{itemize}\item[$\Rightarrow$] Consider $\cmorf\colon\cinfty(\cscheme)\rightarrow\cinfty(\cscheme)/\igen{\charfun_2}$, and let $\ima{\charfun_1}:=\cmorf(\charfun_1)$. The $\cinfty$-ring $(\cinfty(\cscheme)/\igen{\charfun_2})\{\ima{\charfun_1}^{-1}\}$ is a solution to two universal problems: inverting $\charfun_1\in\cinfty(\cscheme)$ and killing $\igen{\charfun_2}\leq\cinfty(\cscheme)$. Therefore
	\begin{equation}\label{TwoVersions}(\cinfty(\cscheme)/\igen{\charfun_2})\{\ima{\charfun_1}^{-1}\}\cong
	(\cinfty(\cscheme)\{\charfun_1^{-1}\})/\igen{\ima{\charfun_2}},\end{equation}
where $\ima{\charfun_2}\in\cinfty(\cscheme)\{\charfun_1^{-1}\}$ is the image of $\charfun_2$. By assumption $\charfun_2$ becomes invertible in $\cinfty(\cscheme)\{\charfun_1^{-1}\}$, therefore the $\cinfty$-rings in (\ref{TwoVersions}) are trivial, i.e.\@ $\ima{\charfun_1}\in\iradical{0}\leq\cinfty(\cscheme)/\igen{\charfun_2}$, implying that $\charfun_1\in\iradical{\igen{\charfun_2}}$ (\cite{BK} Prop.\@ 3). Then $\igen{\charfun_1}\leq\iradical{\igen{\charfun_2}}$ and hence $\iradical{\igen{\charfun_1}}\leq\iradical{\igen{\charfun_2}}$ (\cite{BK} Lemmas 2,4).
\item[$\Leftarrow$] Since $\charfun_1\in\iradical{\igen{\charfun_2}}$, the ideal of $\cinfty(\cscheme)\{\charfun_1^{-1}\}$ generated by the image of $\charfun_2$ contains $1$ (see (\ref{TwoVersions})), i.e.\@ $\prinop{\cscheme}{\charfun_1}\subseteq\prinop{\cscheme}{\charfun_2}$\end{itemize}\end{proof}%

%%%%%%   Definition of locally closed subscheme defined by inequalities
\begin{definition}\label{Inequalities} Let $\cscheme\in\csite$, $\charfun\in\cinfty(\cscheme)$. For $r\in\mathbb R$ let $\alpha_r\in\cinfty(\mathbb R)$ be any function s.t.\@ $\vset{\rpt}{\mathbb R}{\alpha_r=0}=(-\infty,r]$. Then we define 
	\begin{equation*}\locin{\cscheme}{\charfun\leq r}:=\spec(\cinfty(\cscheme)/\iradical{\igen{\alpha_r(\charfun)}}),\quad
	\locin{\cscheme}{\charfun>r}:=\spec(\cinfty(\cscheme)\{\alpha_r(\charfun)^{-1}\}),\end{equation*}
similarly we define $\locin{\cscheme}{\charfun\geq r}$, $\locin{\cscheme}{\charfun<r}$, and, taking intersections, $\locin{\cscheme}{r_1\leq\charfun\leq r_2}$, etc.
\end{definition}
%%%%%%

%%%%%%   Definition by inequalities makes sense
\begin{lemma}\label{InequalitiesWellDefined} Let $\cscheme$, $\charfun$ be as in Def.\@ \ref{Inequalities}. Then $\locin{\cscheme}{\charfun\leq r}$, $\locin{\cscheme}{\charfun>r}$ do not depend on the choice of $\alpha_r$.\end{lemma}
%%%%%%
\begin{proof} Choosing generators we find a surjective $\cinfty(\mathbb R^\set)\rightarrow\cinfty(\cscheme)$ for some set $\set$. Let $\preimage{\charfun}\in\cinfty(\mathbb R^\set)$ be any pre-image of $\charfun$. Let $\alpha_r$, $\alpha'_r\in\cinfty(\mathbb R)$ be any two choices as in Def.\@ \ref{Inequalities}. Since $\alpha_r$, $\alpha'_r$ have the same $0$-sets in $\mathbb R$, also $\alpha_r(\preimage{\charfun})$, $\alpha'_r(\preimage{\charfun})$ have the same $0$-sets. Then $\cinfty(\mathbb R^\set)\{\alpha_r(\preimage{\charfun})^{-1}\}=\cinfty(\mathbb R^\set)\{\alpha'_r(\preimage{\charfun})^{-1}\}$. Using \cite{MR86} Prop.\@ 1.2 we immediately conclude that $\prinop{\cscheme}{\alpha_r(\charfun)}=\prinop{\cscheme}{\alpha'_r(\charfun)}$. This implies that also $\locin{\cscheme}{\charfun\leq r}$ is well defined (Lemma \ref{OpensRadicals}).\end{proof}%

%%%%%%   Zariski cover through a function
\begin{proposition}\label{ZariskiByFunction} Let $\cscheme\in\csite$, $\cscheme\ncong\emptys$, and let $\opensch_1,\opensch_2\subseteq\cscheme$ be open subschemes, s.t.\@ $\cscheme=\opensch_1\cup\opensch_2$. Then there are $\charfun\in\cinfty(\cscheme)$, $r_2< r_1\in\mathbb R$ s.t.\@
	\begin{equation*}\opensch_1=\locin{\cscheme}{\charfun<r_1},\quad\opensch_2=\locin{\cscheme}{\charfun>r_2}.\end{equation*}
\end{proposition}
\begin{proof} According to \cite{MR91} Lemma VI.1.2 we can find $n\in\noneg$, an open set $\openset\subseteq\mathbb R^n$, a surjective morphism $\cinfty(\openset)\rightarrow\cinfty(\cscheme)$, and $\charfun_1,\charfun_2\in\cinfty(\openset)$, s.t.\@ $1\in\igen{\charfun_1,\charfun_2}$ and $\opensch_1=\prinop{\cscheme}{\ima{\charfun_1}}$, $\opensch_2=\prinop{\cscheme}{\ima{\charfun_2}}$, where $\ima{\charfun_1},\ima{\charfun_2}\in\cinfty(\cscheme)$ are the images of $\charfun_1$, $\charfun_2$. Let $\openset_1$, $\openset_2\subseteq\openset$ be the open subschemes defined by $\charfun_1$, $\charfun_2$. Then $\openset_1\cup\openset_2=\openset$, $(\openset\setminus\openset_1)\cap(\openset\setminus\openset_2)=\emptyset$. Using the smooth version of Urysohn's lemma we can find $\charfun\colon\openset\rightarrow[0,1]$, s.t.\@ $\openset\setminus\openset_1=\charfun^{-1}(\{1\})$, $\openset\setminus\openset_2=\charfun^{-1}(\{0\})$. Put $r_2:=0$, $r_1:=1$ and take the image of $\charfun$ in $\cinfty(\cscheme)$.\end{proof}%

%%%%%%%%%%%%%%%%%%%%%%%%%%%%%%%%%%%%%%%%%%%%%%%%%%%%%%%%%%%%%%%%%%%%%%%%%
\subsection{Complements and germs}
%%      Definition of complements and boundaries of open subschemes
%%      Example of the boundary of an open submanifold 
%%      Definition of a principal closed subscheme
%%      Basic facts about principal closed subschemes
%%      Difference of a difference
%%      Complement of an open subscheme is indeed a complement 
%%      Definition of a principal locally closed subscheme
%%      Definition of the germ at the image of a morphism
%%      Comparing two definitions of germs
\hide{%
%%%%%%
\begin{lemma} Let $\cscheme'\in\csite$ and $\cscheme\subseteq\cscheme'$ a closed subscheme (Def.\@ \ref{DefinitionLocallyClosed}). For any open subschemes $\opensch_1\subseteq\opensch_2\subseteq\cscheme$, s.t.\@ $\clos{\opensch_1}\subseteq\opensch_2$, there are open subschemes $\opensch'_1,\opensch'_2\subseteq\cscheme'$, s.t.\@ $\opensch_1=\opensch'_1\cap\cscheme$, $\opensch_2=\opensch'_2\cap\cscheme$ and $\clos{\opensch'_1}\subseteq\opensch'_2$.\end{lemma}
%%%%%%
\begin{proof} Let $\ideal_1:=\kernel{\cinfty(\cscheme)\rightarrow\cinfty(\cscheme)\{\charfun_1^{-1}\}}$, since $\clos{\opensch_1}\subseteq\opensch_2$ there are $\charfun'_2\in\cinfty(\cscheme)$, $\alpha\in\ideal_1$, s.t.\@ $\charfun'_2\charfun_2=1+\alpha$. As $\cinfty(\cscheme)\in\fgrings$, we can find a surjective morphism of $\cinfty$-rings $\cinfty(\mathbb R^n)\rightarrow\cinfty(\cscheme)$ for some $n\in\noneg$, let $\ideal\leq\cinfty(\mathbb R^n)$ be the kernel of this morphism. Let $\preimage{\charfun_2},\preimage{\charfun'_2},\preimage{\alpha}\in\cinfty(\mathbb R^n)$ be any pre-images of $\charfun_2$, $\charfun'_2$, $\alpha$. Then $\exists\beta\in\ideal$ s.t.\@ $\preimage{\charfun_2}\preimage{\charfun'_2}=1+\preimage{\alpha}+\beta$. Let $\closedset$ be the $0$-set of $\beta^2+\preimage{\alpha}^2$, and let $\closedset_2$ be the $0$-set of $\preimage{\charfun_2}$. Let $\charfun'\in\cinfty(\mathbb R^n,[0,1])$ have value $0$ on $\closedset$ and value $1$ on $\closedset_2$, and let $\charfun\in\cinfty(\cscheme)$ be the image of $\charfun'$. 

Then inverting $\charfun_1$ implies dividing by $\igen{\charfun}$, and dividing by $\igen{\charfun_2}$ implies inverting $\charfun$. 

Let $\charfun_1,\charfun_2\in\cinfty(\cscheme)$ be s.t.\@ $\opensch_1=\prinop{\cscheme}{\charfun_1}$, $\opensch_2=\prinop{\cscheme}{\charfun_2}$. Let $\preimage{\charfun_1}, \preimage{\charfun_2}\in\cinfty(\mathbb R^n)$ be any of their pre-images, and let $\openset_1:=\{\preimage{\charfun_1}\neq 0\}$, $\openset_2:=\{\preimage{\charfun_2}\neq 0\}$ be the corresponding open subsets in $\mathbb R^n$. Let $\closedset$ be the zero-set of $\ideal:=\kernel{\cinfty(\mathbb R^n)\rightarrow\cinfty(\cscheme)}$, then $\forall\rpt\in\clos{\openset_1}\setminus\closedset$ there is $\charfun_{\rpt}\in\ideal$, s.t.\@ $\charfun_{\rpt}(\rpt)\neq 0$. We can choose these $\{\charfun_{\rpt}\}$ to be locally finite on $\mathbb R^n$, s.t.\@ their supports cover $\clos{\openset_1}\setminus\closedset$. We define $\charfun'_2:=\preimage{\charfun_2}^2+\underset{\rpt}\sum\,\charfun_{\rpt}^2$. Let $\charfun''\in\cinfty(\mathbb R^n)$ be s.t.\@ it vanishes exactly on $\clos{\openset_1}$, then $\igen{\charfun'_2,\charfun''}=\cinfty(\mathbb R^n)$. The image of $\charfun''$ in $\cinfty(\cscheme)$ defines an open subscheme $\opensch$, s.t.\@ $\opensch\cup\opensch_2=\cscheme'$. Using Prop.\@ \ref{ZariskiByFunction} we can find $\charfun\in\cinfty(\cscheme)$, s.t.\@ $\opensch=\locin{\cscheme}{\charfun< r_1}$, $\opensch_2=\locin{\cscheme}{\charfun> r_2}$ for some $r_2< r_1\in\mathbb R$. Let $\preimage{\charfun}\in\cinfty(\cscheme')$ be any pre-image of $\charfun$. We define
	\begin{equation*}\opensch'_1:=\locin{\cscheme'}{\preimage{\charfun}<r_1},\quad\opensch'_2:=\locin{\cscheme'}{\preimage{\charfun}>r_2}.\end{equation*}
\end{proof}}%

Lemma \ref{OpensRadicals} implies that the following definition makes sense.

%%%%%%   Definition of complements and boundaries of open subschemes
\begin{definition}\label{ClosedComplement} Let $\cscheme\in\csite$, $\charfun\in\cinfty(\cscheme)$. {\it The complement $\comple{\cscheme}{\prinop{\cscheme}{\charfun}}$ of $\prinop{\cscheme}{\charfun}$} is $\spec(\cinfty(\cscheme)/\iradical{\igen{\charfun}})$. {\it The boundary of $\prinop{\cscheme}{\charfun}$ in $\cscheme$} is $\bouns{\prinop{\cscheme}{\charfun}}:=\comple{\clos{\prinop{\cscheme}{\charfun}}}{\prinop{\cscheme}{\charfun}}$.\end{definition}
%%%%%%

%%%%%%   Example of the boundary of an open submanifold 
\begin{example} Let $\opensch\cong\spec(\cinfty(\manifo)\{\charfun^{-1}\})$ be an open subscheme of a manifold $\manifo$, s.t.\@ $0$ is a regular value of $\charfun\in\cinfty(\manifo)$, then $\bouns{\opensch}$ is the submanifold $\vset{\rpt}{\manifo}{\charfun=0}$. Indeed, regularity of $0$ implies that $\igen{\charfun}\leq\cinfty(\manifo)$ is $\mathbb R$-Jacobson radical (\cite{BK} Def.\@ 1). \hide{%
Indeed, let $\charfun'\in\cinfty(\manifo)$ be any function that vanishes at every point of $\vset{\rpt}{\manifo}{\charfun=0}$. We can choose an open cover of an open neighbourhood of this closed subset, s.t.\@ in each element of the cover the restriction of $\charfun$ can be completed to a coordinate system, and hence the restriction of $\charfun'$ is a multiple of the restriction of $\charfun$. Choosing a partition of unity, subordinate to this open cover, we realize $\charfun'$ as a multiple of $\charfun$. }%
In particular $\iradical{\igen{\charfun}}=\igen{\charfun}$ (\cite{BK} Prop.\@ 1).\end{example}
%%%%%%

%%%%%%   Definition of a principal closed subscheme
\begin{definition} Let $\cscheme'\hookrightarrow\cscheme$ be a closed subscheme (i.e.\@ $\cmorf\colon\cinfty(\cscheme)\rightarrow\cinfty(\cscheme')$ is surjective), it is {\it a principal closed subscheme}, if $\exists\charfun\in\cinfty(\cscheme)$, s.t.\@ 
	\begin{equation}\label{PrincipalClosed}\iradical{\kernel{\cmorf}}=\iradical{\igen{\charfun}}.\end{equation}\end{definition}
%%%%%%
\noindent We do not require $\cscheme'$ to be reduced (the radical is on both sides of (\ref{PrincipalClosed})).

%%%%%%   Basic facts about principal closed subschemes
\begin{proposition}\label{BasicFacts} Let $\cscheme\in\csite$.\begin{enumerate}
\item Let $\cscheme'\subseteq\cscheme$ be a principal closed closed subscheme. For any open $\opensch\subseteq\cscheme'$ there is an open $\fullex{\opensch}\subseteq\cscheme$, s.t.\@ $\fullex{\opensch}\cap\cscheme'=\opensch$ and $\forall\,\opensch'\subseteq\cscheme$ with this property we have $\opensch'\subseteq\fullex{\opensch}$. If $\opensch\cong\emptys$, we will write $\comple{\cscheme}{\cscheme'}:=\fullex{\emptys}$.
\item Let $\opensch_1,\opensch_2\subseteq\cscheme$ be open subschemes, then $\comple{\opensch_1}{(\opensch_1\cap\opensch_2)}=\opensch_1\cap(\comple{\cscheme}{\opensch_2})$, where $\cap$ means the fiber product over $\cscheme$.\end{enumerate}\end{proposition}
%%%%%%
\hide{%
\begin{proof}\begin{enumerate}
\item Since $\cinfty(\cscheme)\rightarrow\cinfty(\cscheme')$ is surjective, we can choose an $\charfun'\in\cinfty(\cscheme)$ s.t.\@ $\prinop{\cscheme}{\charfun'}\cap\cscheme'=\opensch$. We define $\fullex{\opensch}:=\prinop{\cscheme}{\charfun'}\cup\prinop{\cscheme}{\charfun}$. Let $\charfun''\in\cinfty(\cscheme)$ be any element, s.t.\@ $\prinop{\cscheme}{\charfun''}\cap\cscheme'=\opensch$, we claim $\prinop{\cscheme}{\charfun''}\subseteq\fullex{\opensch}$. This means that inverting $\charfun''$ implies inverting ${\charfun'}^2+\charfun^2$. Inverting $\charfun''$ is equivalent to pulling everything to $\prinop{\cscheme}{\charfun''}$, i.e.\@ we can assume that $\opensch=\cscheme'$ and $\charfun''=1$. Then the image of $\charfun'$ in $\cinfty(\cscheme')$ is invertible and we can choose $\charfun_1\in\cinfty(\cscheme)$, $\charfun_2\in\iradical{\igen{\charfun}}$, s.t.\@ $\charfun'\charfun_1+\charfun_2=1$. Since $\charfun_2\in\iradical{\igen{\charfun}}$, $\prinop{\cscheme}{\charfun_2}\subseteq\prinop{\cscheme}{\charfun}$ (Lemma \ref{OpensRadicals}), and hence $\{\prinop{\cscheme}{\charfun'},\prinop{\cscheme}{\charfun}\}$ is a Zariski cover of $\cscheme$.
\item Let $\charfun_1,\charfun_2\in\cinfty(\cscheme)$ be the elements that define $\opensch_1$, $\opensch_2$. By definition 
	\begin{equation*}\comple{\opensch_1}{(\opensch_1\cap\opensch_2)}=
	\spec(\cinfty(\cscheme)\{\charfun_1^{-1}\}/\iradical{\igen{\ima{\charfun_2}}}),\end{equation*}
where $\ima{\charfun_2}\in\cinfty(\cscheme)\{\charfun_1^{-1}\}$ is the image of $\charfun_2$. On the other hand
	\begin{equation*}\opensch_1\cap(\comple{\cscheme}{\opensch_2})=
	\spec(\cinfty(\cscheme)\{\charfun_1^{-1}\}/\igen{\ima{\iradical{\igen{\charfun_2}}}})\cong\end{equation*}
	\begin{equation*}\cong\spec((\cinfty(\cscheme)/{{\iradical{\igen{\charfun_2}}}})\{\ima{\charfun_1}^{-1}\})\end{equation*}
Clearly $\igen{\ima{\iradical{\igen{\charfun_2}}}}\leq\iradical{\igen{\ima{\charfun_2}}}$, and since localization of a reduced $\cinfty$-ring is still reduced, we are done.\end{enumerate}\end{proof}}%

%%%%%%   Difference of a difference
\begin{remark}\label{DifDif} Proposition \ref{BasicFacts} immediately implies that for any open subscheme $\opensch\subseteq\cscheme$ we have $\comple{\cscheme}{(\comple{\cscheme}{\opensch})}=\opensch$. On the other hand, given an open subscheme $\opensch'\subseteq\comple{\cscheme}{\opensch}$, we have $\comple{(\comple{\cscheme}{\opensch})}{\opensch'}=\comple{\cscheme}{\fullex{\opensch'}}$. \hide{%
Indeed, let $\charfun$, $\charfun'$ be as above, i.e.\@ $\fullex{\opensch'}$ is defined by $\charfun^2+{\charfun'}^2$. Clearly $\charfun^2+{\charfun'}^2$ is in the kernel of $\cinfty(\cscheme)\rightarrow\cinfty(\comple{(\comple{\cscheme}{\opensch})}{\opensch'})$. Therefore it is enough to show that $\iradical{\igen{\charfun^2+{\charfun'}^2}}$ contains this kernel. Since $\prinop{\cscheme}{\charfun}\subseteq\prinop{\cscheme}{\charfun^2+{\charfun'}^2}$, $\charfun\in\iradical{\igen{\charfun^2+{\charfun'}^2}}$ (Lemma \ref{OpensRadicals}). Therefore also $\charfun'\in\iradical{\igen{\charfun^2+{\charfun'}^2}}$.}%
\end{remark}
%%%%%%

%%%%%%   Complement of an open subscheme is indeed a complement 
\begin{lemma}\label{AvoidingCell} Let $\cmor\colon\cscheme'\rightarrow\cscheme$ be a morphism in $\csite$, and let $\opensch\subseteq\cscheme$ be an open subscheme, s.t.\@ $\pullba{\cmor}{\opensch}=\emptys$. Then the composite morphism $\reduiof{\cscheme'}\rightarrow\cscheme'\rightarrow\cscheme$ factors through $\comple{\cscheme}{\opensch}$.\end{lemma}
%%%%%%
\hide{%
\begin{proof} Let $\charfun\in\cinfty(\cscheme)$ be any element, s.t.\@ $\prinop{\cscheme}{\charfun}=\opensch$. The condition $\pullba{\cmor}{\opensch}=\emptys$ means $\cinfty(\cscheme')\{\cmorf(\charfun)^{-1}\}=0$, where $\cmorf\colon\cinfty(\cscheme)\rightarrow\cinfty(\cscheme')$ corresponds to $\cmor$. This means $\cmorf(\charfun)\in\iradical{0}\leq\cinfty(\cscheme')$, and therefore $\charfun\in\iradical{\kernel{\cmorf}}$ (\cite{BK}, Prop.\@ 3). Then $\iradical{\igen{\charfun}}\leq\iradical{\kernel{\cmorf}}$, and hence $\cinfty(\cscheme)\rightarrow\cinfty(\cscheme')/\iradical{0}$ factors through $\cinfty(\cscheme)/\iradical{\igen{\charfun}}$.\end{proof}}%

%%%%%%   Definition of a principal locally closed subscheme
\begin{definition} A morphism $\cmor\colon\cscheme'\rightarrow\cscheme$ in $\csite$ is {\it a principal locally closed subscheme of $\cscheme$}, if there is an open subscheme $\opensch\subseteq\cscheme$, s.t.\@ $\cmor$ factors through $\cscheme'\rightarrow\opensch$, which is a principal closed subscheme.\end{definition}
%%%%%%

%%%%%%   Definition of the germ at the image of a morphism
\begin{definition}\label{DefinitionOfGerm} Let $\cmor\colon\cscheme'\rightarrow\cscheme$ be a morphism in $\csite$, {\it the germ of $\cscheme$ at the image of $\cmor$} is $\germof{\cscheme}{\cscheme'}:=\specof{\cinfty(\cscheme)/\gerim{\cmor}}$, where $\gerim{\cmor}$ is {\it the ideal of $0$-germs} at the image of $\cmor$, i.e.\@
	\begin{equation*}\gerim{\cmor}:=\{\celement\in\cinfty(\cscheme)\,|\,\exists\celement'\in\cinfty(\cscheme)\text{ s.t.\@ }\celement\celement'=0
	\text{ and }\exists{\celement'}^{-1}\in\cinfty(\cscheme')\}.\end{equation*}
\end{definition}
%%%%%%
In some cases, e.g.\@ that of a locally closed subscheme, we write $\gima{\cscheme'}$ to mean the germ of $\cscheme$ at the image of the inclusion $\cscheme'\hookrightarrow\cscheme$, and we write $\gerim{\cscheme,\cscheme'}$ for the corresponding ideal of $0$-germs.

%%%%%%   Comparing two definitions of germs
\begin{proposition}  Let $\cscheme'\hookrightarrow\cscheme$ be a principal locally closed subscheme, and let $\opensch(\cscheme')$ be the set of open subschemes of $\cscheme$ containing $\cscheme'$. Then
	\begin{equation*}\gima{\cscheme'}\cong\underset{\opensch\in\opensch(\cscheme')}\lim\opensch.\end{equation*}
\end{proposition}
%%%%%%
\begin{proof} Let $\cscheme''\rightarrow\{\opensch\}$ be a cone on the diagram of open subschemes of $\cscheme$, containing $\cscheme'$. Since this diagram is filtered, there is a well defined morphism $\cmor\colon\cscheme''\rightarrow\cscheme$. Since this morphism factors through each $\opensch$, for every $\charfun\in\cinfty(\cscheme)$, that becomes invertible in $\cinfty(\cscheme')$, $\cmorf(\charfun)$ is invertible. This immediately implies that $\cmorf$ maps $\gerim{\cscheme'}$ to $0$, i.e.\@ $\cmor$ factors through $\gima{\cscheme'}$.

We need to show that $\gima{\cscheme'}\rightarrow\cscheme$ factors through each $\opensch$ containing $\cscheme'$. This is equivalent to showing that each $\charfun\in\cinfty(\cscheme)$, that becomes invertible in $\cinfty(\cscheme')$, is also invertible in $\cinfty(\gima{\cscheme'})$. Since $\cscheme'$ is principal locally closed, taking intersections with an open neighbourhood of $\cscheme'$, we can assume that $\cscheme'\subseteq\cscheme$ is a principal closed subscheme, i.e.\@ we have $\charfun'\in\cinfty(\cscheme)$, s.t.\@ $\comple{\cscheme}{\cscheme'}=\prinop{\cscheme}{\charfun'}$. Let $\prinop{\cscheme}{\charfun}\supseteq\cscheme'$, then $\prinop{\cscheme}{\charfun}\cup\prinop{\cscheme}{\charfun'}=\cscheme$. Using Prop.\@ \ref{ZariskiByFunction} we can find $\charfun''\in\cinfty(\cscheme)$, s.t.\@ $\prinop{\cscheme}{\charfun}=\locin{\cscheme}{\charfun''<r_1}$, $\prinop{\cscheme}{\charfun'}=\locin{\cscheme}{\charfun''>r_2}$ for some $r_1,r_2\in\mathbb R$, and then $\charfun$ becomes invertible on $\locin{\cscheme}{\charfun''\leq\frac{r_1+r_2}{2}}\supseteq\cscheme'$. The kernel of $\cinfty(\cscheme)\rightarrow\cinfty(\locin{\cscheme}{\charfun''\leq\frac{r_1+r_2}{2}})$ is contained in $\gerim{\cscheme'}$.\end{proof}%

%%%%%%
\begin{lemma} Let $\cscheme\in\csite$, and let $\cscheme',\cscheme''\subseteq\cscheme$ two principal closed subschemes. Then $\cscheme'\cup\cscheme''\subseteq\cscheme$ is a principal closed subscheme.\end{lemma}
%%%%%%
\begin{proof} Let $\cmorf'\colon\cinfty(\cscheme)\rightarrow\cinfty(\cscheme')$, $\cmorf''\colon\cinfty(\cscheme)\rightarrow\cinfty(\cscheme'')$ be the corresponding morphisms of $\cinfty$-rings. Then by definition 
	\begin{equation*}\cscheme'\cup\cscheme'':=\specof{\cinfty(\cscheme)/(\kernel{\cmorf'}\cap\kernel{\cmorf''})}.\end{equation*}
According to \cite{BK} Lemma 5 $\iradical{\kernel{\cmorf'}\cap\kernel{\cmorf''}}=\iradical{\kernel{\cmorf'}}\cap\iradical{\kernel{\cmorf''}}$. By assumption $\exists\charfun_1,\charfun_2\in\cinfty(\cscheme)$, s.t.\@ $\iradical{\kernel{\cmorf'}}=\iradical{\igen{\charfun_1}}$, $\iradical{\kernel{\cmorf''}}=\iradical{\igen{\charfun_2}}$. Clearly inverting an element of $\igen{\charfun_1}\cap\igen{\charfun_2}$ implies inverting $\charfun_1\charfun_2$, i.e.\@ $\igen{\charfun_1}\cap\igen{\charfun_2}\subseteq\iradical{\igen{\charfun_1\charfun_2}}$. Hence $\iradical{\igen{\charfun_1}}\cap\iradical{\igen{\charfun_2}}=\iradical{\igen{\charfun_1\charfun_2}}$.\end{proof}%

\end{document}